\documentclass[11pt]{amsart}

\usepackage{amsmath, amssymb, comment}

\usepackage{pgf,tikz,hyperref}
\usetikzlibrary{arrows}

\renewenvironment{abstract}{\begin{quote}\textbf{\abstractname.}}{\end{quote}}

\newtheorem{thm}{Theorem}[subsection]
\newtheorem{theorem}[thm]{Theorem}
\newtheorem{lemma}[thm]{Lemma}
\newtheorem{fact}[thm]{Fact}
\newtheorem{proposition}[thm]{Proposition}
\newtheorem{corollary}[thm]{Corollary}
\newtheorem{conjecture}[thm]{Conjecture}

\newtheorem{openproblem}[thm]{Open Problem}

\newcommand\mkthm[2]{\newenvironment{#1}{\begin{#2}\rm}{\end{#2}}}
 \mkthm{definition}{tdefinition}
         \mkthm{remark}{tremark}
       \mkthm{example}{texample}
     \mkthm{question}{tquestion}
     \mkthm{exercise}{texercise}

\renewcommand\O{\mathcal O}

\newcommand\Exercise{Example}
\newcommand\KK{K}
\renewcommand\P{\mathbb P}
\newcommand\C{\mathbb C}
\newcommand\Q{\mathbb Q}
\newcommand\R{\mathbb R}

\newcommand\be[1][@{\;}r@{\;}c@{\;}l@{\;}l@{\;}]{$$\everymath{\displaystyle}\renewcommand\arraystretch{1.2}\begin{array}{#1}}
\newcommand\ee{\end{array}$$}

\newcommand\compact{\itemsep=0cm \parskip=0cm}

\newcommand\smallmatr[2][*{20}{c}]{{\fontsize{8}{9}\arraycolsep=4pt\selectfont\left(\!\!\begin{array}{#1}#2\end{array}\!\!\right)}}
\def\nmat#1 #2 #3 #4 #5 #6 #7 #8 #9 {\smallmatr{
   #1 & #2 & #3 \\
   #4 & #5 & #6 \\
   #7 & #8 & #9
   }}

\newcommand\SolnEater[1]{#1}

\newcommand\newop[2]{\newcommand#1{\mathop{\rm #2}\nolimits}}
\newcommand\renewop[2]{\renewcommand#1{\mathop{\rm #2}\nolimits}}
\newop\Cl{Cl}
\newop\chr{char}
\newop\mult{mult}
\newop\reg{reg}
\newop\Bl{Bl}
\newop\Pic{Pic}
\newop\Fix{Fix}
\newop\PGL{PGL}
\renewop\Re{Re}
\renewop\Im{Im}
\newop\br{branch}

\begin{document}

\title{Asymptotics of linear systems, with connections to line arrangements}
\author{Brian Harbourne}
\address{Department of Mathematics, University of Nebraska-Lincoln, Lincoln, NE 68588, USA}

\date{May 28, 2017}

\maketitle

\tableofcontents

\begin{abstract}
The main focus of these notes is recent work on linear systems 
in which line arrangements play a role, including problems such as 
semi-effectivity, containment problems of symbolic powers of 
homogeneous ideals in their powers, bounded negativity,
and a new perspective on the SHGH Conjecture.
Along the way we will be concerned with asymptotic invariants 
such as Waldschmidt constants, resurgences and $H$-constants.
\end{abstract}

\begin{quote}
\thanks{\noindent {\bf Acknowledgements}: 
My participation at the Spring, 2016 miniPAGES was partially supported by the grant 346300 for IMPAN 
from the Simons Foundation and the matching 2015-2019 Polish MNiSW fund.
Some of the revisions to these notes were done at the Banach Center in Warsaw in the weeks 
after this material was presented at the workshop at the Pedagogical University in Krakow.
I also thank J.\ Szpond for her helpful comments, and the referee for generously working though these notes
and providing copious, helpful feedback.}
\end{quote}


\section{Line arrangements, semi-effectivity and Waldschmidt constants}
\subsection{Line arrangements}
In this section we recall some examples and facts about line arrangements (in the projective plane).
We will always take $\KK$ to be an algebraically closed field.
A line arrangement over $\KK$ is a finite list $L_1,\ldots,L_d\subseteq\P^2_\KK$, $d>1$,
of distinct lines in the projective plane and their crossing points 
(i.e., the points of intersections of the lines).
Line arrangements have been coming up in a range of topics of recent research interest discussed in these notes.
A useful notation is $t_k$, for $k\geq2$, for the number of points lying on exactly $k$ lines.

\begin{exercise}\label{t_k}
Consider a line arrangement $L_1,\ldots,L_d\subseteq\P^2_\KK$.
Let $s$ be the number of crossing points.
\begin{enumerate}
\item[(a)] Then the number of crossing points is $s=t_2+\ldots+t_d$.
\item[(b)] And $\binom{d}{2}=\sum_k t_k\binom{k}{2}$.
\item[(c)] We have $0\leq t_d\leq 1$, and that $t_k=0$ for all $k<d$ if and only if $t_d=1$.
\item[(d)] Then $d^2-\sum_k t_kk^2=d-\sum_k t_kk$.
\item[(e)] If the lines do not all go through a single point, then $s\geq d$. 
(Note: This is a weak form of the de Bruijn-Erd\H{o}s theorem in incidence geometry.
See \cite{1deBE48}
for a combinatorial proof. Here is a sketch for an algebraic geometric proof.
Blow up the crossing points. Look at the classes of the proper transforms of the lines. One can show that
they are linearly independent in the divisor class group of the blow up and span a negative definite
subspace.)
\end{enumerate}
\end{exercise}

\SolnEater{\vskip\baselineskip 
\noindent{\it Details}: (a) Let $T_k$ be the set of points where exactly $k$ lines meet, so $|T_k|=t_k$.
In order for lines to meet, there must be at least 2, so $k\geq2$.
And clearly $T_k=\varnothing$ for $k>d$. Every crossing point is in some $T_k$,
so the set of crossing points is $\cup_kT_k$, hence there are $|\cup_kT_k|$ crossing points.
Since the sets $T_k$ are disjoint, we have  $|\cup_kT_k|=t_2+\cdots+t_d$.

(b) There are $\binom{d}{2}$ pairs of lines. Every pair of lines meet at exactly one point, so we can count the pairs
by counting how many pairs occur at each crossing point. Thus $\binom{d}{2}=\sum_k t_k\binom{k}{2}$.

(c) This follows from $\binom{d}{2}=\sum_k t_k\binom{k}{2}$.

(d) This formula is equivalent to $\binom{d}{2}=\sum_k t_k\binom{k}{2}$.

(e) Let $X\to\P^2$ be the surface obtained by blowing up the points.
Let $C_i\subseteq X$ be the proper transform of the line $L_i$.
Since the lines don't all go through the same point, each line has at least
two crossing points, hence $C_i^2\leq -1$. Also, since every crossing point has been blown up,
we have $C_i\cdot C_j=0$ for $i\neq j$. It now follows that the span of the classes
of the $C_i$ in the divisor class group of $X$ (which is free abelian) is negative definite
and thus has rank at most the number of points blown up, namely $s$.
If for some integers $a_i$ we had
$\sum_ia_iC_i\sim0$, then $\sum_ia_i^2C_i^2=(\sum_ia_iC_i)^2=0$, and since $C_i^2<0$ for all $i$,
we see that $a_i=0$ for all $i$. Thus the classes of the $C_i$ are linearly
independent, hence $d\leq s$.
\newline\qedsymbol\vskip\baselineskip}

An interesting property that a line arrangement can have is the $t_2=0$ property; i.e., that whenever two of the lines $L_i$ cross,
there is at least one other line that also goes through that crossing point. An easy such example is the 
case of $d\geq 3$ concurrent lines (i.e., $d\geq3$ lines through a point $p$). 
Over the reals, these are the only  line arrangements with $t_2=0$, due to the following result \cite{1M41}:

\begin{theorem}\label{MelchiorThm}
Given a real line arrangement of $d$ lines with $t_d=0$ (i.e., the lines are not concurrent), we have
$$t_2\geq 3+\sum_{k>2}t_k(k-3).$$
\end{theorem}

If $\chr(\KK)=p>0$, there are many examples of line arrangements with $t_2=0$.

\begin{exercise}\label{finfieldarrs}
Assume $\chr(\KK)=p>0$. Consider the arrangement of all lines defined over the finite field ${\mathbb F}_q\subseteq \KK$ of order $q$.
Then one can see that there are $q^2+q+1$ lines and $q^2+q+1$ crossing points, that $t_k=0$ except for $t_{q+1}=q^2+q+1$, 
that each line contains $q+1$ of the points and each point is on $q+1$ of the lines.
\end{exercise}

\SolnEater{\vskip\baselineskip 
\noindent{\it Details}: There are $|{\mathbb F}_q^2|=q^2$ points in the affine plane.
These are of the form $[a:b:1]$ with $a,b\in {\mathbb F}_q$. The remaining points are of the form
$[a:b:0]$, hence either $[a:1:0]$, of which there are $q$, or $[1:0:0]$. So there are 
$q^2+q+1$ points of $\P^2$ defined over ${\mathbb F}_q$. 
Lines are dual to points, so there are the same number of lines.
Given any point $[a_0:a_1:a_2]$, at least one of the three sets of forms 
$\{a_1x-a_0y, a_2x-a_0z\}$, $\{a_0y-a_1x, a_2y-a_1z\}$, $\{a_0z-a_2x, a_1z-a_2y\}$
defines a pair of lines crossing at the point, so every point is a crossing point.
For every point $P$ there is a coordinate axis $x=0$, $y=0$ or $z=0$ not containing $P$.
Let $L$ be this coordinate axis. Every ${\mathbb F}_q$-line through $P$ meets $L$ at one of the $q+1$
${\mathbb F}_q$-points of $L$ and this point uniquely determines the line, so there are $q+1$ ${\mathbb F}_q$-lines through $P$.  
Thus every point is on $q+1$ lines, so $t_{q+1}=q^2+q+1$ and $t_k=0$ otherwise. 
Dually, for every line $L$ there is a coordinate vertex $P$, either $[0:0:1]$, $[0:1:0]$ or $[1:0:0]$, not on $L$.
This point is on $q+1$ lines, and every ${\mathbb F}_q$-point of $L$ is on exactly one of these lines,
so $L$ has $q+1$ ${\mathbb F}_q$-points.
\newline\qedsymbol\vskip\baselineskip}

\begin{remark}\label{t_2=0Rem}
Over $\KK=\C$ only four kinds of line arrangements seem to be known with $t_2=0$. Here is the list.
(See \cite{3refIMRN} and especially \cite{2refBetal} for more information about the Klein and Wiman configurations below.)
\begin{enumerate}
\item[(1)] Any set of $s\geq3$ concurrent lines.

\item[(2)] The Fermat arrangement of $3n$ lines for $n\geq3$: The lines of this arrangement are defined by the factors of $(x^n-y^n)(x^n-z^n)(y^n-z^n)$, shown for $n=3$ in
Figure \ref{FermatFigure}. Each line contains $n+1$ of the points, and we have $t_k=0$ except for $t_3=n^2$ and $t_n=3$ when $n>3$ or $t_3=12$ when $n=3$.

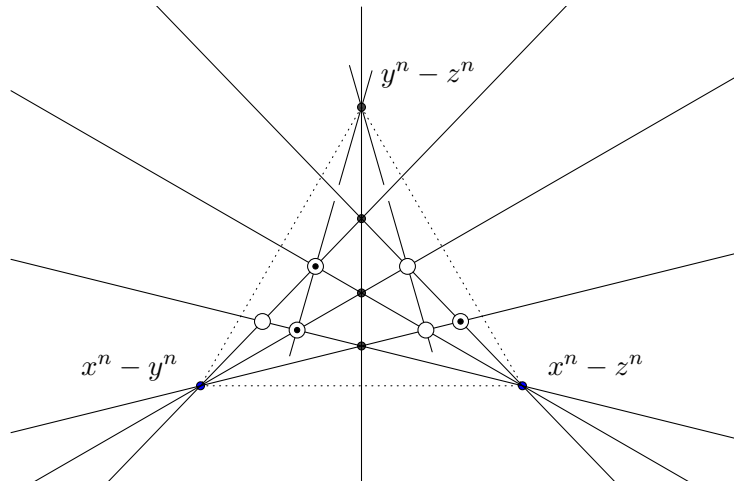
\begin{figure}[htbp]
\begin{center}
\definecolor{uuuuuu}{rgb}{0.26666666666666666,0.26666666666666666,0.26666666666666666}
\definecolor{zzttqq}{rgb}{0.6,0.2,0.0}
\definecolor{qqqqff}{rgb}{0.0,0.0,1.0}
\begin{tikzpicture}[line cap=round,line join=round,>=triangle 45,x=1.0cm,y=1.0cm]
\clip(-4.3,0) rectangle (5.3,6.3);
\draw [color=black,domain=-.6:.5] plot(\x,{(-2.274215998297232-2.118050701827084*\x)/-0.6114285714285721});
\draw [color=black,domain=.2:1.3] plot(\x,{(--3.79921250361273-2.1180507018270847*\x)/0.6114285714285709});


\begin{scriptsize}
\draw [fill=qqqqff] (-1.78,1.26) circle (1.5pt);
\draw [fill=qqqqff] (2.5,1.26) circle (1.5pt);
\draw [fill=uuuuuu] (0.3600000000000009,4.966588728197398) circle (1.5pt);
\draw [fill=uuuuuu] (0.3600000000000006,3.483953236918439) circle (1.5pt);
\draw [fill=uuuuuu] (0.3600000000000003,2.4955295760657994) circle (1.5pt);
\draw [fill=uuuuuu] (0.35999999999999976,1.7895126754567712) circle (1.5pt);
\draw [color=white, fill=white] (0.030769230769231208,3.826099888752045) circle (3pt);
\draw [color=white, fill=white] (0.6892307692307698,3.826099888752045) circle (3pt);
\draw [color=white, fill=white] (-0.6276923076923083,1.545122209861338) circle (3pt);
\draw [color=white, fill=white] (1.3476923076923073,1.5451222098613382) circle (3pt);
\end{scriptsize}
\draw [dotted,domain=-4.3:7.3] (0.3600000000000009,4.966588728197398)-- (-1.78,1.26);
\draw [dotted,domain=-4.3:7.3] (-1.78,1.26)-- (2.5,1.26);
\draw [dotted,domain=-4.3:7.3] (0.3600000000000009,4.966588728197398)-- (2.5,1.26);
\draw [domain=-4.3:7.3] plot(\x,{(-7.035317955123373--0.9266471820493496*\x)/-3.7449999999999997});
\draw [domain=-4.3:7.3] plot(\x,{(--8.677835910246747-1.853294364098699*\x)/3.2099999999999995});
\draw [domain=-4.3:7.3] plot(\x,{(-10.320353865370121--2.7799415461480486*\x)/-2.6749999999999994});
\draw [domain=-4.3:7.3] plot(\x,{(--8.318795952143528--2.7799415461480486*\x)/2.6750000000000007});
\draw [domain=-4.3:7.3] plot(\x,{(--7.343463968095685--1.853294364098699*\x)/3.2100000000000004});
\draw [domain=-4.3:7.3] plot(\x,{(--6.368131984047842--0.9266471820493496*\x)/3.745});
\draw (0.36,-2.58) -- (0.36,6.3);
\draw (-3.5,1.8) node[anchor=north west] {$x^n-y^n$};
\draw (2.7,1.8) node[anchor=north west] {$x^n-z^n$};
\draw (0.48,5.7) node[anchor=north west] {$y^n-z^n$};
\begin{scriptsize}
\draw [fill=white] (-0.956923076923077,2.115366629584015) circle (3pt);
\draw [fill=white] (0.9714285714285718,2.8485380263703135) circle (3pt);
\draw [fill=white] (1.2159999999999997,2.0013177456394797) circle (3pt);
\draw [fill=white] (1.676923076923077,2.115366629584015) circle (3pt);
\draw [fill=black] (1.676923076923077,2.115366629584015) circle (1pt);
\draw [fill=white] (-0.2514285714285713,2.848538026370314) circle (3pt);
\draw [fill=black] (-0.2514285714285713,2.848538026370314) circle (1pt);
\draw [fill=white] (-0.4959999999999998,2.0013177456394797) circle (3pt);
\draw [fill=black] (-0.4959999999999998,2.0013177456394797) circle (1pt);
\end{scriptsize}
\end{tikzpicture}
\caption{The Fermat arrangement of $3n$ complex lines and their $n^2+3$ points of intersection for $n=3$.
(The 12 points for $n=3$ are indicated by the three open circles, the three dotted circles and the six black circular dots. 
The coordinate axes are represented by dotted lines.
At each coordinate vertex there occur $n$ of the $3n$ lines, defined by the forms shown; the $n^2+3$ points consist of
a complete intersection of $n^2$ points plus the 3 coordinate vertices. This arrangement
does not exist over the reals: one must regard the open circles as representing collinear points, and
likewise the dotted circles as representing collinear points.)}
\label{FermatFigure}
\end{center}
\end{figure}

\item[(3)] The Klein arrangement of 21 lines \cite{1Kle79}: here $t_k=0$ except for $t_4=21$ and $t_3=28$. For this arrangement,
each line contains 4 points where 3 lines cross and 4 points where 3 lines cross.

\item[(4)] The Wiman arrangement of 45 lines \cite{1Wim96}: here $t_k=0$ except for $t_5=36$, $t_4=45$ and $t_3=120$. For this arrangement,
each line contains 4 points where 5 lines cross, 4 points where 4 lines cross and 8 points where 3 lines cross.
\end{enumerate}
\end{remark}

\begin{exercise}\label{FC}
We have $t_2\neq0$ for the Fermat arrangement if and only if $n=2$. Note that 
the Fermat arrangement is defined over the reals for $n=1, 2$, so we can
draw it in those cases (see Figure \ref{Fermat1and2}).
\end{exercise}

\SolnEater{\vskip\baselineskip 
\noindent{\it Details}: There are exactly $n$ lines through each coordinate vertex, defined by the 
linear factors of either $x^n-y^n$, $x^n-z^n$ or $y^n-z^n$. Other than the coordinate vertices,
the crossing points are contained in the zero locus of the ideal $(x^n-y^n, x^n-z^n)$. No coordinate
vertex is in this zero locus, and since each curve, $x^n-y^n$ and $x^n-z^n$, is a union of a different set of lines, 
these curves have degree $n$ and meet transversely, hence the zero locus consists of $n^2$ points,
and at each point there is exactly one line from $x^n-y^n$ and one line from $x^n-z^n$.
So there are exactly $n^2+2$ crossing points for these $2n$ lines.
But $y^n-z^n$ is in the ideal $(x^n-y^n, x^n-z^n)$, so the only additional crossing point
coming from the $n$ lines of $y^n-z^n$ is the third coordinate vertex. This gives $n^2+3$ crossing points.
At each coordinate vertex there are $n$ lines, and at each of the $n^2$ other points, there is one line
each, defined by a factor of  $x^n-y^n$, $x^n-z^n$ and $y^n-z^n$. Thus for $n\geq3$, each of the $n^2+3$ 
points is on at least 3 of the lines, so $t_2=0$. If $n=2$, then $t_3=2^2$ but the 3 coordinate vertices give $t_2=3$.
For $n=1$, the coordinate vertices are not crossing points, and we have $t_2=0$ and $t_3=1$.
\newline\qedsymbol\vskip\baselineskip}

\begin{figure}[htbp]
\begin{center}
\definecolor{uuuuuu}{rgb}{0.26666666666666666,0.26666666666666666,0.26666666666666666}
\definecolor{ffffff}{rgb}{1.0,1.0,1.0}
\begin{tikzpicture}[line cap=round,line join=round,>=triangle 45,x=1.0cm,y=1.0cm]
\clip(-4.788647561425146,0.5) rectangle (6.496656545575581,6.52991951524885);
\fill[color=ffffff,fill=ffffff,fill opacity=0.1] (-1.96,3.82) -- (-3.0,2.04) -- (-0.9384747812636993,2.0293335800641827) -- cycle;


\draw [dotted] (-1.77,4)-- (-3.0,2.04);
\draw [dotted] (-3.0,2.04)-- (-0.455,2.02);
\draw [dotted] (-0.45,2.02)-- (-1.77,4);


\draw (-3.4266541975250924,1.7966056873578928)-- (-0.8754380922432023,3.27);
\draw (-1.7524179059551636,1.419960769920086)-- (-1.756999135449185,4.399984473398056);
\draw (-2.571826823529412,3.263988705882353)-- (-0.12,1.8110889411764694);
\draw (-2.278640268661191,1.4126178253502502) node[anchor=north west] {$n=1$};
\draw (1.324845578235138,3.8408434111599417)-- (3.7105287404469367,3.8687577064140037);
\draw (2.3986766603997998,5.913369322346303)-- (2.2997803777715737,2.404829008031016);
\draw [dotted] (0.98,1.88)-- (2.3405927546239114,3.8527284204664203);
\draw (2.329918464729798,3.4740369729731717)-- (3.6325467952353945,1.8457515598411773);
\draw [dotted] (3.6325467952353945,1.8457515598411773)-- (0.98,1.88);
\draw [dotted] (2.3405927546239114,3.8527284204664203)-- (3.6325467952353945,1.8457515598411773);
\draw (2.2743962158873785,1.393160404476111) node[anchor=north west] {$n=2$};
\draw (0.8373114010209874,5.339795802609183)-- (4.206026879797202,1.1289014541389188);
\draw (4.053905097583056,0.9802272419885505)-- (1.6856748507684447,5.844877724657637);
\draw (0.5029349110698615,0.9852396981066656)-- (3.1333976176971077,5.918808848553596);
\draw (0.3415011462185662,1.126035366786157)-- (3.887578748345626,5.31338368041487);
\begin{scriptsize}
\draw [fill=uuuuuu] (-1.756158260421233,2.7697778600213944) circle (1.5pt);
\draw [fill=black] (0.98,1.88) circle (1.5pt);
\draw [fill=black] (3.6325467952353945,1.8457515598411773) circle (1.5pt);
\draw [fill=uuuuuu] (2.357828153297249,4.46418811815827) circle (1.5pt);
\draw [fill=uuuuuu] (2.329918464729798,3.4740369729731717) circle (1.5pt);
\draw [fill=uuuuuu] (2.0298738237137606,3.8490927742432177) circle (1.5pt);
\draw [fill=uuuuuu] (2.653717987465469,3.856392222240052) circle (1.5pt);
\draw [fill=uuuuuu] (2.3405927546239114,3.8527284204664203) circle (1.5pt);
\end{scriptsize}
\end{tikzpicture}
\caption{The Fermat arrangement of $3n$ complex lines and their $n^2+3$ points of intersection for $n=31$ and $n=2$.}
\label{Fermat1and2}
\end{center}
\end{figure}
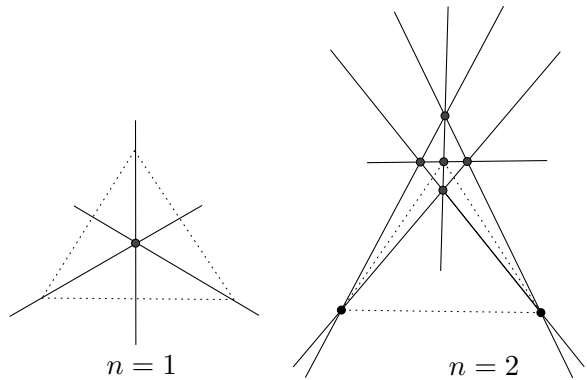

\begin{openproblem}\label{t_2=0OP}
Show either that there are other complex line arrangements with $t_2=0$, or that the four types listed above are the only ones.
\end{openproblem}

If one allows curves of higher degree, there are additional examples of finite sets of curves where more than two curves pass through each point of intersection
of any two of the curves; see Figure \ref{ChudConics} for an example taken from
\cite{refCh} using conics, and see \cite{3refBHRT,3Roulleau14} for examples of cubics.

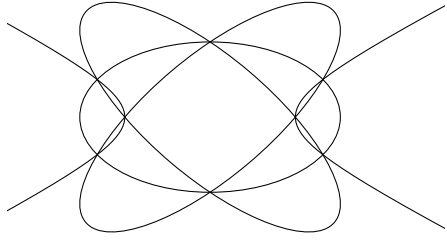
\begin{figure}[htbp]
\begin{center}
\begin{tikzpicture}[line cap=round,line join=round,>=triangle 45,x=0.5cm,y=0.5cm]
\clip(-5.38,-2.66) rectangle (6.22,6.22);
\draw [rotate around={0.:(0.,2.)}] (0.,2.) ellipse (1.7320508075688772cm and 1.cm);
\draw [rotate around={139.7311611040128:(0.,2.)}] (0.,2.) ellipse (2.1666227582959783cm and 0.7994242656857872cm);
\draw [rotate around={-139.7311611040128:(0.,2.)}] (0.,2.) ellipse (2.1666227582959783cm and 0.7994242656857872cm);
\draw [samples=50,domain=-0.99:0.99,rotate around={0.:(0.,2.)},xshift=0.cm,yshift=1.cm] plot ({2.267786838055363*(1+(\x)^2)/(1-(\x)^2)},{1.1547005383792508*2*(\x)/(1-(\x)^2)});
\draw [samples=50,domain=-0.99:0.99,rotate around={0.:(0.,2.)},xshift=0.cm,yshift=1.cm] plot ({2.267786838055363*(-1-(\x)^2)/(1-(\x)^2)},{1.1547005383792508*(-2)*(\x)/(1-(\x)^2)});
\end{tikzpicture}
\caption{Four conics defined over the reals where three conics pass
through each point of intersection.}
\label{ChudConics}
\end{center}
\end{figure}

\subsection{Semi-effectivity}
\begin{definition}
Let $C$ be a plane curve defined as a scheme by a nonzero homogeneous polynomial 
$F\in R=\KK[x,y,z]=\KK[\P^2]$. Then the multiplicity of $C$ or $F$ at $p\in\P^2$,
denoted $\mult_p(C)$ or $\mult_p(F)$, is the largest $m$ such that
$F\in I(p)^m$.
\end{definition}

One way to determine $\mult_p(C)$ is by making a linear change of coordinates so that $p=[0:0:1]$.
If $F=F(x,y,z)$ is the homogeneous form defining $C$ after the change of variables, 
then $\mult_p(C)$ is the least degree among the terms of $F(x,y,1)$.

\begin{example}
The multiplicity of $F=x^3y^4+x^5z^2$ at $p=[0:0:1]$ is 5.
\end{example}

\begin{definition}
We will denote the $\KK$-vector space spanned by the homogeneous forms of degree $t$ by $[R]_t$.
Given a homogeneous ideal $I\subseteq R$, $[I]_t$ denotes $[R]_t\cap I$, and
$[R/I]_t$ denotes $[R]_t/[I]_t$.
Then, given two curves $C$ and $D$ defined by nonconstant forms $F$ and $G$ 
with no common factors, we define
the intersection multiplicity of $C$ and $D$ at $p$ by
$I_p(C,D)=\dim_\KK [R/J]_t$ for $t\gg0$, where $J=I(p)^m+(F,G)$ for any 
$m\geq \deg(F)\deg(G)$. 
\end{definition}

\begin{theorem}[Bezout's Theorem]
Let $C$ and $D$ be curves defined by nonconstant forms $F$ and $G$ 
with no common factors. Then
$\sum_pI_p(C,D)=\deg(C)\deg(D)$. Moreover,
$I_p(C,D)\geq \mult_p(C)\mult_p(D)$
for each point $p\in\P^2$.
\end{theorem}

\begin{corollary}
Let $C$ and $D$ be plane curves defined by nonconstant forms $F$ and $G$. 
Let $S\subseteq \P^2$ be a finite set of points.
If $\sum_{p\in S}\mult_p(C)\mult_p(D)>\deg(C)\deg(D)$, then $C$ and $D$ have a common component
(i.e., $F$ and $G$ have a common factor of positive degree).
\end{corollary}

Consider distinct points $p_1,\ldots,p_s\in\P^N$. Let $\pi:X\to\P^N$ be the blow up of the points. Let $L$ be the pullback of a general hyperplane
and let $E_i$ be the inverse image of $p_i$. Then the divisor class group $\Cl(X)$ is free abelian with basis given by the divisor classes
$[L], [E_1],\ldots,[E_s]$. When $N=2$, this is an orthogonal basis for the intersection form on $\Cl(X)$, with $-L^2=E_1^2=\cdots=E_s^2=-1$ and we have
$-K_X =3L-E_1-\cdots-E_s$.

Given $m_i\geq 0$, consider the 
homogeneous ideal $I=\cap_i I(p_i)^{m_i}\subseteq \KK[\P^N]=\KK[x_0,\ldots,x_N]$. It defines a 0-dimensional subscheme
$Z=m_1p_1+\cdots +m_sp_s\subseteq \P^N$ called a {\it fat point} subscheme, where by definition we have $I(Z)=I$. 
Let $E_Z=m_1E_1+\cdots +m_sE_s$. We will sometimes refer to the degree $\deg(Z)$ of $Z$. By this we will mean the scheme theoretic
degree, hence not the sum of the coefficients $m_i$, but rather $\sum_i\binom{m_i+N-1}{N}$. This will turn up in a number of contexts;
see for example, Examples \ref{Roe} and \ref{ContFacts}, and Theorem \ref{virtdim}.

We refer to \cite{3refHart} for definitions of sheaf cohomology, line bundles and their associated divisors, and for notation such as 
$|D|$ when $D$ is a divisor on a variety. However, reliance on the next example will to some extent make it possible to avoid
dealing with some of this background.

\begin{exercise}\label{iso}
(See \cite[Proposition IV.1.1]{2refCracow}.) There is a canonical $\KK$-vector space isomorphism
$$H^0(X,\O_X(tL-E_Z))\cong [I(Z)]_t.$$
\end{exercise}

\begin{definition}
Given a divisor $D$ on a smooth projective surface $X$, we say $D$ is {\it semi-effective} if
for some $m>0$ we have $h^0(X,\O_X(mD))>0$ (i.e., for some $m>0$, $|mD|\neq\varnothing$, so $mD$ is linearly equivalent to
an effective divisor).
\end{definition}

Here is a question raised by Eisenbud and Velasco (2009) regarding semi-effectivity.

\begin{openproblem}[Eisenbud-Velasco]\label{EVprob}
Given an arbitrary $t\geq0$ and $E_Z=m_1E_1+\cdots +m_sE_s$ with $m_i\geq0$, 
is there an algorithm to determine whether $tL-E_Z$ is semi-effective
(or equivalently $\dim [I(mZ)]_{mt}>0$ for $Z=m_1p_1+\cdots +m_sp_s$)?
\end{openproblem}

\subsection{Waldschmidt constants}
Eisenbud and Velasco's question can be partially addressed by Waldschmidt constants \cite{2refW}.
Let $I\subseteq \KK[\P^N]$ be a nonzero homogeneous ideal. We define 
$\alpha(I)$ to be the least $t$ such that $[I]_t\neq0$. 

\begin{exercise}\label{alphaLikeLog}
If $I, J\subseteq \KK[\P^N]$ are nonzero homogeneous ideals, then
$\alpha(IJ)=\alpha(I)+\alpha(J)$. In particular, we have
$\alpha(I^r)=r\,\alpha(I)$.
\end{exercise}

\SolnEater{\vskip\baselineskip 
\noindent{\it Details}: Let $f\in I$ and $g\in J$ be homogeneous and nonzero
such that $\deg(f)=\alpha(I)$ and $\deg(g)=\alpha(J)$; then 
$fg\in IJ$ so $\alpha(I)+\alpha(J)=\deg(f)+\deg(g)=\deg(fg)\geq\alpha(IJ)$.
But $IJ$ is generated by elements of the form $FG$
where $F\in I$ and $G\in J$ are homogeneous and nonzero, hence
$\deg(FG)\geq \alpha(I)+\alpha(J)$. Thus
$\alpha(IJ)\geq\alpha(I)+\alpha(J)$.
\newline\qedsymbol\vskip\baselineskip}

Note that given $Z=m_1p_1+\cdots+m_sp_s\subseteq\P^N$,
its ideal is $I=I(Z)=I(p_1)^{m_1}\cap\cdots\cap I(p_s)^{m_s}$, and
the $m$th {\it symbolic power} of $I$, denoted $I^{(m)}$, is
$I^{(m)}=I(Z)^{(m)}=I(mZ)=I(p_1)^{mm_1}\cap\cdots\cap I(p_s)^{mm_s}$.
This terminology is often used in the literature. Moreover, one 
can define symbolic powers of any homogeneous ideal,
but doing so involves technicalities, so we will avoid that for now.

\begin{definition}\label{WCdef}
Let $Z=m_1p_1+\cdots+m_sp_s$ be a nonzero fat point subscheme of $\P^N$. 
The {\it Waldschmidt constant} $\widehat{\alpha}(I(Z))$ of $I(Z)$ is
$$\widehat{\alpha}(I(Z)) = \inf\Big\{\frac{\alpha(I(mZ))}{m} : m>0\Big\}.$$
\end{definition}

\begin{exercise}
Let $X$ be the surface obtained by blowing up distinct points $p_1,\ldots,p_r$. Let $I(Z)$ be the ideal of $Z=m_1p_1+\cdots+m_rp_r$
and let $F_{t,m}=tL-mE_Z$. Then 
$$\widehat{\alpha}(I(Z))=\inf\Big\{\frac{t}{m}: h^0(X, \O_X(F_{t,m}))>0\Big\}.$$
\end{exercise}

\SolnEater{\vskip\baselineskip 
\noindent{\it Details}: This is immediate since $h^0(X, \O_X(F_{t,m}))>0$ if and only if
$t\geq \alpha(I(mZ))$, so 
$$\inf\left\{\frac{t}{m}: h^0(X, \O_X(F_{t,m}))>0\right\}=\inf\left\{\frac{\alpha(I(mZ))}{m} : m>0\right\}.$$
\newline\qedsymbol\vskip\baselineskip}

It turns out to be useful to know that $\widehat{\alpha}(I(Z))$ is a limit.
Among other things, the following example shows how to see the infimum is actually a limit.

\begin{exercise}\label{Fekete}
Let $Z$ be a nonzero fat point subscheme of $\P^N$. 
\begin{enumerate}
\item[(a)] Then $1\leq \widehat{\alpha}(I(Z))\leq \sum_im_i$.
\item[(b)] Let $m,n$ be positive integers.
Then $$\alpha(I((m+n)Z))\leq \alpha(I(mZ))+\alpha(I(nZ)).$$
\item[(c)] Let $m,n$ be positive integers.
Then $$\frac{\alpha(I(mnZ))}{mn}\leq \frac{\alpha(I(mZ))}{m}.$$
\item[(d)] Fekete's Subadditivity Lemma \cite{2refFe2} implies for each $n$ that 
$$\widehat{\alpha}(I(Z)) = \lim_{m\to\infty}\frac{\alpha(I(mZ))}{m}\leq \frac{\alpha(I(nZ))}{n}.$$
\item[(e)] We have $\widehat{\alpha}(I(nZ)) =n\,\widehat{\alpha}(I(Z))$.
\item[(f)] 
Over the complexes, Waldschmidt and Skoda \cite{2refW, 2refSk} obtained the bound
$$\frac{\alpha(I(Z))}{N}\leq \widehat{\alpha}(I(Z))$$
using some rather hard analysis. A proof using multiplier ideals is given in
\cite{2refLa}. Here is another approach which comes from \cite[p.\ 2]{refHR} (also see \cite{2refHaHu}).
It is known that 
$$I((N+m-1)rZ)\subseteq I(mZ)^r$$ 
for all $m,r>0$ \cite{2refELS, 2refHH}. 
Assuming this, one can show for each $n>0$ that 
$$\frac{\alpha(I(mZ))}{N+m-1}\leq \widehat{\alpha}(I(Z))$$
and hence that $$\frac{\alpha(I(mZ))}{N+m-1}\leq \widehat{\alpha}(I(Z))\leq \frac{\alpha(I(mZ))}{m}.$$
\end{enumerate}
\end{exercise}

\SolnEater{\vskip\baselineskip 
\noindent{\it Details}: (a) A nonzero form cannot vanish at a point to order more than its degree,
so $\alpha(I(mZ))\geq m$, hence $1\leq \widehat{\alpha}(I(Z))$. 
By taking $mm_i$ hyperplanes through each point $p_i$, we get a form
of degree $m(\sum_im_i)$ in $I(mZ)$, hence $\alpha(ImZ)\leq m(\sum_im_i)$. 
Thus $\widehat{\alpha}(I(Z))\leq \sum_im_i$.

(b) Let $F\in I(mZ)$ of degree $\alpha(I(mZ))$ and $G\in I(nZ)$ of degree $\alpha(I(nZ))$.
Then $FG\in I((m+n)Z)$, so $\alpha(I((m+n)Z))\leq \alpha(I(mZ))+\alpha(I(nZ))$.

(c) This follows from (b).

(d) The equality is immediate from Fekete's Lemma. Then
$\widehat{\alpha}(I(Z))= \lim_{m\to\infty}\frac{\alpha(I(mnZ))}{mn}\leq \frac{\alpha(I(nZ))}{n}$ follows from 
this and from (c).

(e) Note $\widehat{\alpha}(I(nZ))=\lim_{m\to\infty}\frac{\alpha(I(nmZ))}{m}=
n\,\lim_{m\to\infty}\frac{\alpha(I(nmZ))}{nm}$\\
\ \hbox to1.45in{\hfil}$=n\lim_{m\to\infty}\frac{\alpha(I(mZ))}{m}=n\widehat{\alpha}(I(Z))$.

(f) From $I((N+m-1)rZ)\subseteq I(mZ)^r$ we get
$$ \alpha(I((N+m-1)rZ))\geq \alpha(I(mZ)^r)=r\,\alpha(I(mZ)),$$
hence 
$$ \frac{\alpha(I((N+m-1)rZ))}{r(N+m-1)}\geq \frac{r\,\alpha(I(mZ))}{r(N+m-1)}.$$
The result follows by taking limits as $r\to\infty$.
\newline\qedsymbol\vskip\baselineskip}

\begin{exercise}
Let $Z$ be a nonzero fat point subscheme of $\P^N$ and $I=I(Z)$.
If $\frac{\alpha(I^{(m)})}{m}\leq\frac{\alpha(I^{(n)})}{n}$, then 
$$\frac{\alpha(I^{(m)})}{m}\leq \frac{\alpha(I^{(m+n)})}{m+n}\leq \frac{\alpha(I^{(n)})}{n}.$$
\end{exercise}

\SolnEater{\vskip\baselineskip 
\noindent{\it Details}: 
If $a,b,c,d>0$ and $a/b\leq c/d$, then it is easy to see that
$a/b\leq (a+c)/(b+d)\leq c/d$.
Thus 
$$\frac{\alpha(I^{(m)})}{m}\leq \frac{\alpha(I^{(m)})+\alpha(I^{(n)})}{m+n}\leq \frac{\alpha(I^{(n)})}{n},$$
so by \Exercise\ \ref{Fekete}(b) we have
$$\frac{\alpha(I((m+n)Z))}{m+n}\leq \frac{\alpha(I^{(m)})+\alpha(I^{(n)})}{m+n}\leq \frac{\alpha(I^{(n)})}{n}.$$
\newline\qedsymbol\vskip\baselineskip}

\begin{exercise}\label{compare}
Let $Z=m_1p_1+\cdots+m_sp_s$ and
$Z'=m_1'p_1+\cdots+m_s'p_s$ be fat point subschemes of $\P^N$
for distinct points $p_i$ with $0\leq m_i\leq m_i'$ for all $i$.
Then $\widehat{\alpha}(I(Z))\leq \widehat{\alpha}(I(Z'))$.
It is also possible to have $Z\neq Z'$ but $\widehat{\alpha}(I(Z)) = \widehat{\alpha}(I(Z'))$.
\end{exercise}

\SolnEater{\vskip\baselineskip 
\noindent{\it Details}: Since $I(mZ')\subseteq I(mZ)$, we get $\alpha(I(mZ'))\geq \alpha(I(mZ))$ 
and hence $\widehat{\alpha}(I(Z))\leq \widehat{\alpha}(I(Z'))$. For the rest,
let $Z$ be two points on a line and $Z'$ those two points plus a third point on the same line.
Then $\widehat{\alpha}(I(Z)) = \widehat{\alpha}(I(Z'))=1$.
\newline\qedsymbol\vskip\baselineskip}

Given a fat point subscheme $Z=m_1p_1+\cdots+m_sp_s\subset\P^N$,
then $tL-E_z$ is not semi-effective if $t<\widehat{\alpha}(I(Z))$ and it is semi-effective
if $t>\widehat{\alpha}(I(Z))$. (Semi-effectivity is not clear when $t=\widehat{\alpha}(I(Z))$.)
Thus knowing the value of $\widehat{\alpha}(I(Z))$ or at least having bounds on 
$\widehat{\alpha}(I(Z))$ is useful in trying to address Problem \ref{EVprob}.

One can give an upper bound for $\widehat{\alpha}(I(Z))$ that does not depend on the positions of the points.
No examples are known of this bound being attained when it is not rational.

\begin{exercise}\label{Roe}
Let $Z=m_1p_1+\cdots+m_sp_s$ be a nonzero fat point subscheme of $\P^N$. 
Then $\widehat{\alpha}(I(Z))\leq\sqrt[N]{\sum_im_i^N}$.
\end{exercise}

\SolnEater{\vskip\baselineskip 
\noindent{\it Details}: Note that the point $p_i$ of multiplicity $m_i$ imposes $\binom{m_i+N-1}{N}$
conditions for a form to vanish at $p_i$ (i.e., there are that many
homogeneous linear equations
on the coefficients of a form that need to be met in order for
the form to vanish at the point). Pick a positive integer $n$.
Thus there are at most $\sum_i\binom{mnm_i+N-1}{N}$ conditions
for a form to be in $I(mnZ)$. Expanding $\sum_i\binom{mnm_i+N-1}{N}$
as a polynomial in $m$ we get a polynomial whose leading term
is $m^Nn^N\frac{\sum_im_i^N}{N!}$. 
Let $d_m=\raise2pt\hbox{$\Big\lceil$}\frac{m}{n}+m\sqrt[N]{\sum_im_i^N}\raise3pt\hbox{$\Big\rceil$}$. 
Then the number of forms of degree $nd_m$
is $\binom{nd_m+N}{N}$, which is bounded below by a polynomial in $m$ whose
leading term is at least $m^Nn^N(\frac{1}{n^N(N!)}+\frac{\sum_im_i^N}{N!})$. 
Thus for $m\gg0$, there must be a nonzero solution,
hence $\alpha(I(mnZ))\leq nd_m$, hence $\widehat{\alpha}(I(Z))\leq \lim_{m\to\infty}\frac{d_m}{m}=\frac{1}{n}+\sqrt[N]{\sum_im_i^N}$.
This is true for each $n$, hence $\widehat{\alpha}(I(Z))\leq\sqrt[N]{\sum_im_i^N}$.
\newline\qedsymbol\vskip\baselineskip}

By \Exercise\ \ref{Fekete}(f), it is possible to compute $\widehat{\alpha}(I(Z))$
to any desired number of decimal places by just computing $\alpha(I(mZ))$ for large $m$.
Thus for any real number $a\neq\widehat{\alpha}(I(Z))$, it is possible to computationally verify
that $a\neq\widehat{\alpha}(I(Z))$. What is not clear is how to computationally verify
that $a=\widehat{\alpha}(I(Z))$ when $a$ in fact does equal $\widehat{\alpha}(I(Z))$.

\begin{corollary}\label{WCcor}
Let $Z=m_1p_1+\cdots+m_sp_s\subseteq\P^N$ be a nonzero fat point subscheme. Let $t$ be rational.
If $t>\widehat{\alpha}(I(Z))$, then $\dim [I(mZ)]_{mt}>0$ for all $m\gg0$ such that $mt$ is an integer, and
if $t<\widehat{\alpha}(I(Z))$, then $\dim [I(mZ)]_{mt}=0$ for all $m>0$ such that $mt$ is an integer.
\end{corollary}

\begin{proof}
Say $t>\widehat{\alpha}(I(Z))$. Then for $m\gg0$ such that $mt$ is an integer, we have $mt > \alpha(I(mZ))$, so $\dim [I(mZ)]_{mt}>0$.
If $t<\widehat{\alpha}(I(Z))$, then $mt<m\widehat{\alpha}(I(Z))\leq \alpha(I(mZ))$ for all $m$ such that $mt$ is an integer, so $\dim [I(mZ)]_{mt}=0$.
\end{proof}

In addition to computing Waldschmidt constants, recent work \cite{2refCKLLMMT} raises the question
of how large the least $m$ can be in Problem \ref{EVprob}, 
given that $h^0(X,\O_X(tmL-mE_Z))>0$ for some $m>0$.

\begin{exercise}\label{starconfigs}
Let $r>1$. Given distinct lines $L_1,\ldots,L_{2r}\subseteq\P^2$ with $t_k=0$ for $k>2$, let $Z=p_1+\cdots+p_s$
be the $t_2$ points of intersections of the lines (so $t_2=\binom{2r}{2}$).
Then $\dim [I(mZ)]_{mr}=0$ for all odd $m>0$ and $\dim [I(mZ)]_{mr}=1$ for all even $m>0$.
We can conclude that $\widehat{\alpha}(I(Z))=r$ and that the least $m$ such that $h^0(X,\O_X(mrL-mE_Z))>0$
is $m=2$. Moreover, $h^0(X,\O_X(2rL-2E_Z))=1$ and the intersection matrix of the components of the
unique divisor in $|2(rL-E_Z)|$ is negative definite.
\end{exercise}

\SolnEater{\vskip\baselineskip 
\noindent{\it Details}: Note that each line contains $2r-1$ of the $t_2$ points.
Let $\ell_i$ be the linear form defining $L_i$ for each $i$.
Suppose there is a form $F$ of degree $mr$ in $I(mZ)$. The form thus vanishes 
to order $(2r-1)m$ on $L_i$. Since $r>1$, we have $(2r-1)m>rm =\deg(F)$,
so $\ell_i$ divides $F$. This is true for all $i$, so 
$F=G\ell_1\cdots\ell_{2r}$ for some $G\in I((m-2)Z)$.
If $m=1$, this means $G=0$. If $m=2$, this means $G$ is
a constant times $\ell_1\cdots\ell_{2r}$.
By induction we get that $G=0$ if $m$ is odd,
and $G$ is a constant times $(\ell_1\cdots\ell_{2r})^{m/2}$ if $m$ is even.
I.e., $\dim [I(mZ)]_{mr}=0$ for all odd $m>0$ and $\dim [I(mZ)]_{mr}=1$ for all even $m>0$.
Taking the limit over odd $m$, we get $\widehat{\alpha}(I(Z))\geq r$ and over even $m$ we get
$\widehat{\alpha}(I(Z))\leq r$, hence $\widehat{\alpha}(I(Z))=r$.
This also shows $\dim [I(Z)]_r=0$ but $\dim [I(2Z)]_{2r}>0$,
the least $m$ such that $h^0(X,\O_X(mrL-mE_Z))>0$ is $m=2$.

We saw that  $[I(2Z)]_{2r}$ is 1-dimensional, spanned by
$\ell_1\cdots\ell_{2r}$. I.e., $|2rL-2E_Z|$ has a unique element $D$,
the sum of the proper transforms $L_i'$ of the lines $L_i$. 
Note that $(L_i')^2=-2(r-1)$ and $L_i'\cdot L_j'=0$ for $i\neq j$, so the
intersection matrix of the components of $D$ is negative definite.
\newline\qedsymbol\vskip\baselineskip}

\begin{exercise}\label{B3ex}
Let $Z$ be the reduced scheme consisting of the 9 crossing points defined in Figure \ref{B3config}.
Then the least $m>0$ such that $\dim [I(m2Z)]_{5m}>0$ is $m=2$.

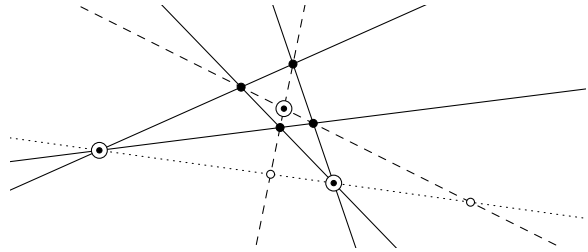
\begin{figure}[htbp]
\begin{center}
\definecolor{ffffff}{rgb}{1.0,1.0,1.0}
\definecolor{uuuuuu}{rgb}{0.26666666666666666,0.26666666666666666,0.26666666666666666}
\definecolor{qqffqq}{rgb}{0.0,1.0,0.0}
\definecolor{qqqqff}{rgb}{0.0,0.0,1.0}
\definecolor{ffqqqq}{rgb}{1.0,0.0,0.0}
\begin{tikzpicture}[line cap=round,line join=round,>=triangle 45,x=.66cm,y=.66cm]
\clip(-4.3,-.58) rectangle (7.3,4.3);
\draw [color=black,domain=-4.3:7.3] plot(\x,{(--9.732000000000001--1.7400000000000002*\x)/3.9});
\draw [color=black,domain=-4.3:7.3] plot(\x,{(--6.1732000000000005--0.4600000000000002*\x)/3.6399999999999997});
\draw [color=black,domain=-4.3:7.3] plot(\x,{(-3.2640000000000007--1.12*\x)/-1.0800000000000003});
\draw [color=black,domain=-4.3:7.3] plot(\x,{(-5.918400000000001--2.4000000000000004*\x)/-0.8200000000000003});
\draw [dashed, color=black,domain=-4.3:7.3] plot(\x,{(--0.9807999999999999-1.28*\x)/-0.26});
\draw [dashed, color=black,domain=-4.3:7.3] plot(\x,{(--4.115843921544174-0.7293839073552844*\x)/1.4532648186217125});
\draw [dotted, domain=-4.3:7.3] plot(\x,{(--4.8636-0.6599999999999999*\x)/4.720000000000001});
\begin{scriptsize}
\draw [fill=white] (-2.5,1.38) circle (3pt);
\draw [fill=black] (-2.5,1.38) circle (1pt);
\draw [fill=black] (1.4,3.12) circle (1.5pt);
\draw [fill=black] (1.14,1.84) circle (1.5pt);
\draw [fill=white] (2.22,0.72) circle (3pt);
\draw [fill=black] (2.22,0.72) circle (1pt);
\draw [fill=black] (0.35520553207837097,2.653860929696504) circle (1.5pt);
\draw [fill=black] (1.8084703507000834,1.9244770223412198) circle (1.5pt);
\draw [fill=white] (0.9486113435910641,0.8977789222944698) circle (1.5pt);
\draw [fill=white] (4.976242847173766,0.3345931612002787) circle (1.5pt);
\draw [fill=white] (1.2174156922996795,2.221123408244577) circle (3pt); 
\draw [fill=black] (1.2174156922996795,2.221123408244577) circle (1pt); 
\end{scriptsize}
\end{tikzpicture}
\caption{Four general points (the small black dots) determine
three reducible conics (given by the three pairs of lines meeting at the large open dotted circles)
which have one singular point each (viz., the three large, dotted circles).
The additional (dotted) line defines two more points (the two small open circles),
thus altogether $4+3+2=9$ points have been specified.}
\label{B3config}
\end{center}
\end{figure}
\end{exercise}

\SolnEater{\vskip\baselineskip 
\noindent{\it Details}: Bezout tells us that any form in $[I(2Z)]_5$ vanishes on all seven lines,
hence must be identically zero. Thus $[I(2Z)]_5=0$.
However, taking the solid lines once each and the dashed and dotted lines 
twice each, we get an element of $[I(m2Z)]_{5m}$ for $m=2$.
\newline\qedsymbol\vskip\baselineskip}

\begin{example}\label{MFOexample}
Consider points $p_1, p_2$ and $p_3$ on an irreducible conic $C'$, and the three points
$p_4, p_5$ and $p_6$ of the conic infinitely near to these first three points, as shown in Figure \ref{MFOexampleFig}
(where the infinitely near points are represented by tangent directions).
Blow up all 6 points to get a surface $X$, let $C$ be the proper transform of $C'$, and let $E_i$ be the blow up of point $p_i$.
Thus $E_i=N_i+E_{i+3}$ for $i=1,2,3$ has two components, as shown. 
Let $L$ be the pullback of a line from $\P^2$ to $X$. Let $F=L-E_4-E_5-E_6$.
Since $F\cdot N_i<0$, if $F$ were linearly equivalent to an effective 
divisor, then $F-N_1-N_2-N_3=L-E_1-E_2-E_3$ would be also, but it is not,
since the points $p_1,p_2,p_3$ are not collinear.

However, $2F\sim D=C+N_1+N_2+N_3$ is linearly equivalent to an effective divisor,
and the intersection matrix of the components of $D$ is clearly negative definite.

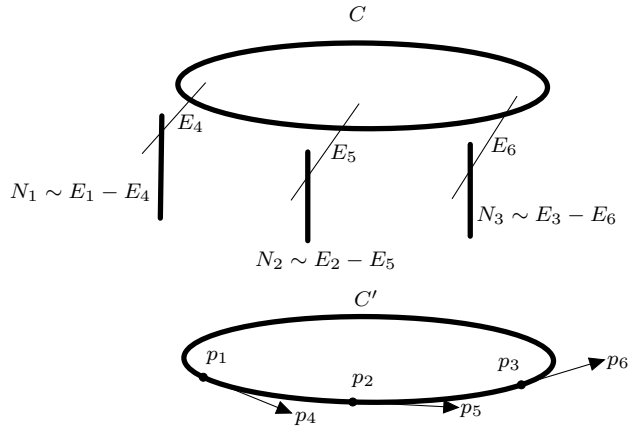
\begin{figure}[htbp]
\begin{center}
\definecolor{xdxdff}{rgb}{0,0,0}
\begin{tikzpicture}[line cap=round,line join=round,>=triangle 45,x=1.0cm,y=1.0cm]
\clip(-4.3,-1.3) rectangle (7.3,4.8);
\draw [rotate around={-0.47945139879656823:(0.7500000000000032,3.4199999999999995)},line width=2.0pt] (0.7500000000000032,3.4199999999999995) ellipse (2.457092923313457cm and 0.5699172166175637cm);
\draw [line width=2.0pt] (-1.94,1.66)-- (-1.92,3.02);
\draw (-2.18,2.52)-- (-1.34,3.46);
\draw [line width=2.0pt] (0.02,1.36)-- (0.02,2.54);
\draw (-0.2,1.9)-- (0.7,3.18);
\draw [line width=2.0pt] (2.18,1.42)-- (2.18,2.64);
\draw (1.94,1.92)-- (2.8,3.28);
\draw [rotate around={-0.47945139879656773:(0.8299999999999997,-0.21999999999999995)},line width=2.0pt] (0.8299999999999997,-0.21999999999999995) ellipse (2.457092923313463cm and 0.5699172166175668cm);
\draw [->] (-1.3690265428333985,-0.45681681022430987) -- (-0.17999999999999994,-0.94);
\draw [->] (0.6181486293852619,-0.7861361362091532) -- (2.04,-0.86);
\draw [->] (2.8592295435913497,-0.5574123445323818) -- (3.98,-0.21999999999999997);
\begin{scriptsize}
\draw[color=black] (0.6799999999999999,4.380000000000001) node {$C$};
\draw[color=black] (-3,2) node {$N_1\sim E_1-E_4$};
\draw[color=black] (-1.54,2.920000000000001) node {$E_4$};
\draw[color=black] (0.26,1.1) node {$N_2 \sim E_2-E_5$};
\draw[color=black] (0.5,2.520000000000001) node {$E_5$};
\draw[color=black] (3.2,1.7) node {$N_3\sim E_3-E_6$};
\draw[color=black] (2.63,2.620000000000001) node {$E_6$};
\draw[color=black] (0.8,.6) node {$C'$};
\draw [fill=xdxdff] (-1.3690265428333985,-0.45681681022430987) circle (1.5pt);
\draw[color=xdxdff] (-1.18,-0.19) node {$p_1$};
\draw[color=xdxdff] (0,-1) node {$p_4$};
\draw [fill=xdxdff] (0.6181486293852619,-0.7861361362091532) circle (1.5pt);
\draw[color=xdxdff] (0.76,-0.49999999999999883) node {$p_2$};
\draw[color=xdxdff] (2.2,-0.9) node {$p_5$};
\draw [fill=xdxdff] (2.8592295435913497,-0.5574123445323818) circle (1.5pt);
\draw[color=xdxdff] (2.7,-0.285) node {$p_3$};
\draw[color=xdxdff] (4.15,-0.26) node {$p_6$};
\end{scriptsize}
\end{tikzpicture}
\caption{A conic with 3 points and 3 infinitely near points blown up. }
\label{MFOexampleFig}
\end{center}
\end{figure}
\end{example}

\begin{exercise}\label{GeneralizeMFOexample}
An example from \cite{InPrep} shows that
Example \ref{MFOexample} generalizes by replacing the conic with a reduced irreducible curve of degree $d>1$ 
to obtain a surface $X$ and a divisor
$F=L-E_{i_1}-\cdots-E_{i_t}$ where $t=\binom{d+1}{2}$ such that the least $m$ with $h^0(X,\O_X(mF))>0$ 
is $d$ and the intersection matrix of the effective divisor 
$D\sim dF$ is negative definite. (This contrasts with other examples in the literature,
where typically either the least $m$ is bounded, or the intersection matrix is not negative definite.)
\end{exercise}

\SolnEater{\vskip\baselineskip 
\noindent{\it Details}: Take $C'$ to be a smooth plane curve of degree $d$.
Let $p_1,\ldots,p_t$, for $t=\binom{d+1}{2}$, be distinct points of $C'$ which do not lie on any
curve of degree $d-1$. This is possible by picking the points one at a time, 
since $C'$ can never be in the base locus of any vector space of forms of degree $d-1$. 
Now successively blow up these points and points on $C'$ infinitely near them for a total of 
$s=d$ 
blow ups at each of the $t$ original points. Let $X$ be the surface obtained after 
doing all of these blow ups.
Index these points so that $p_{i,j}$ is the $j$th point blown up infinitely near to
$p_i$ (so $p_{1,1}=p_1$) and let $E_{i,j}$ be the exceptional curve on $X$
corresponding to $p_{i,j}$ (i.e., the scheme theoretic 
inverse image of $p_{i,j}$ under the blow ups of $p_{i,j}$ and the points blown up subsequent to 
$p_{i,j}$). Denote the prime divisor linearly equivalent to $E_{i,j}-E_{i,j+1}$ by $N_{i,j}$;
thus $E_{i,1}\sim N_{i,1}+N_{i,2}+\cdots+N_{i,s-1}+E_{i,s}$. The proper transform of $C'$
is $C\sim dL-\sum_i\sum_j E_{i,j}$ and, as long as $d\geq2$, we have $C^2=d^2-st=d^2-\binom{d+1}{2}d<0$.

Take $F=L-\sum_iE_{i,s}$. 
If $hF$ is linearly equivalent to an effective divisor for some $h>0$, then so is
$hF+(h-1)\sum_iE_{i,s}\sim hL-\sum_iE_{i,s}$, and, since $F\cdot N_{i,s-1}<0$, 
so is $hF+(h-1)\sum_iE_{i,s}-\sum_iN_{i,s-1}\sim hL-\sum_iE_{i,s}-\sum_iN_{i,s-1}\sim hL-\sum_iE_{i,s-1}$.
But then so is $hL-\sum_iE_{i,s-1}-\sum_iN_{i,s-2}\sim hL-\sum_iE_{i,s-2}$, since
$(hL-\sum_iE_{i,s-1})\cdot N_{i,s-2}<0$. Repeating this, we eventually find that
$hL-\sum_iE_{i,1}$ would be linearly equivalent to an effective divisor,
but by construction this implies $h\geq d$, since there is no curve of degree less than $d$
through the points $p_1,\ldots,p_t$. 

\begin{figure}[htbp]
\begin{center}
\definecolor{qqqqff}{rgb}{0,0,0}
\definecolor{uuuuuu}{rgb}{0,0,0}    
\begin{tikzpicture}[line cap=round,line join=round,>=triangle 45,x=1.0cm,y=1.0cm]
\clip(-4.981707543919536,-5.3) rectangle (7.957347714174262,1.2337441369611737);
\draw[line width=2.0pt](-4.9136623537469735,-0.8872316150000001) -- (-4.833321707519529,-0.8860109118750001) -- (-4.753344989655048,-0.8847902087500001) -- (-4.673731501661603,-0.8835695056250001) -- (-4.59448054504726,-0.8823488025000001) -- (-4.5155914213200905,-0.8811280993750001) -- (-4.4370634319881646,-0.8799073962500001) -- (-4.3588958785595455,-0.8786866931250001) -- (-4.28108806254231,-0.8774659900000001) -- (-4.203639285444522,-0.8762452868750001) -- (-4.1265488487742505,-0.8750245837500001) -- (-4.049816054039569,-0.8738038806250001) -- (-3.9734402027485416,-0.8725831775000001) -- (-3.8974205964092388,-0.8713624743750001) -- (-3.8217565365297297,-0.8701417712500001) -- (-3.7464473246180834,-0.8689210681250001) -- (-3.6714922621823725,-0.8677003650000001) -- (-3.5968906507306606,-0.8664796618750001) -- (-3.522641791771017,-0.8652589587500001) -- (-3.4487449868115156,-0.8640382556250001) -- (-3.3751995373602224,-0.8628175525000001) -- (-3.3020047449252043,-0.8615968493750001) -- (-3.229159911014534,-0.8603761462500001) -- (-3.156664337136281,-0.8591554431250001) -- (-3.084517324798508,-0.8579347400000001) -- (-3.012718175509292,-0.8567140368750001) -- (-2.941266190776698,-0.8554933337500001) -- (-2.8701606721087938,-0.8542726306250001) -- (-2.7994009210136515,-0.8530519275000001) -- (-2.7289862389993385,-0.8518312243750001) -- (-2.658915927573924,-0.8506105212500001) -- (-2.5891892882454766,-0.8493898181250001) -- (-2.5198056225220675,-0.8481691150000001) -- (-2.450764231911762,-0.8469484118750001) -- (-2.382064417922633,-0.8457277087500001) -- (-2.3137054820627494,-0.8445070056250001) -- (-2.2456867258401765,-0.8432863025000001) -- (-2.1106669583392446,-0.8408448962500001) -- (-1.976999527484395,-0.8384034900000001) -- (-1.8446788453401775,-0.8359620837500001) -- (-1.7136993239711433,-0.8335206775000001) -- (-1.5840553754418476,-0.8310792712500001) -- (-1.45574141181684,-0.8286378650000001) -- (-1.3287518451606752,-0.8261964587500001) -- (-1.2030810875379023,-0.8237550525000001) -- (-1.078723551013078,-0.8213136462500001) -- (-0.9556736476507508,-0.8188722400000001) -- (-0.8339257895154741,-0.8164308337500001) -- (-0.7134743886718011,-0.8139894275000001) -- (-0.5943138571842832,-0.8115480212500001) -- (-0.47643860711747354,-0.8091066150000001) -- (-0.3598430505359227,-0.8066652087500001) -- (-0.24452159950418562,-0.8042238025000001) -- (-0.13046866608681107,-0.8017823962500001) -- (-0.0176786623483558,-0.7993409900000001) -- (0.09385399964663055,-0.7968995837500001) -- (0.20413490783359656,-0.7944581775000001) -- (0.31316965014798814,-0.7920167712500001) -- (0.4209638145252539,-0.7895753650000001) -- (0.5275229889008406,-0.7871339587500001) -- (0.6328527612101968,-0.7846925525000001) -- (0.7369587193887703,-0.7822511462500001) -- (0.839846451372007,-0.7798097400000001) -- (0.941521545095358,-0.7773683337500001) -- (1.0419895884942676,-0.7749269275000001) -- (1.141256169504186,-0.7724855212500001) -- (1.2393268760605585,-0.7700441150000001) -- (1.3362072960988352,-0.7676027087500001) -- (1.431903017554462,-0.7651613025000001) -- (1.5264196283628877,-0.7627198962500001) -- (1.619762716459559,-0.7602784900000001) -- (1.7119378697799235,-0.7578370837500001) -- (1.8029506762594307,-0.7553956775000001) -- (1.8928067238335267,-0.7529542712500001) -- (1.9815116004376607,-0.7505128650000001) -- (2.069070894007278,-0.7480714587500001) -- (2.155490192477828,-0.7456300525000001) -- (2.240775083784758,-0.7431886462500001) -- (2.3249311558635157,-0.7407472400000001) -- (2.407963996649549,-0.7383058337500001) -- (2.489879194078305,-0.7358644275000001) -- (2.5706823360852322,-0.7334230212500001) -- (2.650379010605777,-0.7309816150000001) -- (2.728974805575389,-0.7285402087500001) -- (2.8064753089295147,-0.7260988025000001) -- (2.8828861086036017,-0.7236573962500001) -- (2.958212792533098,-0.7212159900000001) -- (3.0324609486534513,-0.7187745837500001) -- (3.105636164900109,-0.7163331775000001) -- (3.1777440292085193,-0.7138917712500001) -- (3.248790129514129,-0.7114503650000001) -- (3.318780053752387,-0.7090089587500001) -- (3.38771938985874,-0.7065675525000001) -- (3.455613725768637,-0.7041261462500001) -- (3.5882897487408494,-0.6992433337500001) -- (3.716852826152607,-0.6943605212500001) -- (3.8413476614874904,-0.6894777087500001) -- (3.961818958229082,-0.6845948962500001) -- (4.078311419860962,-0.6797120837500001) -- (4.190869749866714,-0.6748292712500001) -- (4.299538651729918,-0.6699464587500001) -- (4.404362828934156,-0.6650636462500001) -- (4.5053869849630095,-0.6601808337500001) -- (4.60265582330006,-0.6552980212500001) -- (4.696214047428889,-0.6504152087500001) -- (4.786106360833077,-0.6455323962500001) -- (4.872377466996208,-0.6406495837500001) -- (4.9550720694018615,-0.6357667712500001) -- (5.034234871533621,-0.6308839587500001) -- (5.109910576875065,-0.6260011462500001) -- (5.182143888909778,-0.6211183337500001) -- (5.3164621469933335,-0.6113527087500001) -- (5.437547273652938,-0.6015870837500001) -- (5.545756896757245,-0.5918214587500001) -- (5.641448644174907,-0.5820558337500001) -- (5.724980143774575,-0.5722902087500001) -- (5.856992910994541,-0.5527589587500001) -- (5.944656221366364,-0.5332277087500001) -- (5.990831097839259,-0.5136964587500001) -- (5.998378563362448,-0.49416520875000014) -- (5.970159640885149,-0.47463395875000014) -- (5.909035353356581,-0.45510270875000014) -- (5.8178667237259605,-0.43557145875000014) -- (5.699514774942509,-0.41604020875000014) -- (5.631039125540102,-0.40627458375000014) -- (5.556840529955444,-0.39650895875000014) -- (5.477276616057187,-0.38674333375000014) -- (5.392705011713985,-0.37697770875000014) -- (5.303483344794488,-0.36721208375000014) -- (5.209969243167349,-0.35744645875000014) -- (5.112520334701221,-0.34768083375000014) -- (5.011494247264757,-0.33791520875000014) -- (4.907248608726608,-0.32814958375000014) -- (4.800141046955426,-0.31838395875000014) -- (4.690529189819865,-0.30861833375000014) -- (4.578770665188577,-0.29885270875000014) -- (4.465223100930213,-0.28908708375000014) -- (4.350244124913425,-0.27932145875000014) -- (4.234191365006868,-0.26955583375000014) -- (4.117422449079192,-0.25979020875000014) -- (4.000295004999051,-0.25002458375000014) -- (3.8831666606350956,-0.24025895875000014) -- (3.76639504385598,-0.23049333375000014) -- (3.650337782530355,-0.22072770875000014) -- (3.5353525045268737,-0.21096208375000014) -- (3.4217968377141883,-0.20119645875000014) -- (3.310028409960951,-0.19143083375000014) -- (3.200404849135815,-0.18166520875000014) -- (3.093283783107431,-0.17189958375000014) -- (2.9890228397444525,-0.16213395875000014) -- (2.887979646915532,-0.15236833375000014) -- (2.7905118324893206,-0.14260270875000014) -- (2.696977024334472,-0.13283708375000014) -- (2.607732850319638,-0.12307145875000014) -- (2.5231369383134714,-0.11330583375000014) -- (2.443546916184624,-0.10354020875000014) -- (2.3693204118017483,-0.09377458375000014) -- (2.3008150530334963,-0.08400895875000014) -- (2.182398283815474,-0.06447770875000014) -- (2.0911576314797755,-0.04494645875000014) -- (2.0299541189756205,-0.02541520875000014) -- (2.001648769252227,-0.00588395875000014) -- (2.009102605258814,0.01364729124999986) -- (2.0551766499446003,0.03317854124999986) -- (2.1427319262588043,0.05270979124999986) -- (2.2746294571506454,0.07224104124999986) -- (2.35810063773481,0.08200666624999986) -- (2.453730265569342,0.09177229124999986) -- (2.5618759685228913,0.10153791624999986) -- (2.682895374464112,0.11130354124999986) -- (2.817146111261656,0.12106916624999986) -- (2.889344987440503,0.12595197874999986) -- (2.964985806784175,0.13083479124999986) -- (3.0441132727762548,0.13571760374999986) -- (3.126772088900323,0.14060041624999986) -- (3.2130069586399608,0.14548322874999986) -- (3.3028625854787506,0.15036604124999986) -- (3.3963836729002734,0.15524885374999986) -- (3.493614924388111,0.16013166624999986) -- (3.594601043425844,0.16501447874999986) -- (3.699386733497056,0.16989729124999986) -- (3.8080166980853267,0.17478010374999986) -- (3.9205356406742387,0.17966291624999986) -- (4.0369882647473725,0.18454572874999986) -- (4.15741927378831,0.18942854124999986) -- (4.281873371280634,0.19431135374999986) -- (4.410395260707925,0.19919416624999986) -- (4.543029645553764,0.20407697874999986) -- (4.610902993597259,0.20651838499999986) -- (4.679821229301734,0.20895979124999986) -- (4.749789940602636,0.21140119749999986) -- (4.820814715435414,0.21384260374999986) -- (4.892901141735516,0.21628400999999986) -- (4.966054807438389,0.21872541624999986) -- (5.04028130047948,0.22116682249999986) -- (5.115586208794237,0.22360822874999986) -- (5.1919751203181095,0.22604963499999986) -- (5.269453622986543,0.22849104124999986) -- (5.348027304734986,0.23093244749999986) -- (5.427701753498886,0.23337385374999986) -- (5.508482557213691,0.23581525999999986) -- (5.590375303814849,0.23825666624999986) -- (5.673385581237806,0.24069807249999986) -- (5.7575189774180116,0.24313947874999986) -- (5.842781080290913,0.24558088499999986) -- (5.929177477791958,0.24802229124999986) -- (6.016713757856594,0.25046369749999986) -- (6.105395508420267,0.25290510374999986) -- (6.195228317418429,0.25534650999999986) -- (6.286217772786523,0.25778791624999986) -- (6.378369462459999,0.26022932249999986) -- (6.471688974374306,0.26267072874999986) -- (6.566181896464888,0.26511213499999986) -- (6.661853816667197,0.26755354124999986) -- (6.758710322916677,0.26999494749999986) -- (6.856757003148778,0.27243635374999986) -- (6.955999445298946,0.27487775999999986) -- (7.056443237302632,0.27731916624999986) -- (7.1580939670952795,0.27976057249999986) -- (7.260957222612338,0.28220197874999986) -- (7.365038591789255,0.28464338499999986) -- (7.470343662561479,0.28708479124999986) -- (7.576878022864457,0.28952619749999986) -- (7.6846472606336365,0.29196760374999986) -- (7.793656963804466,0.29440900999999986) -- (7.903912720312392,0.29685041624999986) -- (8.015420118092862,0.29929182249999986);
\draw [line width=2.0pt](-3.888580461770232,-5.254333064837555)-- (-3.8662717458080014,-3.4473270718968716);
\draw [line width=2.0pt](-4.178593769279231,-4.138897266726022)-- (-2.86237952750762,-2.610361001483184);
\draw [line width=2.0pt](-3.1300841190543887,-3.358092208047949)-- (-3.1300841190543887,-1.6180123629939573);
\draw (-3.3531712786766956,-2.121904581294337)-- (-2.036957036905085,-0.5933683160514991);
\draw [line width=2.0pt](0.7815925358467538,-5.1325366423554275)-- (0.8039012518089849,-3.3255306494147483);
\draw [line width=2.0pt](0.491579228337755,-4.017100844243898)-- (1.807793470109367,-2.4885645790010607);
\draw [line width=2.0pt](1.5400888785625988,-3.236295785565826)-- (1.5400888785625988,-1.4962159405118345);
\draw (1.3170017189402914,-2.0001081588122136)-- (2.6332159607119032,-0.47157189356937146);
\draw (3.2948260780680494,0.8690202449647235) node[anchor=north west] {$C$};
\begin{scriptsize}
\draw [fill=uuuuuu] (9.267,0.325) circle (2.0pt);
\draw[color=qqqqff] (-2.5,-4.931245905215249) node {$N_{1,1}\sim E_{1,1}-E_{1,2}$};
\draw[color=qqqqff] (-1.8,-2.3) node {$N_{1,3}\sim E_{1,3}-E_{1,4}$};
\draw[color=qqqqff] (-2.4,-3.7) node {$N_{1,2}\sim E_{1,2}-E_{1,3}$};
\draw[color=black] (-2.25,-1.339542635296112) node {$E_{1,4}$};
\draw [fill=qqqqff] (-.2,-1.3)  circle (1.5pt);
\draw [fill=qqqqff] (-.5,-1.3)  circle (1.5pt);
\draw [fill=qqqqff] (-.8,-1.3)  circle (1.5pt);
\draw[color=qqqqff] (2.1,-4.752776177517403) node {$N_{t,1}\sim E_{t,1}-E_{t,2}$};
\draw[color=black] (2.25,-3.6) node {$N_{t,2}\sim E_{t,2}-E_{t,3}$};
\draw[color=black] (2.78,-2.15) node {$N_{t,3}\sim E_{t,3}-E_{t,4}$};
\draw[color=black] (2.4,-1.205690339522728) node {$E_{t,4}$};
\end{scriptsize}
\end{tikzpicture}
\caption{A curve $C$ obtained by blowing up $s=4$ times at each of $t$ points on a smooth plane curve $C'$ of degree $d$. }
\label{MFOexampleFig2}
\end{center}
\end{figure}
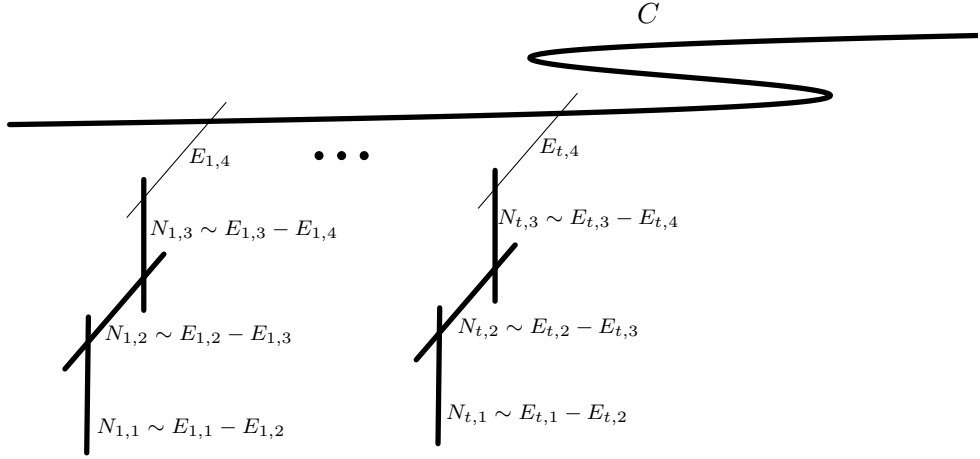

To see in fact that $dF$ is linearly equivalent to an effective divisor, 
note that the total transform of $C'\sim dL$ is 
$C+\sum_{i=1}^t\sum_{j=1}^{s-1}jN_{i,j}+\sum_{i=1}^tsE_{i,s}$.
Since $s=d$, we have $dF=dL-d\sum_iE_{i,s}\sim C+\sum_{i=1}^t\sum_{j=1}^{s-1}jN_{i,j}$. 
The intersection matrix of the components
is a block diagonal matrix with one $1\times1$ block of $C^2$,
and $t$ blocks of size $(s-1)\times (s-1)$ with each diagonal entry being $-2$
and each entry just above and just below the diagonal being 1.
It is not hard to show that each block is negative definite, and thus
the whole matrix is negative definite.
(The $t$ blocks of size $(s-1)\times (s-1)$ correspond to 
the divisors $N_{i,1},\ldots,N_{i,s-1}$, for $1\leq i\leq t$. 
The span of their divisor classes in the divisor class group tensored by the rationals
is the same as 
$N_{i,1}=E_{i,1}-E_{i,2}$, $N_{i,1}+2N_{i,2}=E_{i,1}+E_{i,2}-2E_{i,3}$, $\ldots$, $N_{i,1}+2N_{i,2}+\cdots+(s-1)N_{i,s-1}=E_{i,1}+\ldots+E_{i,s-2}-(s-1)E_{i,s-1}$.
It is easy to see that these are orthogonal with each class having negative self-intersection, 
and hence the subspace spanned by $N_{i,1},\ldots,N_{i,s-1}$ is negative definite.)
\newline\qedsymbol\vskip\baselineskip}

We now present some examples bounding or computing Waldschmidt constants, and relating this to
the question of how large the least $m$ can be in Problem \ref{EVprob}.

\begin{exercise}\label{Fermatn=2}
Let $Z=p_1+\cdots+p_7$ for the 7 points $p_i$ of the Fermat arrangement for $n=2$.
Recall that the Fermat arrangement consists of $n^2+3$ points, three of which are the 
coordinate vertices of $\P^2$; assume that these three are $p_5, p_6$ and $p_7$. 
\begin{enumerate}
\item[(a)] Then $h^0(X,\O_X(3F))>0$ for $F=5L-2E_Z$; conclude that $\alpha(I(6mZ))\leq 15m$.
(Note: One can show that $|3F|$ contains a curve which is a sum of proper transforms of lines.)
\item[(b)] The least $m>0$ such that $h^0(X,\O_X(mF))>0$ is $m=2$.
\item[(c)] We have that $H=4L-E_1-\cdots-E_4-2(E_5+E_6+E_7)$ is nef.
(Recall that {\it nef} means $H\cdot C\geq 0$ for every effective divisor $C$.
Note: One can show that $|3H|$ or even $|H|$ contains a curve $B'$ which is a sum of the proper transforms $B_i$ of lines,
and that $H\cdot B_i\geq 0$ for each summand.)
\item[(d)] One can conclude that $\alpha(I(6mZ))\geq 15m$.
(Note that $H\cdot F=0$.)
\item[(e)] It follows that $\widehat{\alpha}(I(Z))=\frac{15}{6}=2.5$.
\end{enumerate}
\end{exercise}

\SolnEater{\vskip\baselineskip 
\noindent{\it Details}: (a) Take the 6 lines of the Fermat twice each, and the three coordinate axes
once each. That gives a form $D$ of degree 15 vanishing to order 6 at each of the Fermat points.
Now $D^m$ has degree $15m$ and is in $I(6mZ)$, so $\alpha(I(6mZ))\leq 15m$ and
$\dim [I(6Z)]_{15}>0$; i.e., $h^0(X,\O_X(3F))>0$.

(b) It is enough to show $\dim [I(2Z)]_5=0$ but $\dim [I(4Z)]_{10}>0$. A form of degree 5
vanishing to order 2 at each of the points of $Z$, has a total of 6 roots on each Fermat line,
so would by Bezout's Theorem have to be divisible by the linear
forms of all 6 Fermat lines. Thus it is 0, so $\dim [I(2Z)]_5=0$.
To see $\dim [I(4Z)]_{10}>0$, take the 6 Fermat lines, add to them the two Fermat lines
through one coordinate vertex, and add to this the line through the other two
coordinate vertices, taken twice. This gives a divisor $A$ defined by 
a form of degree 10 vanishing to order 6 at all 7 points.

(c) Let $B$ consist of the two Fermat lines
through one coordinate vertex, and add to this the line through the other two
coordinate vertices, taken twice. This gives a curve of degree
4 vanishing to order 1 at the points $p_1,\ldots,p_4$ and to order 2
at $p_5,p_6$ and $p_7$. Its proper transform $B'$ meets each of its components
nonnegatively, hence $B'$ is nef, and thus so is $H$ (since $H$ is linearly equivalent
to $B'$).

(d) Let $t=\alpha(I(6mZ))$. Then $|tL-6m(E_1+\cdots+E_7)|\neq\varnothing$,
so $0\leq (tL-6m(E_1+\cdots+E_7))\cdot H = 4t-24m-36m$, hence $t\geq 15m$.

(e) Together, (a) and (d) give $\alpha(I(6mZ))=15m$, hence $\widehat{\alpha}(I(Z))=\frac{15}{6}$.
\newline\qedsymbol\vskip\baselineskip}

\begin{exercise}\label{Fermatn>2}
Let $Z=p_1+\cdots+p_s$ be the $s=n^2+3$ points $p_i$ of the Fermat arrangement for $n>2$,
where $p_{n^2+1}=p_{s-2}, p_{n^2+2}=p_{s-1}, p_{n^2+3}=p_s$ are the coordinate vertices.
Let $Y=p_1+\cdots+p_{n^2}$. 
\begin{enumerate}
\item[(a)] One can show that $n=\widehat{\alpha}(I(Y))\leq\widehat{\alpha}(I(Z))$.
\item[(b)] The least $m>0$ such that $\dim [I(mZ)]_{mn}>0$ is $m=3$,
hence $\alpha(I(3Z))\leq 3n$ and thus that $\widehat{\alpha}(I(Z))=n$.
\end{enumerate}
\end{exercise}

\SolnEater{\vskip\baselineskip 
\noindent{\it Details}: (a) By Bezout, a form $F$ of degree $t<mn$ vanishing to order $m$ at each point of $Y$
is divisible by the linear form defining every line through $n$ of the points. 
Thus $F$ is divisible by $x^n-y^n$. Factoring it out and applying induction, we see
in fact that $(x^n-y^n)^m$ divides $F$, hence $F=0$. Thus $\alpha(I(mY))\geq mn$.
Since $(x^n-y^n)^m\in I(mY)$, we get $\alpha(I(mY)) = mn$ and hence $n=\widehat{\alpha}(I(Y))$.
Since $mY\subseteq mZ$, we have $I(mZ)\subseteq I(mY)$ and hence 
$\alpha(I(mY))\leq\alpha(I(mZ))$ and so $\widehat{\alpha}(I(Y))\leq\widehat{\alpha}(I(Z))$.

(b) By Bezout, for every line through $n+1$ of the points of $Z$, the line's defining form must divide any form in
$[I(mZ)]_{mn}$. In order for $[I(mZ)]_{mn}\neq0$, we thus need the number of such lines, which is 
at least $3n$, to satisfy $3n\leq mn$; i.e., $m\geq3$. Since the $3n$ Fermat lines indeed give a nonzero element
of $[I(3Z)]_{3n}$, we see that the least $m>0$ such that $\dim [I(mZ)]_{mn}>0$ is $m=3$.
This also shows that $\widehat{\alpha}(I(Z))\leq n$, so applying (a) gives $\widehat{\alpha}(I(Z))= n$.
\newline\qedsymbol\vskip\baselineskip}

\begin{exercise}\label{Klein Lst m}
Let $Z=p_1+\cdots+p_{49}$ be the points of the Klein arrangement. 
One can show that the least $m>0$ such that $\dim [I(mZ)]_{m7}>0$ is $m=3$,
hence $\widehat{\alpha}(I(Z))\leq 7$ and $\widehat{\alpha}(I(3Z))\leq 21$. 
\end{exercise}

\SolnEater{\vskip\baselineskip 
\noindent{\it Details}: A nonzero form $F$ in $[I(mZ)]_{7m}$ has $8m$ roots on every Klein line
but degree only $7m$, and so by Bezout all 21 lines must give factors of $F$. 
Thus we must have $7m=\deg(F)\geq 21$, hence $m\geq 3$. Since the 21 lines do
give a nonzero form in $[I(3Z)]_{21}$, we see that the least $m>0$ such that $\dim [I(mZ)]_{m7}>0$ is $m=3$,
and that $\widehat{\alpha}(I(Z))\leq 7$ and $\widehat{\alpha}(I(3Z))\leq 21$.

Alternatively, let $X=p_1+\cdots+p_{21}$ and $Y=p_{22}+\cdots+p_{49}$,
where $X$ consists of the $t_4=21$ points of the Klein arrangement of multiplicity 4
and $Y$ consists of the $t_3=28$ points of the Klein arrangement of multiplicity 3.
Let $V=4X+3Y$. One can show that $\widehat{\alpha}(I(V))=21$, and hence by \Exercise\ \ref{compare}
that $\widehat{\alpha}(I(3Z))\leq 21$, so $\widehat{\alpha}(I(Z))\leq 7$ by \Exercise\ \ref{Fekete}(e).
\newline\qedsymbol\vskip\baselineskip}

\begin{exercise}\label{Wiman Lst m}
Let $Z=p_1+\cdots+p_{201}$ be the points of the Wiman arrangement. 
Then the least $m>0$ such that $\dim [I(mZ)]_{m15}>0$ is $m=3$, and one can
conclude that $\widehat{\alpha}(I(Z))\leq 15$ and $\widehat{\alpha}(I(3Z))\leq 45$. 
\end{exercise}

\SolnEater{\vskip\baselineskip 
\noindent{\it Details}: A nonzero form $F$ in $[I(mZ)]_{15m}$ has $16m$ roots on every Wiman line
but degree only $15m$, and so by Bezout all 45 lines must give factors of $F$. 
Thus we must have $15m=\deg(F)\geq 45$, hence $m\geq 3$. Since the 45 lines do
give a nonzero form in $[I(3Z)]_{45}$, we see that the least $m>0$ such that $\dim [I(mZ)]_{m15}>0$ is $m=3$,
and that $\widehat{\alpha}(I(Z))\leq 15$ and $\widehat{\alpha}(I(3Z))\leq 45$.

Alternatively, let $W=p_1+\cdots+p_{36}$, $X=p_{37}+\cdots+p_{81}$ and $Y=p_{82}+\cdots+p_{201}$
where $W$ consists of the $t_5=36$ points of the Wiman arrangement of multiplicity 5,
$X$ consists of the $t_4=45$ points of the Wiman arrangement of multiplicity 4 and
$Y$ consists of the $t_3=120$ points of the Wiman arrangement of multiplicity 3.
Let $V=5W+4X+3Y$. One can show that $\widehat{\alpha}(I(V))=45$, and hence by \Exercise\ \ref{compare}
that $\widehat{\alpha}(I(3Z))\leq 45$, so $\widehat{\alpha}(I(Z))\leq 15$ by \Exercise\ \ref{Fekete}(e).
\newline\qedsymbol\vskip\baselineskip}

If $Z$ is the reduced scheme of singular points of the Wiman arrangement of 45 lines, then $\widehat{\alpha}(I(Z))=27/2$
\cite{2refBetal}. The Klein is a little harder, but it is looking like $\widehat{\alpha}(I(Z))=13/2$ for the Klein \cite{2refBetal}.

\begin{openproblem}\label{WCKlein}
Compute $\widehat{\alpha}(I(Z))$ if $Z=\sum_ip_i$ is the reduced scheme consisting of the crossing points
of the Klein arrangement of lines.
\end{openproblem}

Given a fat point scheme $Z\subseteq \P^2$ and a positive integer $t$, the least $m$ 
such that $\dim [I(mZ)]_{mt}>0$, when such an $m$ exists, 
can be bigger than just 3, even without using infinitely near points (as was done in Example \ref{GeneralizeMFOexample}).

\begin{exercise}\label{char p Lst m}
Assume $\chr(\KK)>0$. Let ${\mathbb F}_q\subseteq \KK$ be a subfield of order $q$.
Let $Z=p_1+\cdots+p_s$ be all but one of the points of $\P^2$ defined over ${\mathbb F}_q$
(so $s=q^2+q$). Then the least $m>0$ such that $\dim [I(mZ)]_{mq}>0$
is $m=q$. Moreover, $\widehat{\alpha}(I(Z))=q$.
\end{exercise}

\SolnEater{\vskip\baselineskip 
\noindent{\it Details}: There are $q^2$ ${\mathbb F}_q$-lines that do not contain the missing point.
A nonzero form $F$ in $[I(mZ)]_{qm}$ has $(q+1)m$ roots on every such line
but degree only $qm$, and so by Bezout all $q^2$ lines must give factors of $F$. 
Thus we must have $qm=\deg(F)\geq q^2$, hence $m\geq q$. Since the $q^2$ lines do
give a nonzero form in $[I(qZ)]_{q^2}$, we see that the least $m>0$ such that $\dim [I(mZ)]_{mq}>0$ is $m=q$,
and that $\widehat{\alpha}(I(Z))\leq q$. Since a form in $[I(mqZ)]_t$ for $t\leq q^2m$
is divisible by the form $Q$ defined by the $q^2$ lines, an induction argument shows 
$\alpha(I(mqZ))=q^2m$ and hence $\widehat{\alpha}(I(Z)) = q$
(note that $\deg(Q)=q^2$ and $Q$ vanishes to order $q$ at each of the points of $Z$).
\newline\qedsymbol\vskip\baselineskip}

For additional examples, it is helpful to know the dimension of $[I(Z)]_t$ in each $t$.
For general points in $\P^2$, there is a conjecture for this, the SHGH Conjecture.
But first, we put it in context by recalling a general result.

\begin{theorem}\label{virtdim}
Given a fat point subscheme $Z=m_1p_1+\cdots+m_sp_s\subseteq\P^N$, we have
$$\dim [I(Z)]_t\geq\max\Big\{0,\binom{t+N}{N}-\sum_i\binom{m_i+N-1}{N}\Big\},$$
with equality for $t\geq \sum_im_i-1$.
\end{theorem}

\begin{proof}
Let $I=I(Z)$. The forms in $[I]_t$ are the solutions to $\sum_i\binom{m_i+N-1}{N}$ homogeneous
linear equations (possibly not independent) on the $\binom{t+N}{N}$ dimensional vector space of forms
of degree $t$ (i.e., vanishing on $Z$ imposes 
$\sum_i\binom{m_i+N-1}{N}$ conditions on all forms of degree $t$),
so we get the lower bound on the dimension as claimed.

The equality can be thought of as a form of the Chinese Remainder Theorem.
Let $R=\KK[\P^N]=\KK[x_0,\ldots,x_N]$. Let $S=\KK[y_1,\ldots,y_N]$, where we think of $y_i$ as $x_i/x_0$,
assuming that the coordinates $x_i$ have been chosen such that $x_0$ does not 
vanish at any of the points $p_i$. If $p_i=(a_0,\ldots,a_n)$, let $q_i=(a_1/a_0,\ldots,a_N/a_0)$.
Define $J=J(q_1)^{m_1}\cdots J(q_s)^{m_s}$, where $J(q_i)$ is the ideal of all polynomials in $S$
that vanish at $q_i$. We have a vector space isomorphism $(S/J)_t\cong[R/I]_t=[R]_t/[I]_t$
given for any polynomial in $S$ of degree at most $t$ 
by $f(y_1,\ldots,y_N)\mapsto x_0^tf(x_1/x_0,\ldots,x_N/x_0)$, 
where by $(S/J)_t$ we mean the vector space image under $S\to S/J$ of all polynomials of degree
$t$ or less in $S$.

The ideals $J(q_i)^{m_i}$ are pairwise coprime, so $J=\cap_iJ(q_i)^{m_i}$ and
$S\cong \oplus_i S/J(q_i)^{m_i}$. Since up to a linear change of coordinates $S/J(q_i)^{m_i}$ is
$S/(y_1,\ldots,y_N)^{m_i}$, we see that $\dim S/J(q_i)^{m_i}=\binom{m_i+N-1}{N}$, hence
$\dim S/J=\sum_i S/J(q_i)^{m_i}=\sum_i \binom{m_i+N-1}{N}$, so for $t\gg0$ we have 
$S/J=(S/J)_t\cong[R/I]_t=[R]_t/[I]_t$, hence $\dim [I(Z)]_t=\dim [R]_t - \dim[R/I]_t
=\binom{t+N}{N}-\sum_i \binom{m_i+N-1}{N}$ for $t\gg0$. 

The inverse isomorphism $\sum_i S/J(q_i)^{m_i}\to S/J$ is given by
$(f_,\ldots,f_s)\mapsto \sum_if_ig_i$, where we can represent $f_i$
by a polynomial of degree $m_i-1$ and $g_i$ is represented by a polynomial that doesn't vanish
at $q_i$ and is in $\Pi_{j\neq i}J(q_j)^{m_j}$. By picking linear forms $L_i$ that vanish at $q_i$ but not
at any other $q_j$, we can take $g_i$ to be $L_1^{m_1}\cdots L_s^{m_s}/L_i^{m_i}$.
Thus $\deg(f_ig_i) = \sum_i m_i-1$, so for $t=\sum_i m_i-1$ we have isomorphisms
$S/J=(S/J)_t\cong[R/I]_t=[R]_t/[I]_t$, hence 
$\dim [I(Z)]_t=\binom{t+N}{N}-\sum_i \binom{m_i+N-1}{N}$ for $t\geq \sum_i m_i-1$.
\end{proof}

When $N=2$, this also follows from Riemann-Roch for a blow up $X$ of $\P^2$.

\begin{exercise}\label{RR}
Given distinct points $p_1,\ldots,p_s\in \P^2$ and integers $t, m_1,\ldots,m_s\geq0$, let $Z=m_1p_1+\cdots+m_sp_s$.
Using Riemann-Roch and Serre duality with $F=tL-E_Z$, one can show that 
$$h^0(X, \O_X(F))\geq \frac{F^2-K_XF}{2}+1,$$
and conclude that $$\dim [I(Z)]_t=h^0(X, \O_X(F))\geq\max\Big\{0,\binom{t+2}{2}-\sum_i\binom{m_i+1}{2}\Big\}.$$
\end{exercise}

\SolnEater{\vskip\baselineskip 
\noindent{\it Details}: Since $t\geq0$, we see $(K_X-F)\cdot L<0$. Since $L$ is nef, this means
$0=h^0(X, \O_X(K_X-F))=h^2(X, \O_X(F))$. Thus 
$$h^0(X, \O_X(F))\geq h^0(X, \O_X(F))-h^1(X, \O_X(F))+h^2(X, \O_X(F))=\frac{F^2-K_XF}{2}+1.$$
Of course, $h^0(X, \O_X(F))\geq0$, while $\dim [I(Z)]_t=h^0(X, \O_X(F))$ is \Exercise\ \ref{iso}.
Finally, using the intersection form, $\frac{F^2-K_XF}{2}+1$ simplifies to $\binom{t+2}{2}-\sum_i\binom{m_i+1}{2}$.
\newline\qedsymbol\vskip\baselineskip}

We now recall a special case of the SHGH Conjecture
(see \cite{2refSe, 2refVanc, 2refG, 2refHi} for various equivalent versions of the full conjecture).

\begin{conjecture}\label{SHGHconj}
Let $Z=p_1+\cdots+p_s\subseteq \P^2$ for general points $p_i$, where either $s$ is a square or $s>8$.
Then $\dim [I(mZ)]_t=\max\Big\{0,\binom{t+2}{2}-s\binom{m+1}{2}\Big\}$.
\end{conjecture}

Conjecture \ref{SHGHconj} is known to be true when $s$ is a square \cite{2refEvain, 2refCM, 2Quim1}.

\begin{exercise}\label{s^2 points}
Consider $Z=p_1+\cdots+p_{s^2}$ for $s^2$ general points $p_i$ for $s>6$.
Let $F=(s+1)L-E_Z$. Then the least $m$ such that
$h^0(X, \O_X(mF))>0$ is $m=\lceil \frac{s-3}{2}\rceil$.
\end{exercise}

\SolnEater{\vskip\baselineskip 
\noindent{\it Details}: The least $m$ such that $h^0(X, \O_X(mF))>0$
is, by Conjecture \ref{SHGHconj} and \Exercise\ \ref{RR}, 
the least $m$ such that $\binom{m(s+1)+2}{2}>s^2\binom{m+1}{2}$.
This simplifies to $(2s+1)m^2-(s^2-3s-3)m+2>0$.
This is positive for $m=(s-3)/2$ and negative for $m=(s-4)/2$,
so the least (integral) $m$ is $m=\lceil \frac{s-3}{2}\rceil$.
\newline\qedsymbol\vskip\baselineskip}

Similar examples are expected to arise where the least 
$m$ can be arbitrarily large even when the number of points is fixed,
but these examples are still only conjectural, since they assume the SHGH Conjecture.

\begin{exercise}\label{Padova example}
(See \cite{2refCKLLMMT}.) Let $s>49$ not be a square and consider 
positive integers $t$ and $r$ such that $t^2-sr^2=1$.
Let $Z=rp_1+\cdots+rp_s$ for general points $p_i\in\P^2$.
Let $F=tL-E_Z=tL-r(E_1+\cdots+E_s)$. Since $F^2>0$ and $F\cdot L>0$,
we know $h^0(X, \O_X(mF))>0$ for $m\gg0$. 
Assuming the SHGH Conjecture, it follows that the least such $m$ 
satisfies $m>r(s-3\sqrt{2s})/2$.
(Since there are examples of $t$ and $r$ with $t^2-sr^2=1$ and $r$ arbitrarily large,
there is no bound on the least $m$ such that $h^0(X, \O_X(mF))>0$.)
\end{exercise}

\SolnEater{\vskip\baselineskip 
\noindent{\it Details}: Assuming the SHGH Conjecture, the least $m$ is the one giving
$$\binom{tm+2}{2}-s\binom{mr+1}{2}>0.$$
This simplifies to $m^2-(sr-3t)m+2>0$. 
For $m=(sr-3t)/2$ this is $8>(sr-3t)^2=sr(sr-6t)+9t^2$.
If $s>49$, then $s> 6(\sqrt{s}+1)\geq 6\sqrt{s+\frac{1}{r^2}}=\frac{6t}{r}$,
so $sr-6t>0$, hence $sr(sr-6t)+9t^2\geq 8$. I.e., 
$m^2-(sr-3t)m+2>0$ is still negative for $m=(sr-3t)/2$,
so the least $m$ with $m^2-(sr-3t)m+2>0$ satisfies 
$m>(sr-3t)/2=(sr-3\sqrt{sr^2+1})/2>(sr-3\sqrt{2sr^2})/2=r(s-3\sqrt{2s})/2$.
\newline\qedsymbol\vskip\baselineskip}

\subsection{Zariski Decompositions}

The existence of Zariski decompositions was first proved for effective divisors \cite{2refZ2} on any smooth projective surface $X$. See \cite{2refB}
for a simplified proof. A more general version can
be found in \cite{2refF2}. Here we prove they exist for any effective divisor $D$ on a blow up $X$ of the plane. It is not hard to see that
it actually is enough to assume $D$ is semi-effective (i.e., $tD$ is linearly equivalent to an effective divisor for some $t>0$).

\begin{theorem}\label{ZDthm}
Let $X$ be the blow up of a finite set of points of the plane.
If $D=m_1N_1+\cdots+m_rN_r$ where the $m_i$ are positive integers
and each $N_i$ is a reduced irreducible curve on $X$, then 
we can write $D=P+N$ where $P=a_1N_1+\cdots+a_rN_r$ is nef, the $a_i$ are nonnegative rationals,
$P\cdot N=0$ and either $N=0$ or $N=b_{i_1}N_{i_1}+\cdots+b_{i_j}N_{i_s}$ with the $b_{i_j}$
positive rationals and the matrix $(N_{i_j}\cdot N_{i_k})$ negative definite.
Moreover, if $D'$ is effective with Zariski decomposition $P'+N'$, and linearly equivalent to $D$,
then $P'$ and $P$ are linearly equivalent and $N'=N$.
\end{theorem}

\begin{exercise}\label{LinIndep}
Let $X$ be the blow up of $r$ points of the plane.
If $N_1,\ldots,N_s$ are prime divisors such that the matrix $(N_i\cdot N_j)$ is negative definite,
then the $N_i$ are linearly independent in the divisor class group, and
$s\leq r$.
\end{exercise}

\SolnEater{\vskip\baselineskip 
\noindent{\it Details}: 
If they are not linearly independent, then there are integers $n_i$ not all 0 such that
$\sum_in_iN_i\sim 0$, hence $N^2=0$ but $N^2=\sum_{ij} (N_i\cdot N_j)n_in_j<0$,
where the inequality is because of negative definiteness.
Now note that the divisor class group of $X$ has index $(1,-r)$, hence the biggest 
negative definite subspace has dimension $r$; i.e., $r\geq s$.
\newline\qedsymbol\vskip\baselineskip}

\begin{exercise}\label{6star}
Let $X$ be the blow up of the $r=\binom{6}{2}=15$ points of intersection of 6 general lines in the plane. 
Let $D$ be the sum of the proper transforms of the 6 lines
(so up to linear equivalence $D\sim 6L-2E_1-\cdots-2E_{15}$).
Let $L$ be the proper transform of a general line.
We show how to find a Zariski decomposition for each of the following divisors:
$D_3=D/2\sim 3L-E_1-\cdots-E_{15}$, 
$D_4=L+D/2\sim 4L-E_1-\cdots-E_{15}$, 
$D_5=2L+D/2\sim 5L-E_1-\cdots-E_{15}$,
$D_6=D\sim 6L-2E_1-\cdots-2E_{15}$, and
$D_7=4L+D\sim 10L-2E_1-\cdots-2E_{15}$.
\end{exercise}

\SolnEater{\vskip\baselineskip 
\noindent{\it Details}: 
Since $2D_3=D$ is the sum of the proper transforms
of the $6$ lines, and since these proper transforms are orthogonal,
the intersection matrix of the sum of the components of $D$ is negative definite.
I.e., $D_3=D/2$ is the Zariski decomposition of $D_3$.

Let $H$ be the proper transform of one of the lines.
For $D_4$, look at $2D_4=2L+D=(2L+aD)+(1-a)D$. Since $2L+aD$ must be nef,
so $0\leq(2L+aD)H=2-4a$ so $a\leq(1/2)$. In fact $2L+D/2$ is nef (since
$4L+D$ meets its components nonnegatively) and $D/2$ has negative definite
intersection matrix (as we saw) and $(2L+D/2)(D/2)=0$, so
$D_4=(L+D/2)+(D/2)$ is the Zariski decomposition.

Since $D_5$ is already nef ($2D_5=D_7$ and meets its components nonnegatively), 
$D_5$ is its own Zariski decomposition.

 And $D_6=2D_3=D$ is its own Zariski decomposition, as is $D_7=2D_5$. 
\newline\qedsymbol\vskip\baselineskip}

For the proof of Theorem \ref{ZDthm} we will use a lemma and some examples. 

\begin{exercise}\label{SortOfGramSchmidt}
Let $N_1,\ldots,N_r$ be distinct reduced irreducible curves with $N_i^2<0$ for all $i$
such that no nonzero nonnegative sum $m_1N_1+\cdots+m_rN_r$ is nef.
Then the $N_i$ are linearly independent in the ${\mathbb Q}$-span of $N_1,\ldots,N_r$.
This is because there is an orthogonal basis $N_1^*,\ldots,N_r^*$ for the ${\mathbb Q}$-span of $N_1,\ldots,N_r$.
This basis has the property that $N_1^*=N_1$, $N_2^*=c_{21}N_1^*+N_2$, $N_3^*=c_{31}N_1^*+c_{32}N_2^*+N_3$,
$\ldots$, $N_r^*=c_{r1}N_1^*+c_{r2}N_2^*+\cdots+c_{r,r-1}N_{r-1}^*+N_r$
with each $c_{ij}$ rational and $c_{ij}\geq 0$ (so each $N_i^*$ is a 
nonnegative rational linear combination of the $N_j$) and $(N_i^*)^2<0$ for each $i$. 
\end{exercise}

\SolnEater{\vskip\baselineskip 
\noindent{\it Details}: 
Use Gram-Schmidt orthogonalization without the normalization. I.e.,
define $N_1^*=N_1$, and $N_{i+1}^*=N_{i+1}+\sum_{1\leq j \leq i}\frac{N_{i+1}\cdot N_j^*}{|N_j^*\cdot N_j^*|}N_j^*$
for $1\leq i<r$, so each $c_{i+1,j}=\frac{N_{i+1}\cdot N_j^*}{|N_j^*\cdot N_j^*|}$ is nonnegative and rational,  
the $N_1^*,\ldots,N_r^*$ are orthogonal and each $N_i^*$ is a nonnegative rational linear combination of $N_1,\ldots, N_i$.
Moreover, $N_i^*$ is orthogonal to $N_1,\ldots, N_{i-1}$. Thus $(N_i^*)^2=N_i^*\cdot N_i$, so $N_i^*$ would be nef
if $N_i^*\cdot N_i\geq 0$, hence we must have $N_i^*\cdot N_i<0$, so $(N_i^*)^2<0$.
\newline\qedsymbol\vskip\baselineskip}

\begin{lemma}\label{CriterionNegDef}
Let $N_1,\ldots,N_r$ be reduced irreducible curves.
Then the matrix $(N_i\cdot N_j)$ is negative definite if and only if 
no nonzero nonnegative sum $m_1N_1+\cdots+m_rN_r$ is nef.
\end{lemma}

\begin{proof}
Assume the matrix $(N_i\cdot N_j)$ is negative definite. Thus for any 
nonzero nonnegative linear combination $N=m_1N_1+\cdots+m_rN_r$ we have
$N^2<0$ and hence $N$ is not nef.

Conversely, assume no nonzero nonnegative sum $m_1N_1+\cdots+m_rN_r$ is nef.
Thus $N_i^2<0$ for all $i$. Now apply \Exercise\ \ref{SortOfGramSchmidt}.
Thus the span of $N_1,\ldots,N_r$ has an orthogonal basis where each basis
element has negative self-intersection, hence $(N_i\cdot N_j)$ is negative definite.
\end{proof}

\begin{exercise}\label{dualBasis2Ex}
Let $V$ be a finite dimensional vector space with a positive definite inner product.
Let $v_1,\ldots,v_r$ be a basis such that $v_iv_j\leq0$ for all $i\neq j$.
If $v\in V$ has $vv_i\geq 0$ for all $i$, then $v=a_1v_1+\cdots+a_rv_r$
where $a_i\geq0$ for all $i$.
\end{exercise}

\SolnEater{\vskip\baselineskip 
\noindent{\it Details}: 
Induct on $r$. This is clearly true for $r=1$.
Now assume $r>1$. Let $w$ be orthogonal to $v_1,\ldots,v_{r-1}$ with $wv_r>0$.
Let $p$ be the orthogonal projection of $v$ into the span of $v_1,\ldots,v_{r-1}$.
Then $vv_i=pv_i$ for all $i<r$, so $p$ is a nonnegative linear combination of the $v_i$, $i<r$, and
likewise, so is the orthogonal projection of $-v_r$ (hence $w$ is a nonnegative 
linear combination of $v_1,\ldots,v_r$). But $p=v-cw$ for some nonnegative $c$,
so $v=p+cw$ is a nonnegative linear combination of the $v_i$.
\newline\qedsymbol\vskip\baselineskip}

\begin{corollary}\label{dualBasis2Cor}
Let $N_1,\ldots,N_r$ be reduced irreducible curves with $N_i^2<0$ for all $i$
such that no nonzero nonnegative sum $m_1N_1+\cdots+m_rN_r$ is nef.
Then there is a dual basis $N_1',\ldots,N_r'$ where: 
$N_i'N_j=0$ for all $i\neq j$; 
$N_i'N_i=(N_i')^2<0$ for all $i$; and
each $N_i'$ is a nonnegative rational linear combination of the $N_j$.
\end{corollary}

\begin{proof}
By Lemma \ref{CriterionNegDef}, the intersection form on the span of 
$N_1,\ldots,N_r$ is negative definite.
The dual basis elements $N_i'$ are solutions to the linear equations $N_i'N_j=0$,
which are defined over the integers,
so the solutions are rational linear combinations of the $N_j$. 
Negative definiteness gives $(N_i')^2<0$, and $N_i'N_i=(N_i')^2$
comes down to a choice of scaling.
The fact that each $N_i'$ is a {\it nonnegative} rational linear combination of the $N_j$
comes from \Exercise\ \ref{dualBasis2Ex} (after converting the result to the negative definite case).
\end{proof}

\begin{proof}[Proof of Theorem \ref{ZDthm}]
Start with $D=M+N$, where $M=0$ and $N=m_1N_1+\cdots+m_rN_r$.
If some nonzero nonnegative sum $S=n_1N_1+\cdots+n_rN_r$ is nef,
let $c$ be the minimum of the ratios $m_i/n_i$ for which $n_i>0$.
Replace $M$ by $M+cS$ and replace $N$ by $N-cS$. Then $D=M+N$ and $M$ and $N$ are still
nonnegative sums of the $N_i$, with $M$ still nef but $N$ having one fewer summand.
Repeat this process until either $N=0$ or $N$ is a sum 
$N=b_{i_1}N_{i_1}+\cdots+b_{i_j}N_{i_s}$ 
such that $b_{i_j}>0$ for all $j$ but no nonnegative sum of the $N_{i_j}$ is nef.

Thus we have $D=M+N$ where $M$ is a nef nonnegative rational sum of the curves $N_i$,
and $N$ is either 0 (and we are done) or a positive rational sum
$N=b_{i_1}N_{i_1}+\cdots+b_{i_j}N_{i_s}$ 
where no nonzero nonnegative sum of the $N_{i_j}$ is nef.

In the latter case, if $MN_{i_j}=0$ for all $i$ we take $P=M$ and $N$ as is, and we are done.
So suppose $MN_{i_j}>0$ for some $i$. 
Consider the dual basis $\{N_{j_k}'\}$ given in Corollary \ref{dualBasis2Cor}. 
We can write $N_{i_j}'=\sum_ja_{i_j}N_{i_j}$ with nonnegative rational $a_{i_j}$. 
Choose the maximum $t$ such that $ta_{i_j}\leq b_{i_j}$ for all $j$ and such that
$(M+tN_{i_j}')N_{i_j}\geq0$, and replace $M$ by $M+tN_{i_j}'$ and $N$ by $N-tN_{i_j}'$.
Then either the number of basis elements $N_{j_k}$ meeting $M$ positively has gone down by 1
or the number of terms in $N$ has gone down by 1. Repeating this process
eventually gives a $P=M$ orthogonal to all terms (if any) of $N$.

Moreover, if $D'$ is effective with Zariski decomposition $P'+N'$, and linearly equivalent to $D$,
then $P'$ and $P$ are linearly equivalent and $N'=N$.

For the uniqueness assertion, pick an integer $t>0$ such that
$tP$, $tP'$, $tN$ and $tN'$ are all integral. Then 
some component $C_1$ of $tN$ has $C_1\cdot tN<0$, so $C_1\cdot tN'\leq C_1\cdot tD'=C_1\cdot tD=C_1\cdot tN<0$.
Thus $C_1$ is a component of $tN'$. If $tN\neq C_1$, then for some component $C_2$ of $tN-C_1$, by
negative definiteness we have
$C_2\cdot (tN'-C_1)\leq C_12\cdot (tN-C_1)<0$. Repeating this, we eventually see that
$tN'-tN$ is effective. Reversing the argument shows that $tN'-tN$ is also effective, so $tN'=tN$.
Thus $tP'=tD'-tN'$ is linearly equivalent to $tP=tD-tN$, as claimed.
\end{proof}

Computing $\widehat{\alpha}(I)$ can sometimes come from computing Zariski decompositions.

\begin{proposition}\label{ZarDecProp}
Let $p_1,\cdots,p_r$ be distinct points in the plane and let $I$ be the radical ideal of the points.
Let $X$ be the surface obtained by blowing up of the points and let
$F=dL-m_1E_1-\cdots-m_rE_r$. If 
$F$ has a Zariski decomposition of the form
$P+N$ where $P\neq0$ and $N=aL-b(E_1+\cdots+E_r)\neq0$,
then $\widehat{\alpha}(I)=\frac{a}{b}$.
\end{proposition}

\begin{proof}
Since $N$ is effective we have $\widehat{\alpha}(I)\leq\frac{a}{b}$.
Let $E=E_1+\cdots+E_r$.
Since $P$ is nef, we have $(PL)b(\widehat{\alpha}(I))-bPE=P(b(\widehat{\alpha}(I))L-bE)\geq0=PN=(PL)a-bPE$,
so $b(\widehat{\alpha}(I))\geq a$ or $\widehat{\alpha}(I))\geq a/b$.
\end{proof}

\begin{example}\label{AlmostCollinear}
We demonstrate Proposition \ref{ZarDecProp} by computing $\widehat{\alpha}(I)$ for the ideal $I$ of 
$n\geq2$ points on a line and one point off. Note that these points are the points of intersection
of $n+1$ lines; the case of $n=3$ is shown in Figure \ref{4and3}.
Let $p_0$ be the point off the line, $p_1,\ldots,p_n$ the collinear points.
Take $E=E_1+\cdots+E_n$ and $F=(3n-1)L-(2n-1)E_0-(n+1)E$. Its Zariski decomposition is
$P=nL-(n-1)E_0-E$ and $N=(2n-1)L-nE_0-nE$, so $N$ is the sum of the proper transforms of the lines through $p_0$
and $n-1$ times the proper transform of the line through the other $n$ points.
Thus $\widehat{\alpha}(I)=(2n-1)/n$.
\end{example}

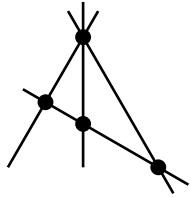
\begin{figure}[htbp]
\begin{center}
\begin{tikzpicture}[x=1.0cm,y=1.0cm]
\draw [line width=1,color=black] (0,0)-- (1.2,2.0784);
\draw [line width=1,color=black] (.8,2.0784)-- (2.2,-.3464);
\draw [line width=1,color=black] (1,0)-- (1,2.2);
\draw [line width=1,color=black] (2.4,-.230854)-- (.2,1.0386);
\fill (2,0) circle (3 pt);
\fill (1,.5773) circle (3 pt);
\fill (.5,.866) circle (3 pt);
\fill (1,1.732) circle (3 pt);
\end{tikzpicture}
\caption{A configuration of four lines with a triple point.}
\label{4and3}
\end{center}
\end{figure}

\begin{exercise}
Here we demonstrate Proposition \ref{ZarDecProp} by computing $\widehat{\alpha}(I)$ for the ideal $I$ 
of the points of intersection of the lines in Figure \ref{TriplePtand7doublePts}.
Let $p_1$ be the triple point and $p_2,\ldots,p_7$ the other six points on lines through the triple point,
and let $p_8$ be the remaining point.
Let $E=E_2+\cdots+E_6$ and take $F=10L-4E_1-4E-4E_8$.
Its Zariski decomposition is $P=3L-E_1-E$ and $N=7L-3E_1-3E-3E_8$,
so $N$ is the sum of the proper transforms
of the lines through the triple point plus twice the sum of proper transforms of the other two lines plus $E_8$,
hence $\widehat{\alpha}(I)=7/3$.
\end{exercise}

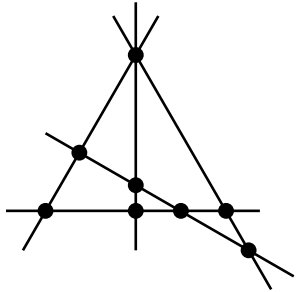
\begin{figure}[htbp]
\begin{center}
\begin{tikzpicture}[x=1.5cm,y=1.5cm]
\draw [line width=1,color=black] (0,0)-- (1.2,2.0784);
\draw [line width=1,color=black] (.8,2.0784)-- (2.2,-.3464);
\draw [line width=1,color=black] (-.15,.35)-- (2.1,.35);
\draw [line width=1,color=black] (1,0)-- (1,2.2);
\draw [line width=1,color=black] (2.4,-.230854)-- (.2,1.0386);
\fill (1.8,.35) circle (3 pt);
\fill (1.4,.35) circle (3 pt);
\fill (1,.35) circle (3 pt);
\fill (.2,.35) circle (3 pt);
\fill (2,0) circle (3 pt);
\fill (1,.5773) circle (3 pt);
\fill (.5,.866) circle (3 pt);
\fill (1,1.732) circle (3 pt);
\end{tikzpicture}
\caption{A configuration of five lines with a triple point.}
\label{TriplePtand7doublePts}
\end{center}
\end{figure}

\begin{exercise}\label{WCforStars}
This time we demonstrate Proposition \ref{ZarDecProp} by computing $\widehat{\alpha}(I)$ for the ideal $I$ of the points of intersection of $d>2$ general lines in the plane.
Here blow up the $r=\binom{d}{2}$ points and let $E=E_1+\cdots+E_r$.
Take $F=(3d-2)L-4E$, $P=2(d-1)L-2E$ and $N=dL-2E$, so $N$ is the sum of the proper transforms of the $d$ lines.
Thus $\widehat{\alpha}(I)=d/2$. 
\end{exercise}

\section{Bounded Negativity Conjecture (BNC) and $H$-constants}

\subsection{Bounded Negativity}
Let $X$ be a smooth projective surface. If $C$ is a curve on $X$,
how negative can $C^2$ be? This certainly depends on $X$. For example,
for $X=\P^2$ we have $C^2>0$ for all $C$.

\begin{exercise}\label{blowupPtsOnLine}
Let $X\to \P^2$ be the blow up of $n\geq 2$ distinct points $p_1,\ldots,p_n$ on a line $L\subseteq \P^2$.
Let $L$ be the total transform of a line and $E_i$ the blow up of $p_i$.
Consider the divisor $F=dL-m_1E_1-\cdots-m_nE_n$ on $X$.
\begin{enumerate}
\item[(a)] Then $|F|$ is nonempty if and only if $d\geq\max(m_1,\ldots,m_n,0)$.
\item[(b)] If $D$ is a divisor on $X$ such that $D\cdot E_i\geq0$ for all $i$ and
$D\cdot H\geq0$ where $H$ is the proper transform of $L$
(so $H\sim L-E_1-\cdots-E_n$), then $D^2\geq 0$, and one can 
conclude that the only reduced irreducible curves $C$ on $X$ with $C^2<0$
are $E_1,\ldots, E_n$ and $H$.
\item[(c)] Let $C$ be an effective divisor and let $m$ be the multiplicity
of the irreducible component of $C$ of maximum multiplicity. Then $C^2\geq -m^2n$
and curves $C$ exist such that equality holds.
(Note: Write $C=P+N$, where $P$ is the sum of the components of $C$ with nonnegative self-intersection,
and $N$ is the sum of the irreducible components of $C$ of negative self-intersection,
hence $N=a_0H+a_1E_1+\cdots+a_nE_n$ for some $a_i\geq 0$, so
$N^2\geq -\sum_i(a_i-a_0)^2$. Conclude that $C^2\geq N^2 \geq -m^2n$. )
\end{enumerate}
\end{exercise}

\SolnEater{\vskip\baselineskip 
\noindent{\it Details}: (a) If $d<\max(m_1,\ldots,m_n,0)$, then either $d<0$ (hence $|F|$
is empty), or some $m_i$ is positive but $d<m_i$, hence again $|F|$ is empty.
Now let $m_i'=\max\{m_i,0\}$. 
If $d\geq\max(m_1,\ldots,m_n,0)$, then $\ell^d\in [I(\sum_im_i')]_d$ 
(where $\ell$ is the linear form defining $L$), so $|dL-\sum_im_i'E_i|$
is nonempty. Let $G\in |dL-\sum_im_i'E_i|$. Then $G-\sum_{m_i<0}m_iE_i\in |F|$,
so $|F|$ is nonempty.

(b) Since $D\cdot E_i\geq 0$, we have $D=dL-\sum_im_iE_i$ for $m_i\geq0$.
Since $H\cdot D\geq 0$ we have $d\geq \sum_im_i$, hence $D^2=d^2-\sum_im_i^2=(\sum_im_i)^2-\sum_im_i^2\geq 0$.
Thus a reduced irreducible curve $C$ with $C^2<0$ must either have $C\cdot E_i<0$ for some $i$
(hence $C=E_i$) or $C\cdot H<0$ (hence $C=H$).

(c) Using the note we have $C^2\geq N^2 \geq -\sum_i(a_i-a_0)^2\geq -nm^2$.
Taking $C=m(E_1+\cdots+E_n)$ gives $C^2=-m^2n$.
\newline\qedsymbol\vskip\baselineskip}

This brings us to the Bounded Negativity Conjecture (BNC), an old still open folklore conjecture
that goes back at least to F.\ Enriques. (There seems only to be oral evidence of its provenance, however.
I heard this conjecture from my advisor, M.~Artin, around 1980.
C.~Ciliberto heard this conjecture from his advisor, A.~Franchetta.
Franchetta was Enriques's last student, who told Ciliberto that
he had heard it from Enriques; see \cite{3refBNC}.)

There are various versions of the BNC. Here's one.

\begin{conjecture}\label{BNC1}
Let $X$ be a smooth projective surface, either rational or complex (i.e., 
either $X$ is a rational surface and the ground field is an arbitrary algebraically closed field,
or $X$ is any smooth projective surface defined over ${\mathbb C}$).
Then there is a bound $B_X$ such that for any effective divisor $C$ on $X$,
we have $C^2/m^2\geq B_X$, as long as $m$ is a positive integer 
at least as big the multiplicity of every component of $C$.
\end{conjecture}

Here's another.

\begin{conjecture}\label{BNC2}
Let $X$ be a smooth projective surface, either rational or complex.
Then there is a bound $B_X$ such that for any effective reduced divisor $C$ on $X$,
we have $C^2\geq B_X$.
\end{conjecture}

And one more:

\begin{conjecture}\label{BNC3}
Let $X$ be a smooth projective surface, either rational or complex.
Then there is a bound $b_X$ such that for any effective reduced irreducible divisor $C$ on $X$,
we have $C^2\geq b_X$.
\end{conjecture}

Over fields of positive characteristic bounded negativity can fail; see \cite[Exercise V.1.10]{3refHart}.
But no counterexamples are known for rational surfaces in any characteristic or for smooth
complex projective surfaces.

All three versions of the BNC given above are equivalent.
For the equivalence of the second and third, see \cite[Proposition 5.1]{3refBNC}.
The method of proof is to apply Zariski decompositions.
In the following theorem statement, $\rho(X)$ is the Picard number for $X$ (i.e., the rank of the N\'eron-Severi group).

\begin{theorem}\label{BNC2and3}
Conjecture \ref{BNC2} holds for $X$ if and only if Conjecture \ref{BNC3} holds for $X$.
Moreover, given a bound $b_X<0$ for the latter, we can always take $B_X\leq (\rho(X)-1)b_X$ for the 
bound in the former.
\end{theorem}

\begin{proof}
Certainly, if self-intersections of reduced curves are bounded below, then so are the 
self-intersections of irreducible curves on $X$.
Conversely, let $D$ be any reduced effective divisor. Write
$D=C_1+\cdots+C_r$ where the $C_i$ are distinct reduced, irreducible curves.
Let $D=P+N$ be a Zariski decomposition (see Theorem \ref{ZDthm}), so $N=n_1N_1+\cdots+n_sN_s$ with
$0\leq n_i\leq 1$ rational for all $i$ and the $N_i$ prime divisors of negative self-intersection. 
Since the intersection matrix for $N$ is negative definite, we have
$D^2\geq N^2\geq (n_1N_1)^2+\cdots+(n_sN_s)^2\geq N_1^2+\cdots+N_s^2$,
and since the components $N_i$ are linearly independent by Example \ref{LinIndep} we have
$N_1^2+\cdots+N_s^2\geq (\rho(X)-1)\min_i\{N_i^2\}\geq (\rho(X)-1)b_X$.
\end{proof}

Now we show that the first and second versions are equivalent.

\begin{theorem}\label{BNC1and2}
Conjecture \ref{BNC1} holds for $X$ if and only if Conjecture \ref{BNC2} holds for $X$,
using the same bound $B_X$.
\end{theorem}

\begin{proof}
Clearly, Conjecture \ref{BNC1} implies Conjecture \ref{BNC2}.
Conversely, given an effective divisor $C=m_1C_1+\cdots+m_nC_n$,
we have $C^2/m^2\geq (C_{i_1}+\cdots+C_{i_r})^2\geq B_X$
for some subset of components $C_{i_j}$, by Lemma \ref{nonreducedLemma},
where $m$ is the maximum of the $m_i$.
\end{proof}

\begin{lemma}\label{nonreducedLemma}
Let $X$ be a smooth projective surface.
Let $C=m_1C_1+\cdots+m_nC_n$ for distinct reduced irreducible curves $C_i$
on $X$ and integers $m\geq m_i>0$ with $m=\max(m_1,\ldots,m_n)$.
Then for some nonempty subset $C_{i_1},\ldots,C_{i_r}$ of the components $C_i$ we have
$C^2\geq m^2(C_{i_1}+\cdots+C_{i_r})^2$.
\end{lemma}

\begin{proof}
If $C\cdot C_i\geq 0$ for all $i$, we may assume that $m=m_n$, and then $C^2\geq m^2C_n^2$, so
assume that $C\cdot C_i< 0$ for some $i$. Let $P=\sum_{C\cdot C_i\geq0}m_iC_i$ and $N=\sum_{C\cdot C_j<0}m_jC_j$. 
Note that $PN\geq 0$ since $P$ and $N$ have no components in common.
Then $C^2=CP+CN\geq CN =PN + N^2\geq N^2$. Note that $N\cdot C_j<0$ for each $C_j$ that appears in $N$.

It now is enough to prove the claim for $N$, so we are reduced to the case that
$C=m_1C_1+\cdots+m_nC_n$ with $C\cdot C_i<0$ for all $i$ and $m\geq m_i$ for all $i$.
We have $0>C\cdot m_iC_i\geq C\cdot mC_i$, hence $0>C^2=C\cdot \sum_i m_iC_i\geq C\cdot \sum_i mC_i=mC\cdot \sum_i C_i$.
Now write $C=P+N$ where now $P$ is the sum of the terms $m_jC_j$ in $C$ such that $C_j\cdot \sum_i C_i\geq0$
and $N$ is the sum of those terms $m_jC_j$ with $C_j\cdot \sum_i C_i<0$. 
Let $Q$ be the same as $N$ except where the coefficient $m_j$ of $C_j$ in each term is replaced by 1.
Then $0>C\cdot \sum_i C_i=(P+N)\cdot\sum_i C_i\geq N\cdot \sum_i C_i\geq mQ\cdot  \sum_i C_i \geq mQ^2$.
Thus $C^2\geq C\cdot m\sum_i C_i\geq m^2Q^2$.
\end{proof}

\begin{example}\label{nonreducedExample}
Given an effective divisor $C=m_1C_1+\cdots+m_nC_n$ it's clear 
in general that $C^2\geq m^2(C_1+\cdots+C_n)^2$ is false,
when $m$ is the maximum of the $m_i$. Take $C=L_1+2L_2$
for two different lines $L_i$ in the plane. Then $C^2=9$, but
$2^2(L_1+L_2)^2=16$. However,
in the proof of Lemma \ref{nonreducedLemma}, we reduce to the case that
$C=m_1C_1+\cdots+m_nC_n$ with $C\cdot C_i<0$ for all $i$.
One might hope in this case that $C^2\geq m^2(C_1+\cdots+C_n)^2$, but alas no.
Blow up the 11 points shown in Figure \ref{ConicFigure} and let 
$A$ and $B$ be the proper transforms of $A'$ and $B'$.
Then $(A+2B)^2=A^2+4AB+4B^2=-6<-4=2^2(A+B)^2$.
However we do have $(A+2B)^2\geq 2^2B^2$.

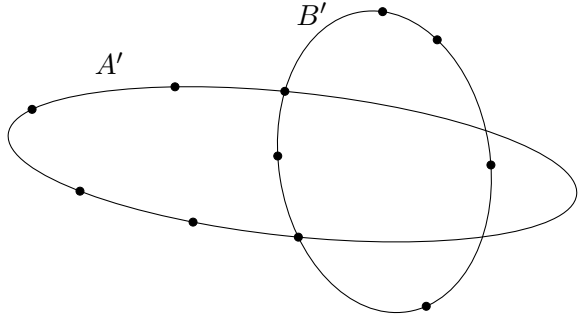
\begin{figure}[htbp]
\begin{center}
\definecolor{xdxdff}{rgb}{0,0,0}
\definecolor{qqqqff}{rgb}{0,0,0}
\begin{tikzpicture}[line cap=round,line join=round,>=triangle 45,x=1.0cm,y=1.0cm]
\clip(-4.3,.5) rectangle (7.3,5.2);
\draw [rotate around={-6.078404330331954:(0.32144191980215925,2.8889288784633465)}] (0.32144191980215925,2.8889288784633465) ellipse (3.7996388686780844cm and 0.9562451906165926cm);
\draw [rotate around={-81.13232785816251:(1.5439407511733567,2.9229148849221755)}] (1.5439407511733567,2.9229148849221755) ellipse (2.0186781860047476cm and 1.404276482564587cm);
\draw (-2.44,4.5200000000000005) node[anchor=north west] {$A'$};
\draw (0.24,5.180000000000001) node[anchor=north west] {$B'$};
\begin{scriptsize}
\draw [fill=qqqqff] (-3.14,3.62) circle (1.5pt);
\draw [fill=qqqqff] (-1.24,3.92) circle (1.5pt);
\draw [fill=qqqqff] (0.22,3.86) circle (1.5pt);
\draw [fill=qqqqff] (0.4,1.92) circle (1.5pt);
\draw [fill=qqqqff] (-1.0,2.12) circle (1.5pt);
\draw [fill=qqqqff] (1.52,4.92) circle (1.5pt);
\draw [fill=qqqqff] (2.96,2.88) circle (1.5pt);
\draw [fill=qqqqff] (2.1,1.0) circle (1.5pt);
\draw [fill=xdxdff] (2.2460053338712918,4.5453368542140495) circle (1.5pt);
\draw [fill=xdxdff] (-2.5031928667690013,2.5317226239881028) circle (1.5pt);
\draw [fill=xdxdff] (0.12592136284642663,2.9992571821416436) circle (1.5pt);
\end{scriptsize}
\end{tikzpicture}
\caption{Two conics, $A'$ and $B'$, in the plane, one through 6 points, the other through 7 points,
giving 11 points with 2 in common.}
\label{ConicFigure}
\end{center}
\end{figure}
\end{example}

\subsection{$H$-constants.}
Given the longstanding difficulty of resolving BNC, it is worth considering variations on the problem,
such as the problem of $H$-constants. A number of different versions have been defined
\cite{3refIMRN,
3refDHS,
3refP,
3refSz}.
Here we define them for any curve (typically they have been defined for reduced curves).

\begin{definition}\label{Hconst}
Let $C_1,\ldots,C_r$ be distinct reduced irreducible plane curves and let $C=m_1C_1+\cdots+m_rC_r$ 
where $m_i>0$ are integers with $m=\max(m_1,\ldots,m_r)$.
Then for any nonempty finite subset $S\subseteq \P^2_\KK$ we define
$$H(C,S)=\frac{d^2-\sum_{p\in S}(\mult_pC)^2}{m^2|S|},$$
where $d=\deg(C)$.
We also define
$$H(C)=\inf\Big\{H(C,S): S\subseteq \P^2, 0 < |S| <\infty\Big\},$$
$$H_{red}(\P^2_\KK)=\inf\Big\{H(D): D\hbox{ is a reduced curve in }\P^2_\KK\Big\},$$
$$H_{rir}(\P^2_\KK)=\inf\Big\{H(D): D\hbox{ is a reduced, irreducible curve in }\P^2_\KK\Big\}$$
and
$$H(\P^2_\KK)=\inf\Big\{H(D): D\hbox{ is a curve in }\P^2_\KK\Big\}$$
\end{definition}
(Clearly $H(\P^2_\KK)\leq H_{red}(\P^2_\KK)\leq H_{rir}(\P^2_\KK)$.)

\begin{exercise}\label{CvP2}
Let $C$ be a plane curve. Then
$$\inf\Big\{H(C,S): S\subseteq C, 0 < |S| <\infty\Big\}=\inf\Big\{H(C,S): S\subseteq \P^2, 0 < |S| <\infty\Big\}.$$
\end{exercise}

\SolnEater{\vskip\baselineskip 
\noindent{\it Details}: Clearly $\inf\Big\{H(C,S): S\subseteq C, 0 < |S| <\infty\Big\}\geq \inf\Big\{H(C,S): S\subseteq \P^2, 0 < |S| <\infty\Big\}$.
By taking $S$ to contain lots of points on the curve, we see both infimums are negative.
If $H(C,S)<0$ but $S$ also contains points off the curve, removing those points from $S$ decreases
the denominator of $H(C,S)$ but leaves the numerator the same, hence gives a more negative ratio.
Thus $\inf\Big\{H(C,S): S\subseteq C, 0 < |S| <\infty\Big\}\leq\inf\Big\{H(C,S): S\subseteq \P^2, 0 < |S| <\infty\Big\}.$
\newline\qedsymbol\vskip\baselineskip}

\begin{theorem}\label{HconstBNCthm}
If $H_{rir}(\P^2_\KK)>-\infty$, then Conjecture \ref{BNC1} holds for every smooth projective 
rational surface $X$ over the field $\KK$.
\end{theorem}

\begin{proof}
Consider a birational morphism $Y\to X$ of smooth projective surfaces.
Let $C'$ be a reduced, irreducible curve on $X$ and $C$ its proper transform on $Y$.
Then $C^2\leq (C')^2$. Thus if Conjecture \ref{BNC3} holds for $Y$, it also holds for $X$.
However, if $X$ is rational, it is a blow up of points (possibly infinitely near) 
on a Hirzebruch surface $H_n$ for some $n$. By blowing up $n$ general points 
of $H_n$, we obtain a surface that is also obtained by blowing up distinct
points of $\P^2$. (Note, for example, that by blowing up $n$ points $p_i$ on a line in $\P^2$
and any point $p$ off that line, we get a surface $B$ which by contracting  
the proper transforms of the lines through $p$ and each $p_i$ gives
a birational morphism $B\to H_n$.) Thus by blowing up $n$ general 
points of $X$ we get a birational morphism $Y\to X$, where there is also 
a birational morphism $Y\to \P^2$
obtained by blowing up a finite set $S$ of distinct points of $\P^2$. 

If Conjecture \ref{BNC3} did not hold for $Y$, there would be an infinite
sequence $C_1, C_2, \ldots$ of reduced, irreducible curves on $Y$
such that $C_1^2> C_2^2 > \cdots$. In all but finitely many cases,
$C_i$ maps to a plane curve $D_i$ under $Y\to \P^2$,
and so $C_i$ is the proper transform of $D_i$, hence
we have $C_i^2/|S|=(\deg(D_i)^2 - \sum_{p\in S}(\mult_pD_i)^2)/|S|=H(D_i,S)$,
which implies $H_{rir}(\P^2)=-\infty$. Thus $H_{rir}(\P^2_\KK)>-\infty$ implies Conjecture \ref{BNC3}
which in turn implies Conjecture \ref{BNC1}. 
\end{proof}

\begin{example}\label{HCharp}
In fact $H_{red}(\P^2_\KK)=-\infty$ if $\chr(\KK)=p>0$. Let $C$ be the union of all of the lines in $\P^2$
defined over a finite field ${\mathbb F}_q\subseteq\KK$ of order $q$.
There are $q^2+q+1$ such lines with $q^2+q+1$ crossing points,
and each point lies on $q+1$ lines. Let $S$ be the points. Then $H(C, S)=-q$,
so $H_{red}(\P^2_\KK)=-\infty$.
\end{example}

\begin{openproblem}\label{Hred}
Is $H(\P^2_\KK)=H_{red}(\P^2_\KK)$ true for all $\KK$?
\end{openproblem}

\begin{openproblem}\label{Cubics}
Is $H_{red}(\P^2_\C)= -4$? We know $H_{red}(\P^2_\C)\leq -4$
due to sequences $C_n$ of reducible curves whose components are plane cubics
(see \cite{3Roulleau14, 3RU:Chern, 3refBHRT}), but 
no complex plane curve $C$ is known with $H(C)\leq -4$.
Thus it is of interest to find some examples or show that none exist.
\end{openproblem}

\begin{openproblem}\label{Hrir}
Is $H_{rir}(\P^2_\KK)=-2$?
In fact, there is no reduced irreducible plane curve $C$ known over any $\KK$
with $H(C)\leq -2$. 
\end{openproblem}

\begin{exercise}\label{Hirr-2}
One can show that $H_{rir}(\P^2_\KK)\leq-2$ over any $\KK$ by giving a sequence of 
reduced, irreducible curves $C_n$ with $\lim_{n\to\infty}H(C_n)=-2$.
\end{exercise}

\SolnEater{\vskip\baselineskip 
\noindent{\it Details}: Take a general map of $\P^1$ into $\P^2$ of degree $n$.
The image is a rational curve $C_n$ of degree $n$ with $\binom{n-1}{2}$ nodes.
By Theorem \ref{singPtsBest}, $H(C_n)=H(C_n,S)$ where $S$ is some subset of
the nodes. Thus $H(C_n)=\frac{n^2-4s}{s}=\frac{n^2}{s}-4$. This is least when 
$s$ is most, so we take $s=\binom{n-1}{2}$ which gives
$H(C_n)=-2+\frac{6n-4}{n^2-3n+2}$, which in the limit gives $-2$.
\newline\qedsymbol\vskip\baselineskip}

\begin{exercise}\label{smoothHconst}
If $C$ is a smooth plane curve of degree $d$, $m\geq 1$ and $S$ any nonempty 
finite subset of $C$, then $H(mC)= -1<H(mC, S)$.
\end{exercise}

\SolnEater{\vskip\baselineskip 
\noindent{\it Details}: We have $H(mC,S)=\frac{m^2d^2-sm^2}{m^2s}=-1+\frac{m^2d^2}{sm^2}$,
where $s=|S|$. Thus the infimum $H(mC)$ over all $S$ is $-1$.
\newline\qedsymbol\vskip\baselineskip}

\begin{exercise}\label{reducedHconst}
If $C$ is a reduced plane curve, $m\geq 1$ and $S$ any nonempty 
finite subset of smooth points of $C$, then $H(mC)\leq -1<H(mC, S)$.
\end{exercise}

\SolnEater{\vskip\baselineskip 
\noindent{\it Details}: Essentially the same solution as for \Exercise\ \ref{smoothHconst} shows
$-1<H(mC, S)$ and the infimum over all such $S$ is $-1$. Thus $H(mC)\leq -1$.
It's possible to have $H(mC)< -1$, since $C$ could have singularities which lower the infimum.
\newline\qedsymbol\vskip\baselineskip}

\begin{exercise}\label{HCbnd}
If $C$ is any plane curve, then $-\infty<H(C)\leq -1$.\\
(Note: One can show $\min\{-\max\{m_1^2,\ldots,m_n^2,0\},-1\}\leq H(C)$, where
the $m_i$ are the multiplicities, if any, of the singular points of the
reduced curve $\operatorname{red}(C)$.)
\end{exercise}

\SolnEater{\vskip\baselineskip 
\noindent{\it Details}: We get $H(C)\leq -1$ by looking at $H(C,S)$ where $S$ is a subset of points
which are smooth points on the component of $\operatorname{red}(C)$
which occurs with maximum multiplicity in $C$. 
Let $d=\deg(C)$, $\mu=\max\{m_1^2,\ldots,m_n^2\}$ and let $m$ be the multiplicity
of the component of $C$ of maximum multiplicity. Given a finite set $S\subseteq\P^2$, 
let $a$ be the number of points of $S$ which are singular points of 
$\operatorname{red}(C)$, and let $b$ be the number of remaining points of $S$.
Then $H(C,S)\geq \frac{d^2-am^2\mu^2-bm^2}{(a+b)m^2}$. If $a=0$ this is at least $-1$.
If $a\neq0$, this is at least $-\mu^2$. This gives the result.
\newline\qedsymbol\vskip\baselineskip}

\begin{theorem}\label{singPtsBest}
Let $C$ be a reduced singular plane curve of some degree $d$, let $T$ be the set of singular points of $C$.
Then $H(C) < -1$ if and only if $|T|>0$ and $H(C,T)<-1$,
in which case $H(C)=H(C,U)$ for some nonempty subset $U\subseteq T$.
\end{theorem}

\begin{proof}
First, assume $|T|>0$ and $H(C,T)<-1$. Then clearly $H(C)<-1$, since $H(C)$ is an infimum over all finite subsets of $C$. 
Conversely, first assume $|T|=0$. Then $C$ is smooth, so $H(C)=-1$ by \Exercise\ \ref{smoothHconst}.

Next, assume $|T|>0$ but $H(C,T)\geq-1$. Let $|T|=t$ and let $m_1,\ldots,m_t$ be the multiplicities of $C$ at these points.
Let $S$ be a finite set of smooth points of $C$; let $s=|S|$.
Then $H(C,T)=(d^2-\sum_im_i^2)/t \geq -1$, so
$H(C,S\cup T) = (d^2-s-\sum_im_i^2)/(s+t) \geq (-t-s)/(s+t)=-1$.
Also, if $s>0$, then $H(C,S)>-1$ by \Exercise\ \ref{reducedHconst}.

Now assume $t>0$ and $H(C,T)\geq-1$, and let $U\cup V=T$ be a disjoint union of nonempty subsets.
Let $u=|U|$, $v=|V|$ and $m_p$ be the multiplicity of $C$ at a point $p$. Then
$H(C,U)=(d^2-\sum_{p\in U}m_p^2)/u$. If this were less than $-1$, then 
$-1\leq H(C,T)=(d^2-\sum_{p\in U}m_p^2-\sum_{p\in V}m_p^2)/(u+v) < (-u-\sum_{p\in V}m_p^2)/(u+v)\leq (-u-4v)/(u+v) < -1$.
Thus $H(C,U)\geq -1$ for every nonempty subset $U\subseteq T$. Now arguing as before for 
finite any set of smooth points $S$ of $C$ we have $H(C,S\cup U)\geq -1$. Thus $H(C)\geq -1$.

Finally, assume $H(C)<-1$. Thus there are finite subsets $W$ of $C$ with $H(C,W)<-1$.
For any finite subset $S$ of smooth points we saw $H(C,S)>-1$, so $W$ must include points from $T$.
Write $W$ as a disjoint union $W=S\cup U$ where $U\subseteq T$ and the points in $S$ are smooth.
If $H(C,U)\geq -1$, then we saw above that we would have $H(C,W)=H(C,S\cup U)\geq -1$.
Thus $H(C,U)< -1$, and so
$H(C,W)=H(C,S\cup U) = (d^2-s-\sum_{p\in U}m_p^2)/(s+u) = (-s+uH(C,U))/(s+u)>H(C,U)$,
where the last inequality is because $-1>H(C,U)$. Thus the least values of $H$
come from subsets of $T$, but $T$ is finite so the infimum is a minimum, and 
this minimum is attained for a subset of $T$.
\end{proof}

\begin{openproblem}\label{UvTOP}
Is there an example of a singular plane curve $C$ such that $H(C, U)< H(C, T)$
for some nonempty proper subset $U$ of the set $T$ of singular points of $\operatorname{red}(C)$?
\end{openproblem}

\begin{exercise}\label{UvTEx}
If $C$ is a reduced singular plane curve $C$ such that $H(C, U)< H(C, T)$
for some nonempty proper subset $U$ of the set $T$ of singular points of $C$,
then $H(C, T)<-4$.
\end{exercise}

\SolnEater{\vskip\baselineskip 
\noindent{\it Details}: Let $V$ be the complement of $U$ in $T$ and let $u=|U|$, $v=|V|$ and $t=|T|$.
Then $\frac{d^2-\sum_{p\in U}m_p^2}{u}=H(C, U)< H(C, T)=\frac{d^2-\sum_{p\in T}m_p^2}{t}\leq\frac{d^2-\sum_{p\in U}m_p^2-4v}{u+v}$.
Simplifying $\frac{d^2-\sum_{p\in U}m_p^2}{u}<\frac{d^2-\sum_{p\in U}m_p^2-4v}{u+v}$
gives $d^2-\sum_{p\in U}m_p^2<-4u$, hence $H(C, T)<\frac{-4u-\sum_{p\in V}m_p^2}{t}\leq \frac{-4u-4v}{u+v}=-4$.
\newline\qedsymbol\vskip\baselineskip}

Given Open Problem \ref{Hrir}, attention turned to the opposite extreme, curves which are unions of lines 
\cite{3refIMRN, 3refSz}. Here are the main facts (see \cite{3refIMRN}). Define
$$H_{rlin}(\P^2_\KK)=\inf\Big\{H(D): D\hbox{ is a reduced union of lines in }\P^2_\KK\Big\}.$$

We have:
$$-2.6\geq H_{rlin}(\P^2_\Q)\geq-3,$$ 
$$H_{rlin}(\P^2_\R)=-3,$$ 
and 
$$-3.358>-\frac{225}{67}\geq H_{rlin}(\P^2_\C)\geq -4.$$
The bound $-2.6\geq H_{rlin}(\P^2_\Q)$ comes from taking horizontal, vertical and diagonal lines.
The equality $H_{rlin}(\P^2_\R)=-3$ comes from $H_{rlin}(\P^2_\R)\geq-3$ (apply Theorem \ref{MelchiorThm})
and by giving examples $H(C)$ approaching $-3$ (there are lots; e.g., regular polygons with their lines of bilateral symmetry).
The bound $-\frac{225}{67}\geq H_{rlin}(\P^2_\C)$ comes from the Wiman arrangement. 
The bound $H_{rlin}(\P^2_\C)\geq -4$ comes from applying an inequality due to Hirzebruch \cite{3Hir83}:
given any complex arrangement of $n>3$ lines such that $t_n=t_{n-1}=0$, we have
\begin{equation*}\label{eq:Hirzebruch}
t_2+\frac34t_3\geq d+\sum\limits_{k\geq 5}(k-4)t_k.
\end{equation*}


\begin{exercise}\label{BndOnHoverC}
Let $L_1,\ldots,L_d$ be distinct lines in the $\P^2_\KK$. Assume that neither the lines nor any subset of $d-1$ of the lines are concurrent.
Also assume that $t_2=0$.
Let $C$ be the curve given by the union of the lines. Let $S$ be the set of the singular points of $C$, and set $s=|S|$.
\begin{enumerate}
\item[(a)] Then $H(C,S)\leq -2$; examples occur where equality holds.
\item[(b)] If $\KK=\C$, then $d\leq 3s/4$.
\item[(c)] If $\KK=\C$, then $H(C,S)\leq -2.25$; examples occur where equality holds.
\end{enumerate}
\end{exercise}

\SolnEater{\vskip\baselineskip 
\noindent{\it Details}: (a) By \Exercise\ \ref{t_k}(e), we have $s\geq d$. Thus $H(C,S)=\frac{d-\sum_kt_kk}{s}\leq 1 - \frac{3s}{s}=-2$.
Equality holds for the 7 lines of the Fano plane in characteristic 2.

(b) This follows from Hirzebruch's inequality, since $(3/4)s\geq (3/4)t_3\geq d$.

(c) Using (b) we have $H(C,S)=\frac{d-\sum_kt_kk}{s}\leq (3/4) - \frac{3s}{s}=-2.25$.
Equality holds for the Fermat arrangement with $n=3$. 
\newline\qedsymbol\vskip\baselineskip}

\begin{openproblem}\label{Hlines}
Can more be said about $H_{rlin}(\P^2_\Q)$ and $H_{rlin}(\P^2_\C)$?
\end{openproblem}

\subsection{Another formulation of bounded negativity}

Let $X$ be the blow up of the plane at a finite set of points $S$.
We say that $X$ has bounded Zariski denominators if there is an integer
$d$ such that for each 
divisor $D$ and integer $t>0$ such that $tD$ is linearly equivalent to an effective divisor,
there is an integer $0\leq e\leq d$ such that
the Zariski decomposition $etD=P+N$ has integral divisors $P$ and $N$ 
(equivalently, there is an integer $d>0$ such that for each 
divisor $D$ and integer $t>0$ such that $tD$ is linearly equivalent to an effective divisor,
the Zariski decomposition $d!tD=P+N$ has integral divisors $P$ and $N$).

We now state a version of the main theorem of \cite{3refBPS2}.

\begin{thm}\label{BPStheorem}
Let $X$ be the blow up of the plane at a finite set of points $S$.
Then bounded negativity holds on $X$ (i.e.,
the set of self-intersections $C^2$ of reduced curves on $X$ is bounded below)
if and only if $X$ has bounded Zariski denominators. 
\end{thm}

Two examples will be helpful.

\begin{exercise}\label{DetFact}
Let $N_1,\ldots,N_r\in{\mathbb R}^r$, where we endow ${\mathbb R}^r$ with the standard inner product, and we write
$N_i=n_{i1}e_1+\cdots+n_{ir}e_r$, where $e_i$ is the standard basis for ${\mathbb R}^r$.
Let $M$ be the matrix $M=(n_{ij})$, so $M^TM$ is the intersection matrix $(N_i\cdot N_j)$.
Then $\det(N_i\cdot N_j)=(\det(M))^2$ is clear, and $|\det(M)|\leq |N_1|\cdots|N_r|$,
since the volume of a parallelepiped with edges of fixed length is largest 
when the edges are orthogonal.
Thus we have $(\det(M))^2\leq |N_1|^2\cdots|N_r|^2=|(N_1\cdot N_1)\cdots(N_r\cdot N_r)|$.
(We include the absolute value sign at the end, since we will apply this in situations where we 
have divisors $N_1,\ldots,N_r$ that span a {\it negative} definite
subspace of the N\'eron-Severi group, but clearly the same result holds, with essentially the same proof.)
\end{exercise}
\vskip\baselineskip

\begin{exercise}\label{relprimeEx}
Let $X$ be a blow up of the plane at $s$ points.
Let $d,m_1,\ldots,m_s>0$ be integers, let
$C=dL-m_1E_1-\cdots-m_sE_s$ be any divisor with $C^2<0$ and let
the gcd of $d,m_1,\ldots,m_s$ be $g$.
Then there is an ample divisor $F$ such that $FC$ and $C^2$ have gcd $g$.
\end{exercise}

\SolnEater{\vskip\baselineskip 
\noindent{\it Details}: 
We can write $g=dd'-m_1m_1'-\cdots-m_sm_s'>0$ for some integers $d', m_1',\ldots,m_s'$.
Thus $g=d(d'+t(m_1+\cdots+m_s))-m_1(m_1'+td)-\cdots-m_s(m_s'+td)$
for all $t$. Now define $F=(d'+t(m_1+\cdots+m_s)+2gr|C^2|)L-(m_1'+td)E_1-\cdots-(m_s'+td)E_s$; then
$FC=g(2rd|C^2|+1)$ and $C^2$ have gcd $g$.
For $t\gg0$, the coefficients $d'+t(m_1+\cdots+m_s)+gr|C^2|$ and $m_i'+td$ for all $i$ will be positive for every $r>0$.
Fix such a $t$.
Now for $r\gg0$, $H=gr|C^2|L-(m_1'+td)E_1-\cdots-(m_s'+td)E_s$, and hence $F=H+(d'+t(m_1+\cdots+m_s)+gr|C^2|)L$,
will be linearly equivalent to an effective divisor. Thus there are only finitely many prime divisors
$D$ such that we can have $F\cdot D<0$. If for such a divisor we were to have $D\cdot L=0$, then $D=E_i$ for some $i$,
but then $F\cdot D=m_i'+td>0$, so we must have $D\cdot L>0$. Therefore, by increasing $r$ some more, 
$F$ will be nef. And if we still have a prime divisor $D$ such that $F\cdot D=0$, then 
we must have $D\cdot L>0$, so any additional increase in $r$ makes $F$ ample.
Thus for $r\gg0$, $F$ is ample.
\newline\qedsymbol\vskip\baselineskip}

\begin{proof}[Proof of Theorem \ref{BPStheorem}]
Assume bounded negativity holds on $X$; i.e., $C^2\geq -b$ for some $b>0$ and every reduced, irreducible curve $C$.
Let $D=d_1D_1+\cdots+d_rD_r$ be effective (so each $d_i$ is positive and
each $D_i$ is a prime divisor) with Zariski decomposition $D=P+N$.
Then $P$ and $N$ are sums of the $D_i$ with nonnegative rational coefficients.
Since $D$ is integral, the largest denominator used for $P$ is also the largest
denominator used for $N$, so it's enough to look at $N$. Say
$N=n_1N_1+\cdots+n_sN_s$ where each $n_i$ is positive rational
and each $N_i$ is a prime divisor of negative self-intersection. Note that
$DN_i=(n_1N_1+\cdots+n_sN_s)N_i$ gives linear equations for the $n_i$.
The solution involves dividing by $\det(N_iN_j)$, so the largest possible
denominator is $\det(N_iN_j)$, but $|\det(N_iN_j)|\leq |N_1^2\cdots N_s^2|$
by Example \ref{DetFact}.
By \Exercise\ \ref{LinIndep}, we have $s\leq |S|$. 
Thus the largest possible denominator is $|b|^{|S|}$, where $b$ is a lower bound for
self-intersections of irreducible curves on $X$.

Conversely, assume $X$ has bounded Zariski denominators, with bound $b$.
Let $C\sim dL-m_1E_1-\cdots-m_rE_r$ be any prime divisor with $C^2<0$ and define
$D=(dL-m_1E_1-\cdots-m_rE_r)/g$ where $g$ is the gcd of $d,m_1,\ldots,m_r$.
Thus $D$ is primitive (i.e., not linearly equivalent to $tD'$ for any integral divisor $D'$
with $t$ an integer bigger than 1). 
By \Exercise\ \ref{relprimeEx} we can pick an ample divisor $F$ such that 
$FC$ and $C^2$ have gcd $g$. Since the Zariski decomposition of $D$ is $D=C/g$, we have
$g\leq b$. But for large $m$, the Zariski decomposition of
$F+mC$ is $P=F+(m-a)C$ and $N=aC$ for some $a$,
so $a=(CF+mC^2)/C^2$, hence (putting $a$ into reduced terms) the denominator
needed here is $|C^2|/\operatorname{gcd}(CF,|C^2|)\leq b$, hence
$C^2\geq -b\operatorname{gcd}(CF,C^2) = -bg\leq -b^2$.
\end{proof}

\begin{exercise}\label{denomCriterion}
Let $X$ be the blow up of the plane at a finite number of points
such that there is a finite list $A=\{a_1,\ldots,a_r\}$ of integers such that
for every prime divisor $D$ with $D^2<-1$ we have $D^2\in A$.
Assume that there are at most $n_i$ distinct divisors $D$ with $D^2=a_i$
for each $i$ with $a_i<-1$.
Then no denominator bigger than $|a_1^{n_1}\cdots a_r^{n_r}|$
is ever needed for a Zariski decomposition on $X$.
\end{exercise}

\SolnEater{\vskip\baselineskip 
\noindent{\it Details}: 
For each Zariski decomposition having a nonzero negative part, the negative part
has an intersection matrix $(N_iN_j)$. Then $|\det(N_iN_j)|$ bounds the denominators which can occur
for that Zariski decomposition, but $|\det(N_iN_j)|$ is just the volume of a parallelepiped 
whose sides have length $|N_i^2|$. The volume is greatest when the sides are perpendicular,
hence $|\det(N_iN_j)|\leq |N_1^2\cdots N_s^2|\leq |a_1^{n_1}\cdots a_r^{n_r}|$. 
\newline\qedsymbol\vskip\baselineskip}

\begin{exercise}
Here we determine the largest denominator needed for a Zariski decomposition
when $X$ is the blow up of $r$ collinear points of the plane.
There is a unique reduced irreducible $D$ with $D^2<-1$, the proper transform $H$ of the line
for which $H^2=1-r$. Thus, by Example \ref{denomCriterion}, no denominator is needed larger than $r-1$. 
But $L+H=(L+aH)+((1-a)H)$ gives the Zariski decomposition when $a=1/(r-1)$ so $r-1$ 
is the biggest denominator.
\end{exercise}
\vskip\baselineskip

\begin{exercise}
And here we determine the largest denominator needed for a Zariski decomposition
when $X$ is the blow up of $r+1$ points of the plane on a line $L_1$
and $s+1$ points on a different line $L_2$, where one of the points
is the point of intersection of the two lines. Assume $r$ and $s$
are coprime and each is at least 2. [Note: By ``adjunction'', if $C$ is a prime divisor on $X$, we have
$C^2\geq -2 - C\cdot K_X$; see \cite[Example 3.2]{3refBPS2}.]
The proper transforms $H_1$ and $H_2$ of the two lines have $H_1^2=-r$ and $H_2^2=-s$.
Note that $-K_X=3L-E-E_1-\cdots-E_{r+s+1}\sim L + H_1 +H_2 + E$ where $E$ is the blow up of the point of intersection
of the two lines. 
By adjunction, if $C$ is a prime divisor other than $H_1$, $H_2$ or $E_i$, we have 
$C^2\geq -2 + C(L + H_1 +H_2 + E)\geq -1$ (since $CL>0$). Thus by \Exercise\ \ref{denomCriterion},
no Zariski denominator larger than $rs$ is ever needed.
But $L+H_1+H_2=(L+aH_1+bH_2)+((1-a)H_1+(1-b)H_2)$ is a Zariski decomposition when $a=1/r$ and $b=1/s$.
\end{exercise}

\section{Containment Problems} 

\subsection{Powers and symbolic powers}
Given distinct points $p_i\in\P^N$, let $Z=m_1p_1+\cdots+m_sp_s\subseteq \P^N$ be a fat point subscheme.
Recall that the $m$th symbolic power of $I(Z)$ is $I(mZ)=I(p_1)^{m_im}\cap\cdots\cap I(p_s)^{mm_i}$, sometimes denoted
$I(Z)^{(m)}$.
It is interesting to compare this with the $r$th ordinary power
$I(Z)^r=(I(p_1)^{m_1}\cap\cdots I(p_s)^{m_s})^r$ for various $m$ and $r$.
A useful fact here is that $I(Z)^r=Q\cap I(rZ)$ for some $M$-primary ideal
$Q$, where $M=(x_0,\ldots,x_N)$, and thus $I(rZ)$ is the saturation of $(I(Z))^r$. 
In particular, we see that $I(Z)^m\subseteq I(Z)^{(m)}$ for all $m\geq1$.
Moreover, $Q$ contains a power of $M$,
hence $[Q]_t=[M]_t$ for all $t\gg0$, hence $[I(Z)^r]_t=[I(rZ)]_t$ for all $t\gg0$.

\begin{exercise}\label{ContFacts}
Let $I=I(Z)$ for $Z=m_1p_1+\cdots+m_sp_s\subseteq \P^N$ with $m_i>0$ for all $i$. Then:
\begin{enumerate}
\item[(a)] $I^m\subseteq I^r$ if and only if $m\geq r$.
\item[(b)] $I^{(m)}\subseteq I^{(r)}$ if and only if $m\geq r$.
\item[(c)] $I^{m}\subseteq I^{(r)}$ if and only if $m\geq r$.
\item[(d)] $I^{(m)}\subseteq I^{r}$ implies $m\geq r$ but $m\geq r$ does not in general imply $I^{(m)}\subseteq I^{r}$.
\end{enumerate}
\end{exercise}

\SolnEater{\vskip\baselineskip 
\noindent{\it Details}: 
(a) Clearly $m\geq r$ implies $I^m\subseteq I^r$. Conversely, if $I^m\subseteq I^r$,
then $m\alpha(I)\geq r\alpha(I)$, hence $m\geq r$.

(b) If $m\geq r$, then $I(mm_ip_i)\subseteq I(rm_ip_i)$ for all $i$, so $I(mZ)=\cap_iI(mm_ip_i)\subseteq \cap_iI(rm_ip_i)=I(rZ)$.
Conversely, if $I^{(m)}\subseteq I^{(r)}$, then $[I^{(m)}]_t\subseteq [I^{(r)}]_t$ for $t\geq0$
so, using Theorem \ref{virtdim}, we have
$\sum_i\binom{mm_i+N-1}{N}=\deg(mZ)=\dim[\KK[\P^N]/I^{(m)}]_t\geq \dim[\KK[\P^N]/I^{(r)}]_t=\deg(rZ)=\sum_i\binom{rm_i+N-1}{N}$ for $t\gg0$.
But $\binom{b+N-1}{N}$ is an increasing function of $b$, so $m<r$ would imply $\binom{mm_i+N-1}{N}<\binom{rm_i+N-1}{N}$ for all $i$.

(c) If $m\geq r$, then $I^m\subseteq I^{(m)}\subseteq I^{(r)}$. Conversely, for $t\gg0$, we have $[I^m]_t=[I^{(m)}]_t$ so if 
$I^m\subseteq I^{(r)}$, then $[I^{(m)}]_t\subseteq [I^{(r)}]_t$ for $t\gg0$, and the argument for (b) shows that $m\geq r$.

(d) The first part follows from (a) and the fact that $I^m\subseteq I^{(m)}$. For the second, let $Z$ be three noncollinear points.
It is not hard to see that $I(Z)^2\subsetneq I(2Z)$.
\newline\qedsymbol\vskip\baselineskip}

It is a subtle and generally open problem to determine for which $m$ and $r$ we have $I^{(m)}\subseteq I^r$,
but for $m\gg0$ we always do have containment. To see this, we 
define the saturation degree of $I^r$:
$\operatorname{satdeg}(I^r)$ is the least $t$ such that $(I^r)_j=(I^{(r)})_j$ for all $j\geq t$.
(The original version of the next result used $m\geq \max(\operatorname{satdeg}(I^r),r)$;
the referee had the very nice suggestion to use $m\geq \max((\operatorname{satdeg}(I^r))/\widehat{\alpha}(I(Z)),r)$ instead.)

\begin{proposition}\label{satdegProp}
Let $I=I(Z)$ be a fat point scheme $Z\subseteq \P^N$. If 
$$m\geq \max\left(\frac{\operatorname{satdeg}(I^r)}{\widehat{\alpha}(I(Z))},r\right),$$ then
$I^{(m)}\subseteq I^r$.
\end{proposition}

\begin{proof}
Since $m\geq r$, we have $I^{(m)}\subseteq I^{(r)}$.
Since $m\geq \operatorname{satdeg}(I^r)$, if $[I^{(m)}]_t\neq0$, then 
$t\geq\alpha(I(mZ))\geq m\widehat{\alpha}(I(Z))\geq \operatorname{satdeg}(I^r)$, 
so $[I^{(m)}]_t\subseteq [I^{(r)}]_t=[I^r]_t$. Hence $I^{(m)}\subseteq I^r$.
\end{proof}

\begin{example}\label{containmentExample} 
The expression $\operatorname{satdeg}(I^r)$ in Proposition \ref{satdegProp} is complicated.
The quantity $\widehat{\alpha}(I(Z))$ is often not known exactly, and even after normalizing by dividing by $r$,
it is not known how large $\operatorname{satdeg}(I^r)/(r\widehat{\alpha}(I(Z)))$ can get, but when $N=2$ it can definitely be bigger than 2.
For example, let $Z$ be the reduced scheme consisting of $n+1=5$ points in $\P^2$, where 4 are on a line and one is off that line. 
Then by Example \ref{AlmostCollinear}, $\widehat{\alpha}(I(Z))=7/4$. A brute force calculation with $r=5$ shows that
$\operatorname{satdeg}(I(Z)^r)=18$, and hence that $\operatorname{satdeg}(I(Z)^r)/(r\widehat{\alpha}(I(Z)))=72/35\approx 2.057$
and $\operatorname{satdeg}(I(Z)^r)/\widehat{\alpha}(I(Z))=72/7\approx10.286$. Thus 
Proposition \ref{satdegProp} requires $m\geq 11$ to ensure $I(mZ)\subseteq I(Z)^r$ in this case.
This can be compared to Theorem \ref{ELSHH}, which shows that $m\geq rN$ suffices
for any fat points subscheme $Z\subseteq \P^N$ and any $r$ to ensure that $I(mZ)\subseteq I(Z)^r$.
\end{example}

Indeed, a formerly open question was: 

\begin{question}
Given $I=I(Z)$ for a fat point subscheme $Z\subseteq\P^N$, 
we know for each $r$ there is an $n$ such that $m\geq rn$ implies
$I^{(m)}\subseteq I^r$ (take $n=\max\{1,(\operatorname{satdeg}(I^r))/r\}$), but is there one $n$ that works for all $r$
and $Z$?
\end{question}
 
Motivated by \cite{refI3}, the papers \cite{2refELS, 2refHH} found the very general simple answer
given in Theorem \ref{ELSHH}. We do not give the definition here of symbolic powers for ideals that are not ideals of 
fat points; the definition used in \cite{2refHH} has the property that $I^{(1)}=I$, and that $I^{(m)}=I^m$ for all $m\geq1$
when $I$ is not saturated. 

\begin{theorem}\label{ELSHH}
Let $I\subseteq \KK[\P^N]$ be a homogeneous ideal. Let $r,s\geq1$.
Then $I^{(r(s+N-1))}\subseteq (I^{(s)})^r$.
In particular (taking $s=1$), if $m\geq rN$, then we have $I^{(m)}\subseteq I^r$ (since $I^{(m)}\subseteq I^{(rN)}\subseteq I^r$).
\end{theorem}

The question now became: is this result optimal? 
There are various approaches to this question.
Here's one showing no constant less than $N$ suffices \cite{refBH} (also see Example \ref{BHstarOptimalityEx}):

\begin{theorem}\label{BHstarOptimalityThm}
If $c<N$, there is an $r>0$ and $m>cr$ such that
$I^{(m)}\not\subseteq I^r$
for some $I=I(Z)$, where $Z=p_1+\cdots+p_s\subseteq \P^N$
for distinct points $p_i$.
\end{theorem}

\begin{exercise}\label{alphaContCrit}
Let $Z\subseteq\P^N$ be a fat point subscheme, $I=I(Z)$.
If $\alpha(I^{(m)})<r\alpha(I)$, then $I^{(m)}\not\subseteq I^r$.
\end{exercise}

\SolnEater{\vskip\baselineskip 
\noindent{\it Details}: We have $\alpha(I^r)=r\alpha(I)$ by \Exercise\ \ref{alphaLikeLog}
and, for any homogeneous ideals $J, J'\subseteq\KK[\P^N]$,
the fact that $J\subseteq J'$ implies $\alpha(J)\geq \alpha(J')$.
\newline\qedsymbol\vskip\baselineskip}

\begin{exercise}\label{BHstarOptimalityEx}
Pick $s>2$ lines in $\P^2$ so that at most two lines meet at any point. For simplicity, assume $s$ is even.
Let $Z$ be the $\binom{s}{2}$ crossing points and take $I=I(Z)$.
If $m<2r$ (again for simplicity, assume $m$ is even), then $I^{(m)}\not\subseteq I^r$ for $s\gg0$.
This shows that there is no $c<2$, such that $m\geq cr$ is enough to guarantee that $I^{(m)}\subseteq I^r$.
A similar construction holds for $\P^N$.
(Note: see \Exercise\ \ref{starconfigs}.)
\end{exercise}

\SolnEater{\vskip\baselineskip 
\noindent{\it Details}: 
Since $\widehat{\alpha}(I)=s/2$ by \Exercise\ \ref{starconfigs},
we have $ms/2=m\widehat{\alpha}(I)\leq \alpha(I^{(m)})$.
Let $F$ be the product of the linear forms defining the $s$ lines;
then $F^{m/2}\in \alpha(I^{(m)})$, so $\alpha(I^{(m)})\leq sm/2$, hence
$\alpha(I^{(m)}) = sm/2$. Using Bezout we get $\alpha(I)=s-1$. 
(More generally, if $m$ is odd the result is $\alpha(I^{(m)}) = s(m-1)/2+(s-1)$.)

Since $m<2r$, we have $\alpha(I^{(m)}) = sm/2<r(s-1)=\alpha(I^r)$,
as long as $2(s-1)/s > m/r$ (i.e., as long as $s>(2r-m)/(2r)$).
By \Exercise\ \ref{alphaContCrit}, this means $I^{(m)}\not\subseteq I^r$.
\newline\qedsymbol\vskip\baselineskip}

\subsection{The resurgence}
Although $m\geq Nr$ is optimal as a universal bound for homogeneous ideals in $\KK[\P^N]$,
what can one say about bounds for a specific ideal? This question leads to the definition of an asymptotic quantity
known as the resurgence \cite{refBH}.

\begin{definition}\label{resurgence}
Given a fat point scheme $Z\subseteq\P^N$, define the {\it resurgence} $\rho(I)$ for $I=I(Z)$ to be
$$\rho(I(Z))=\sup\Big\{\frac{m}{r}: I^{(m)}\not\subseteq I^r\Big\}.$$
\end{definition}

The following result is from \cite{refBH}. For this we need a new quantity, the regularity.

\begin{definition}\label{regDef}
The {\it regularity} $\reg(I)$ of $I=I(Z)$
for a fat point subscheme $Z\subseteq\P^N$
is defined by specifying that $\reg(I)-1$ is the least $t$ such that $\dim[I]_t=\binom{t+2}{2}-\deg(Z)$. 
\end{definition}

\begin{fact}\label{Dubreil}
Let $I$ be the ideal of a fat point subscheme of projective space.
An important fact about $\reg(I)$
is that $[I^r]_t=[I^{(r)}]_t$ for $t\geq r\reg(I)$ or even $t\geq \reg(I^r)$ (because $r\reg(I)\geq \reg(I^r)\geq \operatorname{satdeg}(I^r)$; see \cite{refBH}).
Another is that $I$ has a set of homogeneous generators each of which has degree at most $\reg(I)$
\cite{refD}. 
\end{fact}

\begin{theorem}\label{bndsOnResurgence}
Let $I=I(Z)$ for a nonempty fat point subscheme $Z\subseteq\P^N$.
\begin{enumerate}
\item[(a)] We have $1\leq\rho(I)\leq N$.
\item[(b)] If $m/r<\frac{\alpha(I)}{\widehat{\alpha}(I)}$, then for all $t\gg0$ we have $I^{(mt)}\not\subseteq I^{rt}$.
\item[(c)] If $m/r\geq \frac{\reg(I)}{\widehat{\alpha}(I)}$, then $I^{(m)}\subseteq I^{r}$.
\item[(d)] We have $$\frac{\alpha(I)}{\widehat{\alpha}(I)}\leq\rho(I)\leq \frac{\reg(I)}{\widehat{\alpha}(I)},$$
hence $\frac{\alpha(I)}{\widehat{\alpha}(I)}=\rho(I)$ if $\alpha(I)=\reg(I)$.
\end{enumerate}
\end{theorem}

\begin{proof}
(a) By Theorem \ref{ELSHH}, we have $\rho(I)\leq N$. By \Exercise\ \ref{ContFacts}(d), we have $\rho(I)\geq 1$.

(b) If $m/r<\frac{\alpha(I)}{\widehat{\alpha}(I)}$, then $\widehat{\alpha}(I)<r\alpha(I)/m$, so for $t\gg0$ we have
$\widehat{\alpha}(I)\leq \alpha(I^{(mt)})/(mt)<r\alpha(I)/m=rt\alpha(I)/(mt)=\alpha(I^{rt})/(mt)$ so also
$\alpha(I^{(mt)})<\alpha(I^{rt})$, hence $I^{(mt)}\not\subseteq I^{rt}$ by \Exercise\ \ref{alphaContCrit}.

(c) Now say $m/r\geq\frac{\reg(I)}{\widehat{\alpha}(I)}$. Then 
$\alpha(I^{(m)})\geq m\widehat{\alpha}(I)\geq r\reg(I)$.
If $t<\alpha(I^{(m)})$, then $[I^{(m)}]_t=(0)\subseteq I^r$. If $t\geq \alpha(I^{(m)})$, then
$t\geq r\reg(I)$ hence $[I^{(m)}]_t\subseteq [I^{(r)}]_t=[I^r]_t$. Thus $I^{(m)}\subseteq I^{r}$.

(d) This follows from (b) and (c).
\end{proof}

No examples are known with $\rho(I) = N$, but there are a lot of examples with $\rho(I)=1$.
For example, if $|Z|=1$, so $Z$ consists of a single reduced point,
then $\rho(I)=1$, since $I^{(m)}=I^m$, but it is not not known if
$\rho(I)=1$ guarantees that $I^m=I^{(m)}$ for all $m$.

An asymptotic version of the resurgence was introduced in \cite{refGHVT}.

\begin{definition}\label{asympresurgence}
Given a fat point scheme $Z\subseteq\P^N$, define the {\it asymptotic resurgence} $\widehat{\rho}(I)$ for $I=I(Z)$ to be
$$\widehat{\rho}(I(Z))=\sup\Big\{\frac{m}{r}: I^{(ms)}\not\subseteq I^{rs}\hbox{ for }s\gg0\Big\}.$$
\end{definition}

In contrast to the case of the resurgence, the result of the following example holds not just for ideals $I$ of 
points, which is one advantage of the asymptotic resurgence (see \cite{refGHVT}).

\begin{exercise}\label{asymresurgEx}
Let $Z\subseteq\P^N$ be a fat point subscheme and let $I=I(Z)\subseteq\P^N$.
\begin{enumerate}
\item[(a)] Then $1\leq\widehat{\rho}(I)\leq \rho(I)$.
\item[(b)] One can show that 
$$\frac{\alpha(I)}{\widehat{\alpha}(I)}\leq\widehat{\rho}(I)\leq \frac{\omega(I)}{\widehat{\alpha}(I)},$$
where $\omega(I)$ is the maximal degree among a minimal set of homogeneous generators of $I$.
(Note: One can mimic the proof of Theorem \ref{bndsOnResurgence}(d), using the fact that 
there is a constant $c$ such that $\reg(I^s)\leq s\omega(I)+c$ for all $s>0$ \cite{refK}.)
\end{enumerate}
\end{exercise}

\SolnEater{\vskip\baselineskip 
\noindent{\it Details}: (a) We have $1\leq\widehat{\rho}(I)$ by \Exercise\ \ref{ContFacts}(d);
$\widehat{\rho}(I)\leq\rho(I)$ is clear by definition.

(b) The same proof as for Theorem \ref{bndsOnResurgence}(b) shows 
$\frac{\alpha(I)}{\widehat{\alpha}(I)}\leq\widehat{\rho}(I)$.

Now say $m/r>\frac{\omega(I)}{\widehat{\alpha}(I)}$; i.e., $m\widehat{\alpha}(I)=r\omega(I)+\delta$ for some $\delta>0$. Then, for $s\gg0$,
$\alpha(I^{(ms)})\geq ms\widehat{\alpha}(I) = rs\omega(I)+s\delta\geq s\omega(I)+c\geq \reg(I^s)$.
Now argue as in the proof of Theorem \ref{bndsOnResurgence}(c), using Fact \ref{Dubreil}.
\newline\qedsymbol\vskip\baselineskip}

\subsection{Other perspectives on optimality}
By Theorem \ref{BHstarOptimalityThm}, the bound $m\geq rN$ in Theorem \ref{ELSHH} is optimal,
in the sense that $N$ cannot be replaced by a smaller number and always still have the containment
$I^{(m)}\subseteq I^{r}$. 
But given the containment $I^{(Nr)}\subseteq I^r$, one can ask whether there are other ways to make the 
$I^{(Nr)}$ bigger or the ideal $I^r$ smaller and still always have containment.

For example, Craig Huneke raised the question: Given a reduced 0-dimensional subscheme $Z\subseteq\P^2$,
to what extent is the result $I(4Z)\subseteq I(Z)^2$ optimal? In particular, is it always true that $I(3Z)\subseteq I(Z)^2$?

Experimentation and partial results suggested the answer is Yes (it is true for example
if $\KK$ has characteristic 2; see \cite{refBetal}).
Thus I raised a more general conjecture \cite{refBetal}, a simplified version of which is:

\begin{conjecture}\label{contConj}
Let $Z\subseteq \P^N$ be a fat point subscheme. Then $I((Nr-N+1)Z)\subseteq I(Z)^r$ for all $r\geq 1$. 
\end{conjecture}

No counterexamples over the complexes are known except for $N=r=2$.
Huneke's question is for the case that $r=N=2$. The 
first counterexample for any $r$ and $N$ over any field $\KK$ was for $N=r=2$ over $\C$:
take the points $Z$ of the Fermat arrangement for $n=3$ (see Remark \ref{t_2=0Rem}). Then $I(3Z)\not\subseteq I(Z)^2$ \cite{refDST13}.
Additional counterexamples were soon found: there is a version with $N=r=2$ in characteristic 3 \cite{refBCH}, 
and additional positive characteristic counterexamples are now known for various $r$ and $N$ \cite{refHS14}.
Over $\C$, one can also take $Z$ to be the points of the Fermat for any $n\geq 3$ \cite{refHS14}, 
or the Klein or Wiman \cite{3refIMRN,refSec} (again see Remark \ref{t_2=0Rem}).
Additional failures of containment for $N=r=2$ are given in \cite{refCzEtEl,refDHNSST}.
A recent paper \cite{refA} leverages these examples, by obtaining others by pulling them back by a finite cover of $\P^2$.

\begin{example}\label{FermatConfig}
Here is Macaulay2 code for verifying $I^{(3)}\not\subseteq I^2$ for
the $n^2+3$ points of the Fermat arrangement with $3n$ lines.
\begin{verbatim}
R=QQ[x,y,z];
n=5;
I=ideal(x^n-y^n, x^n-z^n);
J=ideal(x*y,x*z,y*z);
K=intersect(I,J);  -- Ideal of the n^2+3 Fermat points 
K3=intersect(I^3,saturate(J^3));  -- I is a complete intersection
                                  -- so I^3 is already saturated
isSubset(K3,K^2)
\end{verbatim}
\end{example}

\begin{example}\label{KleinConfig}
Here is Macaulay2 code for verifying $I^{(3)}\not\subseteq I^2$ for
the 49 points of the Klein arrangement of 21 lines.
\begin{verbatim}
-- Define the field
K=toField(QQ[c]/(c^2+c+2))
R=K[x,y,z];
-- Define the lines
F={x, x+c*y-z, -x+c*y-z, x+c*y+z, -x+y+c*z, y+z, c*x+y-z, z, c*x+y+z, c*x-y-z, 
-x+z, -x-y+c*z, -x+y, c*x-y+z, -x+c*y+z, x+z, -y+z, x+y, x-y+c*z, x+y+c*z, y};
-- Find the product of the 21 linear forms
H=product F;
-- Make a list of the ideals of the 49 intersection points of pairs of lines
W=subsets(21,2);
W4={};
apply(W,s->(flag=0;apply(W4,t->(if ideal(F_(t_0),F_(t_1))==ideal(F_(s_0),F_(s_1)) 
            then flag=1)); if flag==0 then W4=W4|{s}));
-- Define the ideal of the points
I=ideal(1_R);
apply(W4,s->(I=intersect(I,ideal(F_(s_0),F_(s_1)))));
-- Since H is in I^(3), it is enough to check that H is not in I^2
isSubset(ideal(H),I^2)
\end{verbatim}
\end{example}

\begin{example}\label{WimanConfig}
Here is Macaulay2 code for verifying $I^{(3)}\not\subseteq I^2$ for
the 201 points of the Wiman arrangement of 45 lines.
\begin{verbatim}
-- Define the field
K=toField(QQ[a]/(a^4-a^2+4))
R=K[x,y,z];
-- Define the lines
A=(-1/4)*(a^3-3*a-2);
B=(1/4)*(a^3+a-2);
F={y,(-1+A)*x+A*y+z,z,(1-A)*x+A*y-z,A*x+y+(-1+A)*z,-A*x+y+(1-A)*z,
(-1+A)*x-B*y+(-A-A*B)*z,(1-A)*x-B*y+(A+A*B)*z, (1-A)*x+A*y+z, 
A*x+y+(1-A)*z, -x+(-1+A)*y+A*z, (-1-A*B)*x+y+(-1-B)*z, 
(1-A)*x+B*y+(-A-A*B)*z, A*x+(B-A*B)*y+(-1-B)*z, (-A-A*B)*x+(1-A)*y-B*z, 
(-1+A)*x+A*y-z, -A*x+y+(-1+A)*z, x+(-1+A)*y-A*z, (1+A*B)*x+y+(1+B)*z, 
(-1+A)*x+B*y+(A+A*B)*z, -A*x+(B-A*B)*y+(1+B)*z, (A+A*B)*x+(1-A)*y+B*z, 
(1+B)*x+(-1-A*B)*y+z, x+(-1+A)*y+A*z, x+(1-A)*y+A*z, (-1-A*B)*x+y+(1+B)*z, 
(-A-B)*x+(-1+A+A*B)*y, -B*x+y+(-A+B-A*B)*z, (-1-A*B)*x-y+(1+B)*z, 
(-1-B)*x+A*y+(B-A*B)*z, (-1-B)*x+(-1-A*B)*y-z, (A+B)*x+(-1+A+A*B)*y, 
B*x+y+(A-B+A*B)*z, (1+B)*x+A*y+(-B+A*B)*z, (-1+A+A*B)*x+(-A-B)*z, x, 
(-1-B)*x+A*y+(-B+A*B)*z, (-A-B)*y+(-1+A+A*B)*z, -B*x+y+(A-B+A*B)*z, 
(1+B)*x-B*y+(1-A+B)*z, x-A*B*y+(1-A+B-A*B)*z, (-A-B)*y+(1-A-A*B)*z, 
(1+B)*x+(1+A*B)*y-z, (A+B)*x+(1+B-A*B)*z, B*x+(-1+A-B)*y+(-1-B)*z};
-- Find the product of the 45 linear forms
H=product F;
-- Make a list of the ideals of the 49 intersection points of pairs of lines
W=subsets(45,2);
W4={};
apply(W,s->(flag=0;apply(W4,t->(if ideal(F_(t_0),F_(t_1))==ideal(F_(s_0),F_(s_1)) 
            then flag=1)); if flag==0 then W4=W4|{s}));
-- Define the ideal of the points
I=ideal(1_R);
apply(W4,s->(I=intersect(I,ideal(F_(s_0),F_(s_1)))));
-- Since H is in I^(3), it is enough to check that H is not in I^2
isSubset(ideal(H),I^2)
\end{verbatim}
\end{example}

Additional counterexamples arise by taking subsets of points of the Wiman arrangement.

\begin{example}\label{WimanConfigDropA3Point}
Here is Macaulay2 code for verifying $I^{(3)}\not\subseteq I^2$ for
200 of the 201 points of the Wiman arrangement of 45 lines. The missing point
has multiplicity 3 in this case, but similar failures of containment occur
by instead excluding a 4-point or a 5-point. The ideal of the 201 Wiman points 
is generated by three forms of degree 16. The ideal of the 200 points
has an additional generator of degree 25, but the symbolic cube is
generated in degree at most 49 (it has the usual degree 45 element,
20 generators of degree 48 and 6 of degree 49). Thus all 
homogeneous elements of $I^2$ of degree 49 or less
vanish at all 201 points, but $I^{(3)}$ has elements of degree
49 that do not vanish at the missing point, and so 
$I^{(3)}\not\subseteq I^2$.
\begin{verbatim}
-- Define the field
K=toField(QQ[a]/(a^4-a^2+4))
R=K[x,y,z];
-- Define the lines
A=(-1/4)*(a^3-3*a-2);
B=(1/4)*(a^3+a-2);
F={y,(-1+A)*x+A*y+z,z,(1-A)*x+A*y-z,A*x+y+(-1+A)*z,-A*x+y+(1-A)*z,
(-1+A)*x-B*y+(-A-A*B)*z, (1-A)*x-B*y+(A+A*B)*z, (1-A)*x+A*y+z, 
A*x+y+(1-A)*z, -x+(-1+A)*y+A*z, (-1-A*B)*x+y+(-1-B)*z, 
(1-A)*x+B*y+(-A-A*B)*z, A*x+(B-A*B)*y+(-1-B)*z, (-A-A*B)*x+(1-A)*y-B*z, 
(-1+A)*x+A*y-z, -A*x+y+(-1+A)*z, x+(-1+A)*y-A*z, (1+A*B)*x+y+(1+B)*z, 
(-1+A)*x+B*y+(A+A*B)*z, -A*x+(B-A*B)*y+(1+B)*z, (A+A*B)*x+(1-A)*y+B*z, 
(1+B)*x+(-1-A*B)*y+z, x+(-1+A)*y+A*z, x+(1-A)*y+A*z, (-1-A*B)*x+y+(1+B)*z, 
(-A-B)*x+(-1+A+A*B)*y, -B*x+y+(-A+B-A*B)*z, (-1-A*B)*x-y+(1+B)*z, 
(-1-B)*x+A*y+(B-A*B)*z, (-1-B)*x+(-1-A*B)*y-z, (A+B)*x+(-1+A+A*B)*y, 
B*x+y+(A-B+A*B)*z, (1+B)*x+A*y+(-B+A*B)*z, (-1+A+A*B)*x+(-A-B)*z, x, 
(-1-B)*x+A*y+(-B+A*B)*z, (-A-B)*y+(-1+A+A*B)*z, -B*x+y+(A-B+A*B)*z, 
(1+B)*x-B*y+(1-A+B)*z, x-A*B*y+(1-A+B-A*B)*z, (-A-B)*y+(1-A-A*B)*z, 
(1+B)*x+(1+A*B)*y-z, (A+B)*x+(1+B-A*B)*z, B*x+(-1+A-B)*y+(-1-B)*z};
-- Make a list of the ideals of the 49 intersection points of pairs of lines
W=subsets(45,2);
W4={};
apply(W,s->(flag=0;apply(W4,t->(if ideal(F_(t_0),F_(t_1))==ideal(F_(s_0),F_(s_1)) 
            then flag=1)); if flag==0 then W4=W4|{s}));
-- Find the multiplicity of the point where line i and line j intersect
W5={}
W5=apply(W4,s->(n=0;apply(F,t->(if isSubset(ideal(t),ideal(F_(s_0),F_(s_1))) 
         then n=n+1)); W5|{s,n}));
-- {1,2} turns out to be a 3-point:
W5_2
-- Remove this 3-point
W6=delete({1,2},W4);
-- Define the ideal of the points
I=ideal(1_R);
apply(W6,s->(I=intersect(I,ideal(F_(s_0),F_(s_1)))));
-- It turns out that the product H of the linear forms is in I^2 so we need to  
-- compute I^(3), which is slow.
I3=ideal(1_R);
apply(W6,s->(I3=intersect(I3,(ideal(F_(s_0),F_(s_1)))^3)));
-- Alternatively, one could try: I3=saturate(I^3);
isSubset(I3,I^2)
\end{verbatim}
\end{example}

\begin{openproblem}\label{WimanSubsets}
For which subsets $Z$ of the 201 points of the Wiman
arrangement do we have $I^{(3)}\not\subseteq I^2$,
for $I=I(Z)$? 
\end{openproblem}

Counterexamples also occur over the reals \cite{refCzEtEl} and 
one of them can be made to work over the rationals \cite{refDHNSST}.
This one is displayed in Figure \ref{realCntrExmple}. Take for $Z$ the 19 crossing points
of multiplicity 3. Then $I(3Z)\not\subseteq I(Z)^2$. In all of these counterexamples
(i.e., the counterexample coming from the Fermat, Klein and Wiman line arrangements and the
counterexample coming from the arrangement displayed in Figure \ref{realCntrExmple}),
the failure is due to the fact that the form $F$ coming from taking all of the lines of a line arrangement
satisfies $F\in I(3Z)$ but $F\not\in I(Z)^2$. 

Another common feature of all of these counterexamples is that $t=\deg(F)/3$ is an integer, 
and the least $m$ with $\dim [I(mZ)]_{mt}>0$ is $m=3$. One might hope that 
$I(3Z)\not\subseteq I(Z)^2$ if and only if $t=\deg(F)/3$ is an integer, 
and the least $m$ with $\dim [I(mZ)]_{mt}>0$ is $m=3$.
It is possible that this gives a necessary condition, but it is not sufficient.

For example, consider the line arrangement shown in Figure \ref{realCntrExmple}.
It has 12 lines and 19 triple points. Let $Z$ be the reduced scheme consisting of those 19 points.
Then $I(3Z)\not\subseteq I(Z)^2$, 3 divides $\deg(F)$ and the least $m$ with $\dim [I(mZ)]_{m4}>0$ is $m=3$.
Now consider the line arrangement shown in Figure \ref{McKeeDual}. It is the dual
of a combinatorially well-known arrangement of 13 points known as the McKee arrangement, shown in Figure \ref{McKeeConfig}.
Take as the lines of a line arrangement the 12 nondotted lines in Figure \ref{McKeeDual}.
It has 1 quadruple point, 18 triple points and 6 double points. Take for $Z$ the
19 points of multiplicity more than 2. Then $I(3Z)\subseteq I(Z)^2$ even though
3 divides $\deg(F)$ and the least $m$ with $\dim [I(mZ)]_{m4}>0$ is $m=3$.

This raises the question: 
\begin{openproblem}\label{ContQues}
Given a point set $Z$ coming from a line arrangement such that $I(3Z)\not\subseteq I(Z)^2$, must there 
be a $t$ such that the least $m$ with $\dim [I(mZ)]_{mt}>0$ is $m=3$? Must $\deg(F)$ be a multiple of 3
(where $F$ is the form defining the union of  the lines)? 
\end{openproblem}

\begin{figure}[h]
\begin{tikzpicture}[x=.75cm,y=.75cm]
\clip(-6.295541561786057,-2.7118849043725834) rectangle (7.507390761802504,8.288421566710701);
\draw [domain=-6.295541561786057:7.507390761802504] plot(\x,{(--4.308854742566771-3.7781170691944217*\x)/1.4573087803652975});
\draw [domain=-6.295541561786057:7.507390761802504] plot(\x,{(--3.6262606986009076--2.3868366676731565*\x)/2.587598630374626});
\draw [domain=-6.295541561786057:7.507390761802504] plot(\x,{(-21.1897240279081--1.3912804015212652*\x)/-4.044907410739923});
\draw [domain=-6.295541561786057:7.507390761802504] plot(\x,{(--9.052652261566836--1.4354160754323595*\x)/5.353689635521689});
\draw [domain=-6.295541561786057:7.507390761802504] plot(\x,{(-23.336433152348445-7.686967124283885*\x)/-2.7803230550587763});
\draw [domain=-6.295541561786057:7.507390761802504] plot(\x,{(-13.823121538270541-5.621052349141927*\x)/-3.5771950646237727});
\draw [domain=-6.295541561786057:7.507390761802504] plot(\x,{(-21.119980432508115--0.2767862186201042*\x)/-6.0141264063219895});
\draw [domain=-6.295541561786057:7.507390761802504] plot(\x,{(--3.82935328761021--2.3478598593786866*\x)/-0.06914043897792954});
\draw [domain=-6.295541561786057:7.507390761802504] plot(\x,{(-8.429716098031342--1.5803614713398195*\x)/-1.461534024303433});
\draw [domain=-6.295541561786057:7.507390761802504] plot(\x,{(--13.97744076695705--4.8878253451632405*\x)/0.8810581619647937});
\draw [domain=-6.295541561786057:7.507390761802504] plot(\x,{(--3.7685284024981756--2.601269392035054*\x)/4.649191503365426});
\draw [domain=-6.295541561786057:7.507390761802504] plot(\x,{(--22.578370086368736-2.3095271031601747*\x)/3.712783904352978});
\begin{scriptsize}
\draw [fill=black] (-1.0148207931868989,5.587674083345914) circle (1.5pt);
\draw [fill=black] (0.44248798717839866,1.8095570141514923) circle (1.5pt);
\draw [fill=black] (3.0300866175530246,4.196393681824649) circle (1.5pt);
\draw [fill=black](5.796177622700087,3.2449730895838518) circle (1.5pt);
\draw [fill=black](-0.21794878362190206,3.521759308203956) circle (1.5pt);
\draw [fill=black](-3.795143848245675,-2.0992930409379706) circle (1.5pt);
\draw [fill=black] (-2.684591809853696,0.9711332512128636) circle (1.5pt);
\draw [fill=black] (-1.6676251409406595,1.2437994691144594) circle (1.5pt);
\draw [fill=black] (2.179109096757502,3.4114402790218783) circle (1.5pt);
\draw [fill=black] (-1.736765579918589,3.591659328493146) circle (1.5pt);
\draw [fill=black] (0.7175750724540692,4.991801750361698) circle (1.5pt);
\draw [fill=black] (-1.8600762902341084,7.779024512473213) circle (1.5pt);
\draw [fill=black] (3.0211486423482095,2.5009401927647215) circle (1.5pt);
\draw [fill=black] (-1.8035336478889021,5.858958596376104) circle (1.5pt);
\draw [fill=black] (-1.628042861017216,-0.10032919927033278) circle (1.5pt);
\draw [fill=black] (-0.6826972867999531,6.5059207849848235) circle (1.5pt);
\draw [fill=black] (-3.017903127344981,-0.8779699928119649) circle (1.5pt);
\draw [fill=black] (4.460785376258695,3.306431420402868) circle (1.5pt);
\draw [fill=black] (-1.5856672542009285,7.067610790860788) circle (1.5pt);
\end{scriptsize}
\end{tikzpicture}
\caption{An arrangement of $12$ lines with $19$ triple points (and 9 double points).}
\label{realCntrExmple}
\end{figure}
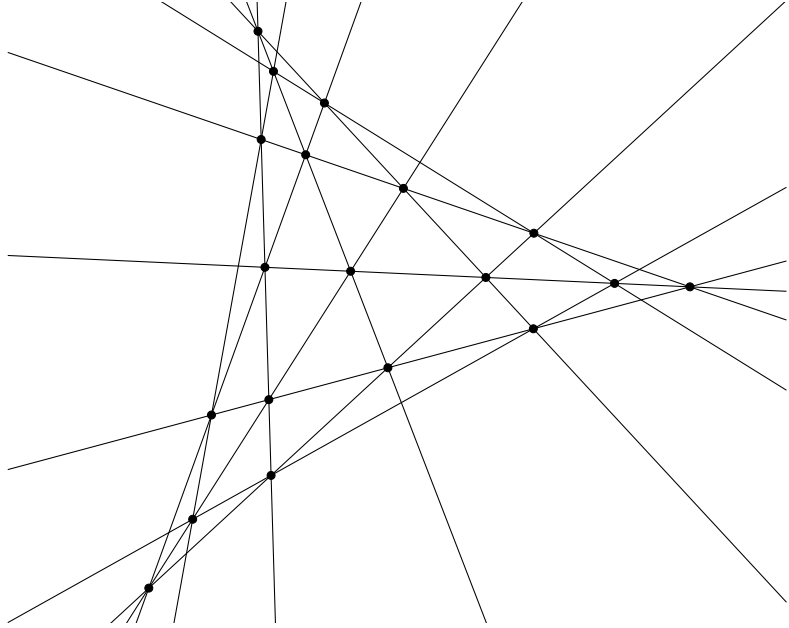

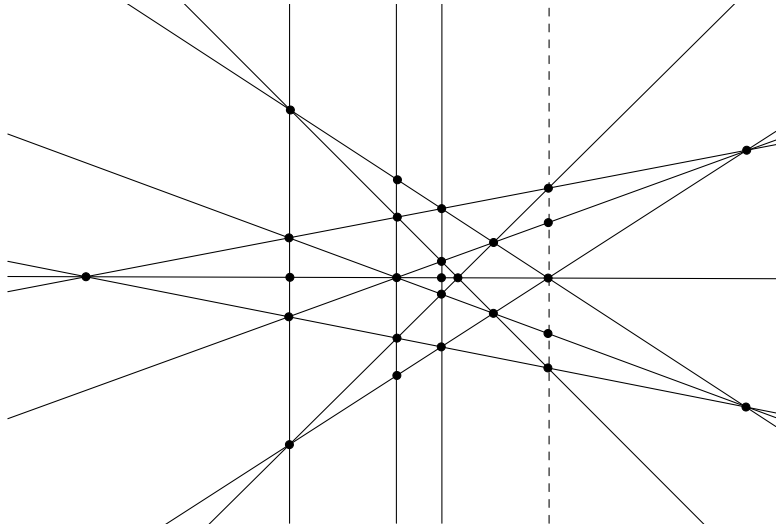
\begin{figure}[h]
\definecolor{uuuuuu}{rgb}{0,0,0}
\definecolor{xdxdff}{rgb}{0,0,0}
\definecolor{qqqqff}{rgb}{0,0,0}
\begin{tikzpicture}[line cap=round,line join=round,>=triangle 45,x=.5cm,y=.5cm]
\clip(-9.791393015282868,-4) rectangle (10.631344423249304,9.813777815169532);
\draw [domain=-9.791393015282868:10.631344423249304] plot(\x,{(--31.190119039675295-0.03521161627333136*\x)/12.288854079392635});
\draw (-2.3,-4)-- (-2.3,10);
\draw (.54,-4)-- (.54,10);
\draw (1.75,-4)-- (1.75,10);
\draw [dashed](4.6,-4)-- (4.6,10);
\draw [domain=-9.791393015282868:10.631344423249304] plot(\x,{(--49.65427264460399--2.3561329613382114*\x)/12.295706069586362});
\draw [domain=-9.791393015282868:10.631344423249304] plot(\x,{(--12.725965434746593-2.426556193884873*\x)/12.282002089198905});
\draw [domain=-9.791393015282868:10.631344423249304] plot(\x,{(--37.76463828098703-4.472040914700182*\x)/6.853653088250459});
\draw [domain=-9.791393015282868:10.631344423249304] plot(\x,{(--32.26752198361097-6.863385492311724*\x)/6.846801098056729});
\draw [domain=-9.791393015282868:10.631344423249304] plot(\x,{(-2.9096123949887405--4.432691856906153*\x)/6.879168081865377});
\draw [domain=-9.791393015282868:10.631344423249304] plot(\x,{(--2.587503902387315--6.824036434517696*\x)/6.886020072059106});
\draw [domain=-9.791393015282868:10.631344423249304] plot(\x,{(--3.2361294530801166--0.5026766643280705*\x)/1.384488680372556});
\draw [domain=-9.791393015282868:10.631344423249304] plot(\x,{(--3.7843930648259576-0.5106023777446933*\x)/1.3815853020285078});
\begin{scriptsize}
\draw [fill=qqqqff] (-7.713907655156318,2.56018486286327) circle (1.5pt);
\draw [fill=qqqqff] (4.574946424236316,2.524973246589939) circle (1.5pt);
\draw [fill=xdxdff] (4.581798414430045,4.916317824201482) circle (1.5pt);
\draw [fill=xdxdff] (4.568094434042586,0.13362866897839698) circle (1.5pt);
\draw [fill=xdxdff] (-2.2914641608216018,2.5446477754869528) circle (1.5pt);
\draw [fill=xdxdff] (-2.278706664014143,6.997014161290121) circle (1.5pt);
\draw [fill=xdxdff] (-2.3042216576290606,-1.9077186103162145) circle (1.5pt);
\draw [fill=uuuuuu] (0.5648410882553065,4.146578737488936) circle (1.5pt);
\draw [fill=uuuuuu] (1.7435819315658139,4.372451911268611) circle (1.5pt);
\draw [fill=uuuuuu] (0.555614142704531,0.9263747402683429) circle (1.5pt);
\draw [fill=uuuuuu] (1.7330412394674868,0.6937503689526409) circle (1.5pt);
\draw [fill=uuuuuu] (1.739560618893815,2.9690137887411328) circle (1.5pt);
\draw [fill=uuuuuu] (3.124049299266371,3.471690453069203) circle (1.5pt);
\draw [fill=uuuuuu] (1.7370625521394856,2.0971884914801193) circle (1.5pt);
\draw [fill=uuuuuu] (3.1186478541679934,1.586586113735426) circle (1.5pt);
\draw [fill=uuuuuu] (9.853959736012978,5.92658213899344) circle (1.5pt);
\draw [fill=uuuuuu] (9.834379753471463,-0.9068317679941086) circle (1.5pt);
\draw [fill=uuuuuu] (-2.314997040391133,3.5947371840375295) circle (1.5pt);
\draw [fill=uuuuuu] (-2.321014308920466,1.49471046730219) circle (1.5pt);
\draw [fill=uuuuuu] (0.5483465453734465,2.53651078205946) circle (1.5pt);
\draw [fill=uuuuuu] (4.579172899364175,4.0000130662128415) circle (1.5pt);
\draw [fill=uuuuuu] (4.570719949108456,1.0499334269670344) circle (1.5pt);
\draw [fill=uuuuuu] (1.7383115855166502,2.533101140110626) circle (1.5pt);
\draw [fill=uuuuuu] (2.175670094769602,2.5318479638663476) circle (1.5pt);
\draw [fill=uuuuuu] (0.567686791093013,5.139729027848474) circle (1.5pt);
\draw [fill=uuuuuu] (0.5527684398668246,-0.06677555009119462) circle (1.5pt);
\end{scriptsize}
\end{tikzpicture}
\caption{The dual of the McKee arrangement: 13 lines with 6 double points, 18 triple points and 3 quadruple points
(one of which is at infinity, in the direction of the vertical lines).}
\label{McKeeDual}
\end{figure}

\begin{figure}[h]
\definecolor{zzttqq}{rgb}{0,0,0}
\definecolor{xdxdff}{rgb}{0,0,0}
\definecolor{uuuuuu}{rgb}{0,0,0}
\begin{tikzpicture}[line cap=round,line join=round,>=triangle 45,x=.5cm,y=.5cm]
\clip(-5,-5) rectangle (7,5);
\fill[color=zzttqq,fill=zzttqq,fill opacity=0.1] (0.0,-0.0) -- (2.0,0.0) -- (2.618033988749895,1.9021130325903064) -- (1.0000000000000002,3.077683537175253) -- (-0.6180339887498947,1.9021130325903073) -- cycle;
\fill[color=zzttqq,fill=zzttqq,fill opacity=0.1] (0.0,-0.0) -- (2.0,0.0) -- (2.618033988749895,-1.9021130325903064) -- (1.0000000000000002,-3.077683537175253) -- (-0.6180339887498947,-1.9021130325903073) -- cycle;
\draw (0.0,-7.394651999999992) -- (0.0,6.906676799999993);
\draw (2.0,-7.394651999999992) -- (2.0,6.906676799999993);
\draw [domain=-7.249149600000005:11.432766400000004] plot(\x,{(-1.9021130325903064--1.9021130325903064*\x)/2.618033988749895});
\draw [domain=-7.249149600000005:11.432766400000004] plot(\x,{(--1.9021130325903064-1.9021130325903064*\x)/2.618033988749895});
\draw [domain=-7.249149600000005:11.432766400000004] plot(\x,{(--6.155367074350506-0.0*\x)/2.0});
\draw [domain=-7.249149600000005:11.432766400000004] plot(\x,{(-6.155367074350506--0.0*\x)/2.0});
\draw [dash pattern=on 8pt off 8pt,domain=-7.249149600000005:11.432766400000004] plot(\x,{(-0.0--1.9021130325903064*\x)/2.618033988749895});
\draw [dash pattern=on 8pt off 8pt,domain=-7.249149600000005:11.432766400000004] plot(\x,{(-0.0-1.9021130325903064*\x)/2.618033988749895});
\begin{scriptsize}
\draw [fill=uuuuuu] (0.0,-0.0) circle (1.5pt);
\draw [fill=xdxdff] (1.0,0.0) circle (1.5pt);
\draw [fill=xdxdff] (2.0,0.0) circle (1.5pt);
\draw [fill=uuuuuu] (2.618033988749895,1.9021130325903064) circle (1.5pt);
\draw [fill=uuuuuu] (1.0000000000000002,3.077683537175253) circle (1.5pt);
\draw [fill=uuuuuu] (-0.6180339887498947,1.9021130325903073) circle (1.5pt);
\draw [fill=uuuuuu] (-0.6180339887498947,-1.9021130325903073) circle (1.5pt);
\draw [fill=uuuuuu] (1.0000000000000002,-3.077683537175253) circle (1.5pt);
\draw [fill=uuuuuu] (2.618033988749895,-1.9021130325903064) circle (1.5pt);
\end{scriptsize}
\end{tikzpicture}
\caption{The McKee arrangement is based on two abutting regular pentagons. It has  $n=13$ points (including 4 at infinity in the directions of the lines)
having only 6 ordinary lines (i.e., less than $n/2$ lines through exactly two points of the arrangement, shown as solid lines; 
note that the solid diagonal lines are parallel to the dashed lines).}
\label{McKeeConfig}
\end{figure}
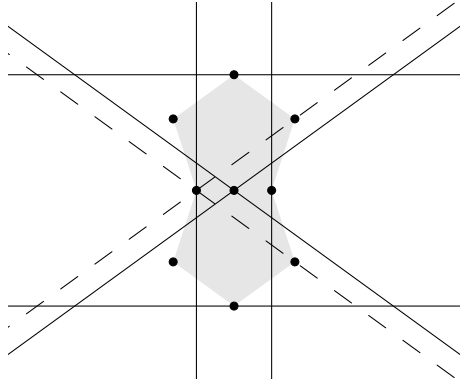

\begin{example}\label{ContFails}
All of the complex examples $I(3Z)\not\subseteq I(Z)^2$ known so far come via \cite{refA} by starting with a subset $Z$ of the singular points of
a line arrangement, where $Z$ excludes the points of multiplicity 2 and includes 
most of the points of multiplicity at least 3, for line arrangements that have only a few or no points of multiplicity 2
(i.e., $t_2$ is small or 0).
To see why the size of $t_2$ might be relevant,
consider a line arrangement having $t_2=0$. Let $F$ be the product of the linear forms of the lines.
Let $Z$ be the crossing points of the lines. Then since at least three lines cross at each crossing point,
we have $F\in I(3Z)$. If it turns out that $F\not\in I(Z)^2$, then we have  
$I(3Z)\not\subseteq I(Z)^2$. It can take work to check whether $F\not\in I(Z)^2$ (see \cite{refSec}).
The simplest case might be as follows \cite{refBCH}.
Take $\chr(\KK)=3$. Choose the point $p_0=[0:0:1]$ (represented in Figure \ref{12ptsChar3} by the open dot). There are 9 
lines defined over the prime
subfield $\mathbb F_3$ which do not contain this point. They give an arrangement of 9 lines with 12 crossing points, and
every crossing point has multiplicity 3. Take $Z$ to be these 12 points. Note that for each point, the 3 lines
of the arrangement through that point also go through 3 more of the points.
By Bezout's Theorem, $\alpha(I(Z))>3$, and clearly $\alpha(I(Z))\leq 4$.

The claim is that $\dim[I(Z)]_4=3$. There are various ways to verify this: for example, use facts about Hilbert functions,
or run it on a computer. Here's a third way (based on the method of \cite{refCHT}), in reference to Figure \ref{12ptsChar3}. 
Blow up the 11 points $p_2,\ldots,p_{12}$ shown 
to get a surface $X$. Let $E_i$ be the blow up of $p_i$. Denote the proper transforms of the lines $L_i$ 
also by $L_i$. 
We have
$$0\to \O_X\to \O_X(L_4)\to \O_{L_4}(-1)\to 0$$
and since $h^1(X, \O_X)=0=h^1(\P^1,\O_{\P^1}(-1))=h^1(L_4,\O_{L_4}(-1))$, we get
$h^1(X,\O_X(L_4))=0$. Then from
$$0\to \O_X(L_4)\to \O_X(L_4+L_3)\to \O_{L_3}(-1)\to 0$$
we get $h^1(X,\O_X(L_4+L_3))=0$.
Now from
$$0\to \O_X(L_4+L_3)\to \O_X(L_4+L_3+L_2)\to \O_{L_2}\to 0$$
we get $h^1(X,\O_X(L_4+L_3+L_2))=0$.
Note that $L_4+L_3+L_2=3L-E_5-\cdots-E_{12}$. Now blow up $p_1$ to get $Y$.
Then from
$$0\to \O_Y(3L-E_5-\cdots-E_{12})\to \O_Y(4L-E_1-\cdots-E_{12})\to \O_{L_1}\to 0$$
we get $h^1(Y,\O_Y(4L-E_1-\cdots-E_{12}))=0$ and hence $h^0(Y,\O_Y(4L-E_1-\cdots-E_{12}))=3$,
so $\dim[I(Z)]_4=3$. It's easy to check that the three quartics (namely 
$x^2y^2(x^2-y^2)$, 
$x^2z^2(x^2-z^2)$ and 
$y^2z^2(y^2-z^2)$) 
given by the four $F_3$-lines through each of the coordinate vertices $p_0, p_1, p_2$
are linearly independent and so give a basis of $[I(Z)]_4$. Note that they all vanish at all 13 $F_3$-points of $\P^2$.
Thus every element of $[I(Z)^2]_8$ vanishes at all 13 points.
But $F(p_0)\neq0$, so $F\not\in [I(Z)^2]_8$. Hence $I(3Z)\not\subseteq I(Z)^2$.

\begin{figure}[h]
\definecolor{uuuuuu}{rgb}{0.0,0.0,0.0}
\definecolor{xdxdff}{rgb}{0.0,0.0,0.0}
\definecolor{qqqqff}{rgb}{0.0,0.0,0.0}
\begin{tikzpicture}[line cap=round,line join=round,>=triangle 45,x=1.5cm,y=1.5cm]
\clip(-2.893892758828592,1.5255102359774015) rectangle (3.487225599963166,5.4103663589145405);
\draw (-1.74,4.96)-- (-1.74,1.9);
\draw (-1.74,1.9)-- (1.48,1.88);
\draw (1.48,1.88)-- (-1.74,4.96);
\draw (-1.74,4.96)-- (-0.8197870534680965,1.894284391636448);
\draw (-1.74,4.96)-- (0.12017668389784739,1.8884461075534298);
\draw (1.48,1.88)-- (-1.74,2.7);
\draw (1.48,1.88)-- (-1.7399999999999998,3.58);
\begin{scriptsize}
\draw [fill=qqqqff] (-1.74,4.96) circle (1.5pt);
\draw[color=qqqqff] (-1.661676799889494,5.113314297384579) node {$p_4$};
\draw [fill=white] (-1.74,1.9) circle (1.5pt);
\draw[color=qqqqff] (-1.85,1.75) node {$p_0$};
\draw [fill=qqqqff] (1.48,1.88) circle (1.5pt);
\draw[color=qqqqff] (1.561888164120825,2.0327744000368337) node {$p_1$};
\draw[color=black] (-0.24242806146856846,1.75) node {$L_4$};
\draw[color=black] (-0.05539528198674106,3.573044348710707) node {$L_1$};
\draw [fill=xdxdff] (-1.7399999999999998,3.58) circle (1.5pt);
\draw[color=xdxdff] (-1.9,3.6) node {$p_7$};
\draw [fill=xdxdff] (-1.74,2.7) circle (1.5pt);
\draw[color=xdxdff] (-1.9,2.7) node {$p_{10}$};
\draw [fill=xdxdff] (-0.8197870534680965,1.894284391636448) circle (1.5pt);
\draw[color=xdxdff] (-0.7375148306851703,1.7) node {$p_{12}$};
\draw [fill=xdxdff] (0.12017668389784739,1.8884461075534298) circle (1.5pt);
\draw[color=xdxdff] (0.19764906672396657,1.7) node {$p_{11}$};
\draw[color=black] (-0.5394801229985295,2.25) node {$L_3$};
\draw[color=black] (0.16464328210952645,2.76990358975933) node {$L_2$};
\draw [fill=uuuuuu] (-1.2477718782915457,3.320128010278146) circle (1.5pt);
\draw[color=uuuuuu] (-1.166590030672892,3.474026994867386) node {$p_6$};
\draw [fill=uuuuuu] (-0.5114396015606132,2.9313811561034298) circle (1.5pt);
\draw[color=uuuuuu] (-0.42946084095039583,3.088959507698918) node {$p_5$};
\draw [fill=uuuuuu] (-1.0054873443557129,2.512950193283132) circle (1.5pt);
\draw[color=uuuuuu] (-0.9245476101669977,2.6708862359160097) node {$p_9$};
\draw [fill=uuuuuu] (-0.12173546345690864,2.287895366470393) circle (1.5pt);
\draw[color=uuuuuu] (.05,2.439845743614929) node {$p_{8}$};
\draw [fill=xdxdff] (-0.5334649555774924,3.8059230009871667) circle (1.5pt);
\draw[color=xdxdff] (-0.4514646973600226,3.958111835879175) node {$p_3$};
\draw [fill=xdxdff] (0.2948075024679171,3.0136623889437315) circle (1.5pt);
\draw[color=xdxdff] (0.3736799180009806,3.165973005132612) node {$p_2$};
\end{scriptsize}
\end{tikzpicture}
\caption{The 13 ${\mathbb F}$-points $p_i$ of $\P^2$ over a field $\KK$ of characteristic 3.}
\label{12ptsChar3}
\end{figure}
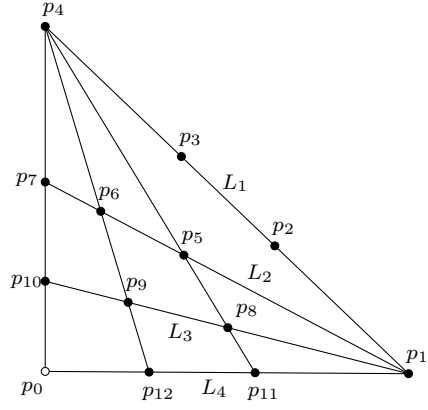
\end{example}

\begin{exercise}\label{FermatResurg}
Let $Z\subseteq\P^2$ be the 12 points of the Fermat arrangement for $n=3$.
Let $I=I(Z)$. It is easy to check by computer that
$\alpha(I)=\omega(I)=4$ and we know from above that $I^{(3)}\not\subseteq I^2$.
Hence $\widehat{\rho}(I)=\frac{4}{3}<\rho(I)$.
\end{exercise}

\SolnEater{\vskip\baselineskip 
\noindent{\it Details}: We know $\widehat{\alpha}(I)=3$ by \Exercise\ \ref{Fermatn>2}.
We know $\frac{\alpha(I)}{\widehat{\alpha}(I)}\leq \widehat{\rho}(I)\leq \frac{\omega(I)}{\widehat{\alpha}(I)}$,
by \Exercise\ \ref{asymresurgEx}. Thus $\widehat{\rho}(I) = \frac{4}{3}$.
But $I^{(3)}\not\subseteq I^2$ implies by definition of $\rho(I)$ that
$\frac{3}{2}\leq \rho(I)$.
\newline\qedsymbol\vskip\baselineskip}

In fact, by \cite[Theorem 2.1]{refDHNSST}, we have $\rho(I)=\frac{3}{2}$ and $\widehat{\rho}(I)=\frac{n+1}{n}$ for the ideal $I$ of the 
$n^2+3$ points of the Fermat arrangement for $n\geq 3$.

If there are nontrivial complex line arrangements in addition to the Fermat, Klein and Wiman with $t_2=0$,
it seems reasonable to expect that the ideal $I$ of their singular points would give additional counterexamples to 
$I^{(3)}\subseteq I^2$.

Another way to address optimality of $I(rNZ)\subseteq I(Z)^r$ is to make the right hand side of the containment
smaller. This led to the following conjecture \cite{2refHaHu}:

\begin{conjecture}\label{HHConj}
Let $Z\subseteq \P^N$ be a fat point subscheme and let $M=(x_0,\ldots,x_N)$ for $R=\KK[\P^N]=\KK[x_0,\ldots,x_N]$. 
Then $I(NrZ)\subseteq M^{Nr-r}I(Z)^r$ for all $r\geq 1$. 
\end{conjecture}

So far this conjecture remains open in all characteristics. The motivation was a conjecture of Chudnovsky \cite{refCh},
aimed at improving the bound of Waldschmidt and Skoda (see \Exercise\ \ref{Fekete}(f)):

\begin{conjecture}\label{ChudConj}
Let $Z\subseteq \P^N$ be a fat point subscheme (in the original statement, $Z$ was reduced).
Then 
$$\frac{\alpha(I(Z))+N-1}{N}\leq\widehat{\alpha}(I(Z)).$$
\end{conjecture}

\begin{exercise}\label{HaHuImpliesChudConj}
Conjecture \ref{HHConj} implies Conjecture \ref{ChudConj}.
(Note: One can mimic \Exercise\ \ref{Fekete}(f).) 
\end{exercise}

\SolnEater{\vskip\baselineskip 
\noindent{\it Details}: 
From $I(NrZ)\subseteq M^{Nr-r}I(Z)^r$ we get
$$ \alpha(I(NrZ))\geq \alpha(I(Z)^r)+\alpha(M^{Nr-r})=r\,\alpha(I(Z))+Nr-r,$$
hence 
$$ \frac{\alpha(I(NrZ))}{Nr}\geq \frac{r\alpha(I(Z))+Nr-r}{Nr}=\frac{\alpha(I(Z))+N-1}{N}.$$
The result follows by taking limits as $r\to\infty$.
\newline\qedsymbol\vskip\baselineskip}

When $N=2$ and $Z=p_1+\cdots+p_s$ is reduced, Conjecture \ref{ChudConj} 
is a result of Chudnovksy \cite{refCh}; see \cite{2refHaHu} for one proof.
Here's a more geometric proof that is probably along the lines of how Chudnovksy did it (but he wasn't very explicit
in his paper). Let $I=I(Z)$ and let $a_m=\alpha(I(mZ))$. 
Since $[I(Z)]_{a_1-1}=0$, we can pick a subscheme $U=q_1+\cdots+q_r$ of $Z$ (so $\{q_1,\ldots,q_r\}\subseteq \{p_1,\ldots,p_s\}$)
where $r=\dim \KK[x,y,z]_{a_1-1}=\binom{a_1+1}{2}$ such that $\alpha(I(U))=a_1$;
i.e., $U$ imposes independent conditions on points in degree $t = a_1-1$ (and hence also in degrees $t\geq a_1$
but we don't need this). Thus for every $q_i$, there is a form $F_i\in I(U)_{a_1-1}$
with $F_i(q_i)\neq0$ but $F_i(q_j)=0$ for $j\neq i$. This means $I(U)_{a_1}$ has no base points other than $U$.
(Let $p$ be any point not in $U$. Then there is a form $P\in \KK[x,y,z]_{a_1-1}$ with $P(p)\neq0$. But the $F_i$
give a basis of $\KK[x,y,z]_{a_1-1}$, so we have $F_i(p)\neq0$ for some $i$. Now pick a linear form $B$ vanishing at $p_i$
but not at $p$. Then $FB\in I(U)_{a_1}$ but $FB(p)\neq0$. Alternatively, conclude that $\reg(I(U))=a_1$ and use Fact \ref{Dubreil} that 
$I(U)$ is then generated in degree $a_1$, and hence has no base points other than $U$.)
In particular, $I(U)_a$ has 0-dimensional zero locus. Thus, given a nonzero $F\in I(mZ)_{a_m}$, we can pick
a nonzero $G\in I(U)_{a_1}$ with no components in common with $F$. Hence, by Bezout's Theorem,
we have $a_ma=\deg(F)\deg(G)\geq m|U|=m\binom{a+1}{2}$; i.e.,
$\frac{a_m}{m}\geq \frac{a+1}{2}$, so $\widehat{\alpha}(I(Z))\geq \frac{\alpha(I(Z))+1}{2}$.

\section{A new perspective on the SHGH Conjecture}

It is easy to find examples of a fat point subscheme $X$ of $\P^2$ whose points are general
yet $X$ fails to impose independent conditions on the space $V=R_t=(\KK[\P^2])_t$
of all forms of degree $t$.
For example, take $X$ to be a reduced scheme of 5 general points and take $t=1$. The first 3 points get us down to the 0 vector space, so the next two
points do not reduce the dimension any further; i.e., they do not impose additional conditions, so $X$ does not impose independent conditions on linear forms.
This is entirely expected. It is somewhat more unexpected to have a fat point subscheme $X=m_1p_1+\cdots+m_rp_r\subset\P^2$ 
where the points $p_i$ are general and a $t$ where $I(X)_t\neq0$ such that the conditions imposed by $X$ on $V$ are not independent.
In such a situation we will say $X$  {\it unexpectedly} fails to impose independent
conditions on $V$.

\subsection{Conditions imposed by fat points}
\begin{openproblem}\label{SHGHProb} 
Find all degrees $t$ and integers $m_i>0$ such that $X=\sum_im_ip_i\subseteq\P^2$ 
unexpectedly fails to impose independent
conditions on $V=(\KK[\P^2])_t$ when the points $p_i$ are general; i.e., 
$$\dim I(X)_t > \min\Bigg(0, \dim V - \sum_i\Big(\binom{m_i+1}{2}\Big)\Bigg).$$
\end{openproblem} 

The SHGH Conjecture \cite{2refSe, 2refVanc, 2refG, 2refHi} gives a conjectural solution for this.
It says that the following sufficient condition is also necessary.

\begin{exercise}\label{SHGHmotivation}
Suppose we are given a smooth rational surface $X$, an exceptional curve $E$ (i.e., a smooth rational curve with $E^2=-1$),
and a divisor $F$ on $X$. If $h^0(X,\O_X(F))>0$ and $(F+rE)\cdot E\leq -2$
for some $r\geq 1$, then $h^1(X,\O_X(F+rE))>0$.
\end{exercise}

\SolnEater{\vskip\baselineskip 
\noindent{\it Details}: By blowing up additional general points we may assume
$X$ is a blow up of points of $\P^2$. Let $L$ be the pullback of a general line.
Then $-K_X\cdot L<0$.
Since $F\cdot L\geq 0$ we have $h^2(X, \O_X(F+rE))=0$ by duality. Look at
$$0\to \O_X(F+rE-E)\to \O_X(F+rE)\to \O_E((F+rE)\cdot E)\to0.$$
If $(F+rE)\cdot E\leq -2$, then $h^1(E,\O_E((F+rE)\cdot E))>0$.
But $h^0(X,\O_X(F+rE-E))>0$, so $(F+rE-E)\cdot L\geq 0$ (since $L$ is nef),
hence $h^2(X,\O_X(F+rE-E))=h^0(X,\O_X(K_X-F-(rE-E)))=0$ (since 
$(K_X-F-(rE-E))\cdot L<0$);
thus $h^1(X,\O_X(F+rE))>0$.
\newline\qedsymbol\vskip\baselineskip}

If $F$ is a divisor on a surface $S$ obtained by blowing up $s$ general points $p_i$ of $\P^2$, the SHGH Conjecture says:

\begin{conjecture}\label{SHGHConj} 
If $h^0(S,\O_S(F))>0$ and $h^1(S,\O_S(F))>0$, then
there is an exceptional curve $E$ on $S$ such that $F\cdot E\leq -2$.
\end{conjecture}

If this is true, then standard techniques allow one to compute $h^0(S,\O_S(F))$ exactly for any divisor $F$ on $S$.
For simplicity it is best to assume $s>2$ (the procedure to be described below involves a reduction which can run 
into some special cases when $s=1,2$).

Here is the idea: Given $F=tL-\sum_im_iE_i$ for $s$ general points $p_i$, there is an algorithmic procedure
which gives either a nef divisor $H$ such that $H\cdot F<0$ (and hence $h^0(S,\O_S(F))=0$), or which gives
a Zariski-like decomposition $F=A+\sum_ic_iC_i$ such that $A\cdot E\geq 0$ for all exceptional curves $E$,
and the $C_i$ are exceptional with $c_i\geq 0$, $A\cdot C_i=0$ and $C_i\cdot C_j=0$ for all $i\neq j$.

In this case $h^0(S,\O_S(F))=h^0(S,\O_S(A))\geq \frac{A^2-K_SA}{2}+1$; the content of the
SHGH Conjecture is that this is an equality.

\begin{exercise}\label{corOfSHGH}
Assuming the SHGH Conjecture, if $C^2<0$ for a reduced irreducible curve $C$ on a blow up $S$
of $\P^2$ at general points $p_i$, then $C$ is an exceptional curve.
\end{exercise}

\SolnEater{\vskip\baselineskip 
\noindent{\it Details}: Since $S$ is a smooth projective rational surface, we have  $h^1(S,\O_S)=h^2(S,\O_S)=0$.
Look at
$$0\to \O_S\to \O_S(C)\to \O_C(C^2)\to0.$$
If $C$ is not smooth and rational with $C^2=-1$, then $h^1(C,\O_C(C^2))>0$ and hence $h^1(S,\O_S(C))>0$.
But then $E\cdot C\geq 0$ for all exceptional curves $E$, contradicting the SHGH Conjecture.
\newline\qedsymbol\vskip\baselineskip}

\noindent{\bf A sample more general problem}: Find examples of reduced point schemes 
$Z\subseteq \P^2$ and $t$ such that $\dim I(Z)_t\leq 3$ but for every $p\in \P^2$ there is a curve of degree $t$ containing $Z$ 
and singular at $p$.

\vskip\baselineskip
We will relate this to the following open problem (this and all that follows is based on \cite{5refCHMN},
which in turn was motivated by the paper \cite{5refDIV}):

\begin{openproblem}\label{CHMNProb} 
Find all $t$ and integers $m_i>0$ and all fat point subschemes $Z=\sum_ja_jq_j$ such that
$X=\sum_im_ip_i$ {\it unexpectedly} fails to impose independent conditions on $V=I(Z)_t$
where the points $p_i\in\P^2$ are general; i.e.,
\begin{equation}\label{unexpIneq}
\dim (I(X)_t \cap V) > \min\Big(0, \dim V - \sum_i\binom{m_i+1}{2}\Big).
\end{equation}
\end{openproblem} 

\begin{example}\label{CHMNExample}
Each of the following give examples of an $X$ and $Z$ where $X$ unexpectedly fails to impose independent conditions on $V=I(Z)_t$.
\begin{enumerate}
\item[(a)] If $Z=0$, this is just is just a case of Problem \ref{SHGHProb}, so is solved by the SHGH Conjecture.

\item[(b)] If $Z$ consists of fat points where the points are general, this also in principle is solved by the SHGH Conjecture.

\item[(c)] If $Z$ is reduced and consists of the 7 points of the Fano plane (so $\chr(\KK)=2$),
then this is an example of both the sample problem and Problem \ref{CHMNProb},
where we have $V=I(Z)_t$, $t=3$, $X=2p$, $\dim V=3$. Being singular at $p$ imposes 3 conditions, so we expect no curve,
but for every point $p$ there is a cubic through $Z$ singular at $p$ (specifically $F=\alpha^2yz(y + z)+\beta^2xz(x + z)+\gamma^2xy(x + y)$
vanishes at the 7 points and is singular at $p=(\alpha,\beta,\gamma)$).

\item[(d)] Take $X=mp$ and let $Z$ be reduced, consisting of the points dual to the $3n$ Fermat lines where $n\geq 5$,
$n+1\leq m\leq 2n-4$ and $t=m+1$. (Its splitting type, defined below, is $(n+1,2n-2)$.)

\item[(e)] Take $X=mp$ and let $Z$ be reduced, consisting of the points dual to the Klein lines;
$m=9$ and $t=10$. (Its splitting type is $(9,11)$.)

\item[(f)] Take $X=mp$ and let $Z$ be reduced, consisting of  the points dual to the Wiman lines;
$19\leq m\leq 23$ and $t=m+1$. (Its splitting type is $(19,25)$.)
\end{enumerate}
\end{example}

It is not obvious how to find such examples. Doing so uses some theory.
Let $Z=q_1+\cdots+q_r$ be a reduced point scheme in ${\bf P}^2$, $\ell_j$ the linear form dual to $q_j$,
and let $F=\ell_1\cdots\ell_r$. Now assume that $\chr(\KK)$ does not divide $\deg(F)$, and let
${\mathcal J}_F$ be the Jacobian sheaf (i.e., the sheafification of the ideal $(F_x, F_y, F_z)$
generated by the partial derivatives of $F$), and $\mathcal D$ the syzygy bundle; i.e., the sheaf defined by the exact sheaf sequence
$$ 0\to {\mathcal D}\to {\mathcal O}^3\to {\mathcal J}_F(r-1)\to 0.$$
Restricted to a general line $L$ we get ${\mathcal D}|_L= {\mathcal O}_L(-a)\oplus {\mathcal O}_L(-b)$ where $a+b=d-1$, $a\leq b$.
Call $(a,b)$ the splitting type of $Z$.

To state the results, it's convenient to introduce a quantity $t_Z$, defined as the least $j$ such that  $\dim I(Z)_{j+1} > \binom{j+1}{2}$.

\begin{theorem}\label{mainThm1} 
Let $Z$ be a reduced 0-dimensional subscheme of $\P^2$ 
of splitting type $(a_Z,b_Z)$. Then
\eqref{unexpIneq} holds for some degree $t$ with $X=m_1p_1$ and $m_1=t-1$
if and only if $a_Z < t_Z$. 
Furthermore, in this case the degrees $t$ for which \eqref{unexpIneq} holds 
with $X=m_1p_1$ and $m_1=t-1$
are precisely those in the range $a_Z <  t < b_Z$. For each $t$ in this range,
$\dim I(X+Z)_t=1$ (so there is a unique curve of degree $t$ containing $Z$
with $p_1$ being a point of order $t-1$; this curve is 
denoted $C_t(Z)$ and is said to be the unexpected curve for $Z$ of degree $t$).
\end{theorem}

Here's another version:

\begin{theorem}\label{CHMNthm}
Given a reduced 0-dimensional subscheme $Z\subseteq \P^2$ 
of splitting type $(a,b)$ and $X=mp$ for a general point $p\in\P^2$, then
\eqref{unexpIneq} holds in degree $t=m+1$ if and only if 
\begin{enumerate}
\item[(a)] $a\leq m\leq b-2$ and
\item[(b)] $t_Z+1\geq \hbox{reg}(I(Z))$.
\end{enumerate}
\end{theorem}

\begin{proof} See \cite{5refCHMN}. The proof uses ideas of \cite{5refFV} to relate 
syzygies and singular curves via the duality between points and lines on the projective plane.
\end{proof}

Here's an example run using Macaulay2 \cite{5refGS}.
\begin{example}\label{M2example}
Let's verify an instance of Theorem \ref{mainThm1}.
Consider the Fermat line arrangement for $n=5$.
The form defining the lines is $F=(x^5-y^5)(x^5-z^5)(y^5-z^5)$.
The scheme $Z$ of points dual to the 15 lines has ideal
$I(Z)=(x^5+y^5+z^5,xyz)$. First we compute the splitting type.
In this case we do not need to restrict to a general line, since
the syzygy bundle is free, so we can read the splitting off 
directly from the first syzygy module in a minimal free resolution of
the Jacobian ideal $J=(F_x,F_y,F_z)$.

\begin{verbatim}
i1 : R=QQ[x,y,z];

i2 : F=(x^5-y^5)*(x^5-z^5)*(y^5-z^5);

i3 : J=ideal(jacobian(ideal(F)));

i4 : betti res J

            0 1 2
o4 = total: 1 3 2
         0: 1 . .
         1: . . .
         .
         .
         .
        11: . . .
        12: . . .
        13: . 3 .
        14: . . .
        15: . . .
        16: . . .
        17: . . .
        18: . . 1
        19: . . .
        20: . . 1
\end{verbatim}
We see that $J$ has three generators of degree 14, as expected, and the generators
of the syzygy module have degrees 6 and 8, giving the splitting type $(6,8)$,
in agreement with Example \ref{CHMNExample}(d).

It is possible to pick a generic point for $p$. We can take $p=(A,B,1)$
where $A$ and $B$ are variables, but this requires working over the field $\KK=\Q(A,B)$.
The Macaulay2 commands for this are ${\tt \KK=\operatorname{frac}(\Q[A,B])}$, ${\tt R=\KK[x,y,z]}$.
But working over $\KK=\Q(A,B)$ entails a noticeable performance penalty, so we
pick a random point for $p$ instead. 

\begin{verbatim}
i1 : R=QQ[x,y,z];

i2 : p=ideal(random(1,R), random(1,R));

i3 : Z=ideal(x^5+y^5+z^5,x*y*z);

i4 : I5=intersect(p^5,Z);

i5 : betti res I5

            0 1 2
o5 = total: 1 7 6
         0: 1 . .
         1: . . .
         2: . . .
         3: . . .
         4: . . .
         5: . . .
         6: . 6 4
         7: . 1 2
-- Note I5 has no element of degree 6.

i6 : I6=intersect(p^6,Z);

i7 : betti res I6

            0 1 2
o7 = total: 1 7 6
         0: 1 . .
         1: . . .
         2: . . .
         3: . . .
         4: . . .
         5: . . .
         6: . 1 .
         7: . 6 5
         8: . . 1
-- So the least m for which I(Z+mp) has an element of degree m+1 is m=6;
-- i.e., a_Z = 6.

i8 : for i from 2 to 8 do print {i,hilbertFunction(i,R)-hilbertFunction(i,I6),
     hilbertFunction(i,R)-hilbertFunction(i,Z)}
{2, 0, 0}
{3, 0, 1}
{4, 0, 3}
{5, 0, 7}
{6, 0, 13}
{7, 1, 21}
{8, 9, 30}

-- 6p should impose 21 conditions on I(Z)_7, making I(Z+6p)_7 = 0
-- but instead we see dim I(Z+6p)_7 = 1, so 6p unexpectedly
-- fails to impose independent conditions on I(Z)_7.

i9 : I7=intersect(p^7,Z);

i10 : for i from 2 to 9 do print {i,hilbertFunction(i,R)-hilbertFunction(i,I7),
      hilbertFunction(i,R)-hilbertFunction(i,Z)}
{2, 0, 0}
{3, 0, 1}
{4, 0, 3}
{5, 0, 7}
{6, 0, 13}
{7, 0, 21}
{8, 2, 30}
{9, 12, 40}

-- 7p should impose 28 conditions on I(Z)_8, making I(Z+6p)_8 = 2
-- and it is, so 7p imposes independent conditions on I(Z)_8.
\end{verbatim}

\end{example}

\begin{exercise}\label{B3ex2}
Let $Z$ be the 9 points as shown in Figure \ref{B3config}.
Let $p$ be a general point.
Then $Z$ has an unexpected irreducible quartic $C_4(Z)$, so it has a triple point
at $p$; $C_4(Z)$ is shown in Figure \ref{UnexpCurveDeg4}.
(Note: Assume the four general points in the figure, shown in black,
are $(0,0,1)$, $(0,1,0)$, $(1,0,0)$, $(1,1,1)$. The other 5 points then become
$(0,1,1)$, $(1,0,1)$, $(1,1,0)$, $(-1,1,0)$, $(1,1,2)$. Using a computer one can compute the splitting type
and $t_Z$, and apply Theorem \ref{mainThm1}.
Then using Bezout one can verify irreducibility and uniqueness.)
\end{exercise}

\SolnEater{\vskip\baselineskip 
\noindent{\it Details}: 
Given the choice of coordinates in the note, the seven lines become
$A=x+y-z$, $B=z$, $C=x-y$, $D=y$, $E=y-z$, $F=x-z$, $G=x$. 
The product of the linear forms dual to the nine points is
$f=xyz(x+y+z)(y+z)(x+z)(x+y)(x-y)(x+y+2z)$.
Using Macaulay2, we find the splitting type is $(3,5)$, and $t_Z=5$.
Thus, by Theorem \ref{mainThm1}, there is a quartic through the points of 
$Z$ with a triple point at $p$.

Let $Q$ be a quartic through $Z$ with a triple point at $p$. If $C$ were not irreducible,
then it has an irreducible component $T$ of degree at most 3 through $p$. 
Let $m=\mult_p(T)$. If $\deg(T)<3$, then $m=1$. If $\deg(T)=3$, then $m=2$,
since otherwise the other component of $Q$ would be a line vanishing twice at $p$.
Thus if $\deg(T)=3$, the other component is a line through $p$. This line can go through
at most one other point of $Z$ (since $p$ is general), so $T$ goes through 8 points of $Z$,
but any 8 points of $Z$ includes a set of 4 collinear points, so $T$ could not be irreducible.

If $\deg(T)=2$, the other component of $Q$ is a conic vanishing twice at $p$, hence consists
of two lines through $p$. These two lines can contain at most two points of $Z$,
hence $T$ contains at least 7 points of $Z$, so contains at least 3 collinear points
(since there are three sets of 4 collinear points in $Z$, removing any two points still leaves 
at least one set of 3 collinear points). Thus $T$ could not be irreducible.

Thus $\deg(T)<4$ implies all components of $Q$ must be lines, with three of them containing $p$. Therefore these
three can contain at most 3 points of $Z$ total, and the other line can contain at most 4
(since $Z$ does not contain any 5 collinear points). Thus $Q$ is not a union of lines.
It is therefore irreducible.

If it were not unique, there would be two different quartics meeting at least 18 times, at least 9 times at $p$ and
at least once each at the 9 points of $Z$. Since they have degree 4 this is impossible.

We can even write down $Q$ explicitly. We first check that $\dim [(I(Z)]_4=6$.
Blow up the 5 points $(001), (010), (100), (111), (011)$ to get the surface $X$. 
Let $L$ be the pullback of a line. Let $e$ be the sum of the
three exceptional curves for the three coordinate vertices.
Let $e'$ be the sum of the
two exceptional curves for the points $(111)$ and $(110)$, and let
$a\sim L-e'$ be the proper transform of the line through these two points.
From
$$0\to \O_X(2L-e)\to \O_X(3L-e-e')\to \O_a(1)\to0$$
we see that $h^1(X,\O_X(3L-e-e'))=0$.
Now blow up the remaining 4 points of $Z$ (which we note are collinear, since they
are on the line $x+y-z$) and
let $e''$ be the sum of their exceptional curves.
Call the new surface $Y$ and let $b\sim L-e''$ be the proper transform
of the line through these 4 points.
From
$$0\to \O_X(3L-e-e')\to \O_X(4L-e-e'-e'')\to \O_b\to0$$
we see that $h^1(X,\O_X(4L-e-e'-e''))=0$.
Thus $\dim [(I(Z)]_4=6$. 

We can now give a basis for $[(I(Z)]_4$, namely
$ABCD, ABCF, ABCG, ABFG, ACDF, BCDG$.
Using the factorization it is easy to check that each of these is a quartic 
vanishing on $Z$, and that they are linearly independent, hence a basis for
$[(I(Z)]_4$. Given a general point $p=(u,v,w)$, we want to find which linear combination
of these basis elements is in the ideal $I(p)^3=(wx-uz,wy-vz)^3$. Using Macaulay2 we find that it is
$(-4u^3+6u^2w-2w^3)ABCD+
(-2u^3+6u^2v)ABCF+
(-6uv^2-2v^3+6v^2w)ABCG+
(2u^3-6u^2v+6uv^2-2v^3)ABFG
-2w^3ACDF+
(2u^3+6u^2v+6uv^2+2v^3-6u^2w-12uvw-6v^2w+6uw^2+6vw^2-2w^3)BCDG$.
\newline\qedsymbol\vskip\baselineskip}

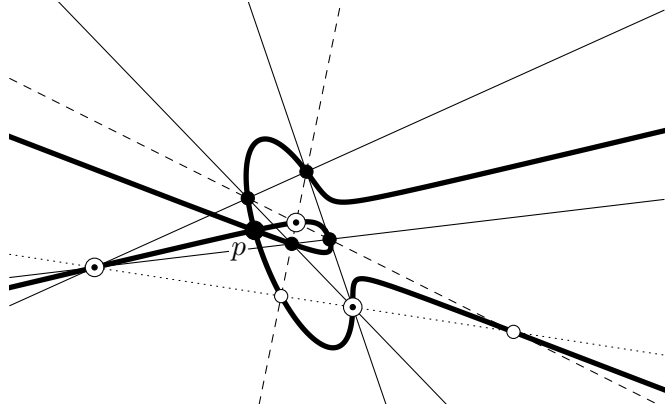
\begin{figure}[h]
\begin{tikzpicture}[line cap=round,line join=round,>=triangle 45,x=0.75cm,y=0.75cm]
\clip(-3.86,-1) rectangle (7.74,6.12);
\draw [domain=-3.86:7.74] plot(\x,{(--2.5868--0.47*\x)/1.04});
\draw [domain=-3.86:7.74] plot(\x,{(--1.1416--0.08*\x)/0.67});
\draw [domain=-3.86:7.74] plot(\x,{(--2.3586-0.81*\x)/0.78});
\draw [domain=-3.86:7.74] plot(\x,{(--2.9592-1.2*\x)/0.41});
\draw [dashed,domain=-3.86:7.74] plot(\x,{(--0.9808-1.28*\x)/-0.26});
\draw [dashed,domain=-3.86:7.74] plot(\x,{(-4.1053--0.73*\x)/-1.45});
\draw [dotted,domain=-3.86:7.74] plot(\x,{(--4.84997347092839-0.7049481606502077*\x)/4.576844052716465});
\draw[line width=2.0pt, color = black] (1.73109,2.08) -- 
(1.72719,2.08544) -- 
(1.7232,2.09085) -- 
(1.7191,2.09622) -- 
(1.7149,2.10156) -- 
(1.71059,2.10685) -- 
(1.70616,2.11211) -- 
(1.70161,2.11732) -- 
(1.69694,2.1225) -- 
(1.69212,2.12762) -- 
(1.68716,2.13271) -- 
(1.68205,2.13774) -- 
(1.67677,2.14272) -- 
(1.67132,2.14765) -- 
(1.66568,2.15252) -- 
(1.65984,2.15733) -- 
(1.65378,2.16208) -- 
(1.64749,2.16676) -- 
(1.64095,2.17137) -- 
(1.63413,2.1759) -- 
(1.62702,2.18035) -- 
(1.61959,2.18471) -- 
(1.6118,2.18898) -- 
(1.60363,2.19314) -- 
(1.59504,2.1972) -- 
(1.58598,2.20112) -- 
(1.57641,2.20492) -- 
(1.56626,2.20857) -- 
(1.55547,2.21205) -- 
(1.54396,2.21535) -- 
(1.53165,2.21845) -- 
(1.51843,2.22132) -- 
(1.50417,2.22394) -- 
(1.48873,2.22626) -- 
(1.47192,2.22824) -- 
(1.45352,2.22984) -- 
(1.43327,2.23098) -- 
(1.41084,2.2316) -- 
(1.38581,2.23158) -- 
(1.35765,2.23081) -- 
(1.32566,2.22911) -- 
(1.28896,2.22628) -- 
(1.24632,2.22203) -- 
(1.19609,2.21595) -- 
(1.1359,2.20749) -- 
(1.06231,2.19583) -- 
(0.970083,2.17971) -- 
(0.85084,2.15713) -- 
(0.690285,2.1247) -- 
(0.461876,2.07606) -- 
(0.110139,1.998) -- 
(-0.503542,1.85745) -- 
(-1.84998,1.54208) -- 
(-7.21482,0.26811); 
\draw[line width=2.0pt, color = black] (19.828,6.72504) -- 
(6.23062,3.48717) -- 
(4.29817,3.03198) -- 
(3.52468,2.85325) -- 
(3.10736,2.75948) -- 
(2.84577,2.70286) -- 
(2.66611,2.66577) -- 
(2.53485,2.64022) -- 
(2.43456,2.62205) -- 
(2.35528,2.60888) -- 
(2.29091,2.59927) -- 
(2.2375,2.59226) -- 
(2.19238,2.58723) -- 
(2.15368,2.58372) -- 
(2.12006,2.58142) -- 
(2.09053,2.58008) -- 
(2.06433,2.57954) -- 
(2.04089,2.57964) -- 
(2.01975,2.58029) -- 
(2.00055,2.5814) -- 
(1.98301,2.5829) -- 
(1.96689,2.58473) -- 
(1.952,2.58685) -- 
(1.93819,2.58922) -- 
(1.92532,2.59181) -- 
(1.91327,2.5946) -- 
(1.90195,2.59755) -- 
(1.89129,2.60065) -- 
(1.88121,2.60389) -- 
(1.87165,2.60725) -- 
(1.86256,2.61071) -- 
(1.85389,2.61428) -- 
(1.84561,2.61793) -- 
(1.83768,2.62166) -- 
(1.83006,2.62546) -- 
(1.82274,2.62933) -- 
(1.81568,2.63326) -- 
(1.80886,2.63724) -- 
(1.80227,2.64127) -- 
(1.79589,2.64535) -- 
(1.7897,2.64947) -- 
(1.78368,2.65363) -- 
(1.77783,2.65783) -- 
(1.77213,2.66206) -- 
(1.76657,2.66632) -- 
(1.76114,2.67061) -- 
(1.75583,2.67493) -- 
(1.75064,2.67927) -- 
(1.74556,2.68364) -- 
(1.74057,2.68803) -- 
(1.73568,2.69244) -- 
(1.73087,2.69686) -- 
(1.72615,2.70131) -- 
(1.72151,2.70577) -- 
(1.71693,2.71025) -- 
(1.71243,2.71474) -- 
(1.70798,2.71925) -- 
(1.7036,2.72377) -- 
(1.69928,2.7283) -- 
(1.695,2.73285) -- 
(1.69078,2.7374) -- 
(1.68661,2.74197) -- 
(1.68248,2.74655) -- 
(1.67839,2.75113) -- 
(1.67434,2.75573) -- 
(1.67034,2.76033) -- 
(1.66636,2.76495) -- 
(1.66242,2.76957) -- 
(1.65851,2.7742) -- 
(1.65463,2.77883) -- 
(1.65078,2.78348) -- 
(1.64696,2.78813) -- 
(1.64316,2.79278) -- 
(1.63938,2.79745) -- 
(1.63563,2.80212) -- 
(1.63189,2.80679) -- 
(1.62818,2.81147) -- 
(1.62449,2.81616) -- 
(1.62081,2.82085) -- 
(1.61715,2.82555) -- 
(1.61351,2.83025) -- 
(1.60988,2.83496) -- 
(1.60626,2.83967) -- 
(1.60266,2.84439) -- 
(1.59906,2.84911) -- 
(1.59548,2.85384) -- 
(1.59191,2.85857) -- 
(1.58835,2.86331) -- 
(1.5848,2.86805) -- 
(1.58126,2.87279) -- 
(1.57772,2.87754) -- 
(1.57419,2.88229) -- 
(1.57067,2.88705) -- 
(1.56715,2.89181) -- 
(1.56363,2.89658) -- 
(1.56012,2.90135) -- 
(1.55662,2.90612) -- 
(1.55312,2.9109) -- 
(1.54962,2.91568) -- 
(1.54612,2.92046) -- 
(1.54263,2.92525) -- 
(1.53913,2.93004) -- 
(1.53564,2.93484) -- 
(1.53215,2.93964) -- 
(1.52866,2.94445) -- 
(1.52517,2.94926) -- 
(1.52167,2.95407) -- 
(1.51818,2.95888) -- 
(1.51468,2.9637) -- 
(1.51119,2.96853) -- 
(1.50769,2.97335) -- 
(1.50418,2.97819) -- 
(1.50068,2.98302) -- 
(1.49717,2.98786) -- 
(1.49366,2.9927) -- 
(1.49014,2.99755) -- 
(1.48662,3.0024) -- 
(1.4831,3.00725) -- 
(1.47956,3.01211) -- 
(1.47603,3.01697) -- 
(1.47249,3.02184) -- 
(1.46894,3.0267) -- 
(1.46539,3.03158) -- 
(1.46183,3.03645) -- 
(1.45826,3.04133) -- 
(1.45469,3.04622) -- 
(1.45111,3.05111) -- 
(1.44752,3.056) -- 
(1.44392,3.06089) -- 
(1.44032,3.06579) -- 
(1.4367,3.0707) -- 
(1.43308,3.0756) -- 
(1.42945,3.08051) -- 
(1.42581,3.08543) -- 
(1.42216,3.09035) -- 
(1.4185,3.09527) -- 
(1.41483,3.10019) -- 
(1.41115,3.10512) -- 
(1.40746,3.11005) -- 
(1.40376,3.11499) -- 
(1.40005,3.11993) -- 
(1.39633,3.12487) -- 
(1.3926,3.12982) -- 
(1.38885,3.13477) -- 
(1.38509,3.13973) -- 
(1.38132,3.14468) -- 
(1.37754,3.14964) -- 
(1.37374,3.15461) -- 
(1.36993,3.15958) -- 
(1.36611,3.16455) -- 
(1.36228,3.16952) -- 
(1.35843,3.1745) -- 
(1.35457,3.17948) -- 
(1.35069,3.18446) -- 
(1.3468,3.18945) -- 
(1.34289,3.19444) -- 
(1.33897,3.19943) -- 
(1.33503,3.20443) -- 
(1.33108,3.20942) -- 
(1.32712,3.21443) -- 
(1.32313,3.21943) -- 
(1.31913,3.22443) -- 
(1.31512,3.22944) -- 
(1.31109,3.23445) -- 
(1.30704,3.23947) -- 
(1.30297,3.24448) -- 
(1.29889,3.2495) -- 
(1.29479,3.25451) -- 
(1.29067,3.25953) -- 
(1.28654,3.26456) -- 
(1.28238,3.26958) -- 
(1.27821,3.2746) -- 
(1.27402,3.27963) -- 
(1.26981,3.28465) -- 
(1.26558,3.28968) -- 
(1.26133,3.29471) -- 
(1.25706,3.29974) -- 
(1.25277,3.30476) -- 
(1.24846,3.30979) -- 
(1.24413,3.31482) -- 
(1.23978,3.31985) -- 
(1.23541,3.32487) -- 
(1.23102,3.3299) -- 
(1.2266,3.33492) -- 
(1.22217,3.33995) -- 
(1.21771,3.34497) -- 
(1.21323,3.34999) -- 
(1.20873,3.35501) -- 
(1.2042,3.36002) -- 
(1.19965,3.36504) -- 
(1.19508,3.37004) -- 
(1.19049,3.37505) -- 
(1.18587,3.38005) -- 
(1.18123,3.38505) -- 
(1.17656,3.39004) -- 
(1.17187,3.39503) -- 
(1.16716,3.40001) -- 
(1.16242,3.40499) -- 
(1.15765,3.40996) -- 
(1.15286,3.41493) -- 
(1.14804,3.41989) -- 
(1.1432,3.42484) -- 
(1.13833,3.42978) -- 
(1.13344,3.43471) -- 
(1.12851,3.43964) -- 
(1.12357,3.44455) -- 
(1.11859,3.44946) -- 
(1.11359,3.45435) -- 
(1.10855,3.45924) -- 
(1.10349,3.46411) -- 
(1.09841,3.46896) -- 
(1.09329,3.47381) -- 
(1.08815,3.47864) -- 
(1.08297,3.48346) -- 
(1.07777,3.48826) -- 
(1.07254,3.49305) -- 
(1.06728,3.49781) -- 
(1.06199,3.50256) -- 
(1.05667,3.5073) -- 
(1.05131,3.51201) -- 
(1.04593,3.5167) -- 
(1.04052,3.52137) -- 
(1.03508,3.52602) -- 
(1.0296,3.53065) -- 
(1.0241,3.53525) -- 
(1.01856,3.53983) -- 
(1.01299,3.54438) -- 
(1.00739,3.5489) -- 
(1.00176,3.55339) -- 
(0.996091,3.55786) -- 
(0.990393,3.56229) -- 
(0.984663,3.56669) -- 
(0.978901,3.57106) -- 
(0.973105,3.57539) -- 
(0.967277,3.57968) -- 
(0.961416,3.58394) -- 
(0.955521,3.58816) -- 
(0.949593,3.59234) -- 
(0.943632,3.59647) -- 
(0.937638,3.60056) -- 
(0.93161,3.60461) -- 
(0.925548,3.6086) -- 
(0.919453,3.61255) -- 
(0.913324,3.61645) -- 
(0.907161,3.62029) -- 
(0.900965,3.62408) -- 
(0.894735,3.62781) -- 
(0.888472,3.63148) -- 
(0.882175,3.63509) -- 
(0.875844,3.63864) -- 
(0.86948,3.64212) -- 
(0.863083,3.64553) -- 
(0.856653,3.64887) -- 
(0.85019,3.65214) -- 
(0.843694,3.65533) -- 
(0.837166,3.65844) -- 
(0.830605,3.66147) -- 
(0.824012,3.66442) -- 
(0.817388,3.66728) -- 
(0.810732,3.67005) -- 
(0.804044,3.67273) -- 
(0.797327,3.67531) -- 
(0.790579,3.67779) -- 
(0.783801,3.68016) -- 
(0.776994,3.68243) -- 
(0.770158,3.68459) -- 
(0.763294,3.68664) -- 
(0.756402,3.68856) -- 
(0.749483,3.69037) -- 
(0.742539,3.69205) -- 
(0.735569,3.69359) -- 
(0.728574,3.69501) -- 
(0.721556,3.69628) -- 
(0.714515,3.69742) -- 
(0.707452,3.6984) -- 
(0.700368,3.69923) -- 
(0.693265,3.6999) -- 
(0.686143,3.70041) -- 
(0.679004,3.70076) -- 
(0.67185,3.70092) -- 
(0.66468,3.70091) -- 
(0.657498,3.70072) -- 
(0.650305,3.70033) -- 
(0.643101,3.69975) -- 
(0.63589,3.69897) -- 
(0.628672,3.69798) -- 
(0.62145,3.69677) -- 
(0.614225,3.69534) -- 
(0.607,3.69369) -- 
(0.599777,3.69179) -- 
(0.592559,3.68966) -- 
(0.585347,3.68727) -- 
(0.578144,3.68463) -- 
(0.570953,3.68173) -- 
(0.563777,3.67855) -- 
(0.556618,3.67509) -- 
(0.54948,3.67135) -- 
(0.542366,3.6673) -- 
(0.535279,3.66296) -- 
(0.528222,3.65829) -- 
(0.521199,3.65331) -- 
(0.514214,3.648) -- 
(0.507271,3.64234) -- 
(0.500374,3.63634) -- 
(0.493527,3.62997) -- 
(0.486734,3.62324) -- 
(0.480001,3.61613) -- 
(0.473331,3.60863) -- 
(0.466729,3.60073) -- 
(0.460202,3.59242) -- 
(0.453753,3.5837) -- 
(0.447389,3.57454) -- 
(0.441116,3.56495) -- 
(0.434938,3.5549) -- 
(0.428863,3.54439) -- 
(0.422896,3.53342) -- 
(0.417043,3.52195) -- 
(0.411313,3.51) -- 
(0.40571,3.49754) -- 
(0.400244,3.48456) -- 
(0.39492,3.47106) -- 
(0.389746,3.45702) -- 
(0.38473,3.44243) -- 
(0.37988,3.42728) -- 
(0.375204,3.41156) -- 
(0.370711,3.39526) -- 
(0.366408,3.37837) -- 
(0.362304,3.36088) -- 
(0.358409,3.34277) -- 
(0.354732,3.32404) -- 
(0.351281,3.30469) -- 
(0.348066,3.28469) -- 
(0.345097,3.26403) -- 
(0.342383,3.24273) -- 
(0.339934,3.22075) -- 
(0.33776,3.1981) -- 
(0.335872,3.17477) -- 
(0.334279,3.15074) -- 
(0.332991,3.12603) -- 
(0.332019,3.10061) -- 
(0.331374,3.07448) -- 
(0.331065,3.04765) -- 
(0.331103,3.0201) -- 
(0.331498,2.99184) -- 
(0.332261,2.96285) -- 
(0.333402,2.93315) -- 
(0.33493,2.90273) -- 
(0.336856,2.8716) -- 
(0.33919,2.83974) -- 
(0.341941,2.80718) -- 
(0.345119,2.77391) -- 
(0.348732,2.73993) -- 
(0.352789,2.70526) -- 
(0.3573,2.66991) -- 
(0.362271,2.63387) -- 
(0.367711,2.59717) -- 
(0.373627,2.55982) -- 
(0.380026,2.52182) -- 
(0.386913,2.4832) -- 
(0.394295,2.44398) -- 
(0.402177,2.40416) -- 
(0.410562,2.36377) -- 
(0.419456,2.32283) -- 
(0.42886,2.28136) -- 
(0.438778,2.2394) -- 
(0.44921,2.19695) -- 
(0.460157,2.15405) -- 
(0.47162,2.11074) -- 
(0.483597,2.06703) -- 
(0.496087,2.02296) -- 
(0.509087,1.97856) -- 
(0.522592,1.93387) -- 
(0.536599,1.88892) -- 
(0.551103,1.84375) -- 
(0.566095,1.7984) -- 
(0.58157,1.75289) -- 
(0.597518,1.70728) -- 
(0.613931,1.6616) -- 
(0.630797,1.61589) -- 
(0.648107,1.5702) -- 
(0.665847,1.52456) -- 
(0.684005,1.47902) -- 
(0.702566,1.43362) -- 
(0.721517,1.3884) -- 
(0.740841,1.34341) -- 
(0.760523,1.29868) -- 
(0.780546,1.25426) -- 
(0.800891,1.21019) -- 
(0.821541,1.16651) -- 
(0.842477,1.12326) -- 
(0.863679,1.08048) -- 
(0.885129,1.03821) -- 
(0.906805,0.996492) -- 
(0.928687,0.955356) -- 
(0.950755,0.91484) -- 
(0.972987,0.874977) -- 
(0.995363,0.835802) -- 
(1.01786,0.797345) -- 
(1.04046,0.759636) -- 
(1.06314,0.722703) -- 
(1.08588,0.686574) -- 
(1.10865,0.651273) -- 
(1.13145,0.616824) -- 
(1.15424,0.583246) -- 
(1.177,0.550561) -- 
(1.19973,0.518786) -- 
(1.22239,0.487937) -- 
(1.24497,0.458028) -- 
(1.26745,0.42907) -- 
(1.28982,0.401075) -- 
(1.31205,0.374051) -- 
(1.33413,0.348004) -- 
(1.35604,0.32294) -- 
(1.37777,0.298861) -- 
(1.3993,0.27577) -- 
(1.42062,0.253666) -- 
(1.44172,0.232549) -- 
(1.46258,0.212414) -- 
(1.4832,0.193257) -- 
(1.50355,0.175072) -- 
(1.52364,0.157853) -- 
(1.54345,0.14159) -- 
(1.56297,0.126274) -- 
(1.58219,0.111894) -- 
(1.60111,0.098438) -- 
(1.61972,0.085894) -- 
(1.63802,0.0742479) -- 
(1.656,0.0634852) -- 
(1.67365,0.0535908) -- 
(1.69097,0.0445488) -- 
(1.70795,0.0363426) -- 
(1.7246,0.0289553) -- 
(1.74092,0.0223693) -- 
(1.75689,0.0165667) -- 
(1.77252,0.0115292) -- 
(1.78781,0.0072383) -- 
(1.80276,0.00367507) -- 
(1.81737,0.000820544) -- 
(1.83164,-0.00134441) -- 
(1.84557,-0.00283904) -- 
(1.85916,-0.00368261) -- 
(1.87241,-0.0038944) -- 
(1.88533,-0.00349365) -- 
(1.89792,-0.00249953) -- 
(1.91018,-0.000931092) -- 
(1.92212,0.00119275) -- 
(1.93373,0.00385326) -- 
(1.94503,0.00703189) -- 
(1.95601,0.0107103) -- 
(1.96668,0.0148705) -- 
(1.97704,0.0194946) -- 
(1.98711,0.024565) -- 
(1.99687,0.0300647) -- 
(2.00634,0.0359766) -- 
(2.01553,0.0422842) -- 
(2.02443,0.0489712) -- 
(2.03305,0.0560218) -- 
(2.0414,0.0634204) -- 
(2.04948,0.0711517) -- 
(2.0573,0.0792011) -- 
(2.06486,0.087554) -- 
(2.07217,0.0961963) -- 
(2.07923,0.105114) -- 
(2.08605,0.114295) -- 
(2.09263,0.123725) -- 
(2.09898,0.133392) -- 
(2.1051,0.143283) -- 
(2.111,0.153387) -- 
(2.11668,0.163693) -- 
(2.12215,0.174189) -- 
(2.12742,0.184864) -- 
(2.13248,0.195707) -- 
(2.13735,0.20671) -- 
(2.14202,0.217862) -- 
(2.14651,0.229153) -- 
(2.15081,0.240575) -- 
(2.15494,0.252118) -- 
(2.1589,0.263775) -- 
(2.16269,0.275536) -- 
(2.16631,0.287395) -- 
(2.16978,0.299343) -- 
(2.17309,0.311374) -- 
(2.17625,0.323479) -- 
(2.17927,0.335653) -- 
(2.18214,0.347889) -- 
(2.18488,0.36018) -- 
(2.18748,0.372521) -- 
(2.18995,0.384905) -- 
(2.1923,0.397328) -- 
(2.19452,0.409783) -- 
(2.19663,0.422266) -- 
(2.19863,0.434772) -- 
(2.20051,0.447296) -- 
(2.20229,0.459833) -- 
(2.20396,0.472379) -- 
(2.20554,0.48493) -- 
(2.20701,0.497482) -- 
(2.2084,0.510031) -- 
(2.20969,0.522574) -- 
(2.21091,0.535107) -- 
(2.21203,0.547626) -- 
(2.21308,0.560128) -- 
(2.21405,0.57261) -- 
(2.21495,0.58507) -- 
(2.21578,0.597504) -- 
(2.21654,0.60991) -- 
(2.21724,0.622284) -- 
(2.21787,0.634626) -- 
(2.21845,0.646931) -- 
(2.21897,0.659198) -- 
(2.21944,0.671424) -- 
(2.21987,0.683608) -- 
(2.22024,0.695748) -- 
(2.22057,0.70784) -- 
(2.22087,0.719884) -- 
(2.22112,0.731878) -- 
(2.22134,0.743819) -- 
(2.22153,0.755707) -- 
(2.22169,0.767538) -- 
(2.22183,0.779313) -- 
(2.22194,0.791028) -- 
(2.22204,0.802683) -- 
(2.22212,0.814276) -- 
(2.22219,0.825805) -- 
(2.22224,0.837269) -- 
(2.2223,0.848666) -- 
(2.22235,0.859995) -- 
(2.2224,0.871254) -- 
(2.22246,0.882441) -- 
(2.22252,0.893556) -- 
(2.22261,0.904596) -- 
(2.22271,0.91556) -- 
(2.22283,0.926445) -- 
(2.22298,0.937251) -- 
(2.22316,0.947975) -- 
(2.22339,0.958616) -- 
(2.22365,0.96917) -- 
(2.22397,0.979637) -- 
(2.22434,0.990013) -- 
(2.22478,1.0003) -- 
(2.22529,1.01048) -- 
(2.22588,1.02057) -- 
(2.22656,1.03056) -- 
(2.22735,1.04044) -- 
(2.22824,1.05021) -- 
(2.22925,1.05987) -- 
(2.2304,1.0694) -- 
(2.2317,1.07882) -- 
(2.23317,1.0881) -- 
(2.23483,1.09725) -- 
(2.23669,1.10625) -- 
(2.23878,1.1151) -- 
(2.24112,1.12379) -- 
(2.24375,1.1323) -- 
(2.2467,1.14063) -- 
(2.25001,1.14875) -- 
(2.25372,1.15666) -- 
(2.25788,1.16433) -- 
(2.26256,1.17174) -- 
(2.26782,1.17886) -- 
(2.27376,1.18567) -- 
(2.28046,1.19211) -- 
(2.28806,1.19815) -- 
(2.29669,1.20373) -- 
(2.30655,1.20878) -- 
(2.31784,1.21322) -- 
(2.33086,1.21693) -- 
(2.34596,1.21978) -- 
(2.36361,1.22159) -- 
(2.3844,1.22213) -- 
(2.40917,1.22108) -- 
(2.43904,1.218) -- 
(2.4756,1.21229) -- 
(2.52119,1.20303) -- 
(2.5794,1.18886) -- 
(2.65593,1.16755) -- 
(2.76063,1.13533) -- 
(2.91185,1.0851) -- 
(3.1483,1.00193) -- 
(3.56821,0.847928) -- 
(4.51463,0.490715) -- 
(8.60136,-1.07823); 
\draw[line width=2.0pt, color = black] (-7.98127,5.32798) -- 
(-0.804646,2.56671) -- 
(0.378574,2.11792) -- 
(0.863408,1.93855) -- 
(1.12649,1.8447) -- 
(1.29111,1.78878) -- 
(1.40344,1.75299) -- 
(1.4847,1.72914) -- 
(1.54598,1.71295) -- 
(1.59366,1.70196) -- 
(1.63166,1.69466) -- 
(1.66252,1.69006) -- 
(1.68798,1.6875) -- 
(1.70923,1.68652) -- 
(1.72716,1.68677) -- 
(1.74241,1.68801) -- 
(1.75546,1.69005) -- 
(1.7667,1.69275) -- 
(1.77642,1.69598) -- 
(1.78485,1.69967) -- 
(1.79218,1.70374) -- 
(1.79856,1.70813) -- 
(1.80412,1.71279) -- 
(1.80896,1.71768) -- 
(1.81318,1.72278) -- 
(1.81683,1.72804) -- 
(1.81999,1.73345) -- 
(1.82271,1.73899) -- 
(1.82503,1.74465) -- 
(1.82699,1.75039) -- 
(1.82863,1.75622) -- 
(1.82997,1.76212) -- 
(1.83105,1.76809) -- 
(1.83188,1.7741) -- 
(1.83248,1.78016) -- 
(1.83288,1.78626) -- 
(1.83309,1.79239) -- 
(1.83312,1.79855) -- 
(1.83298,1.80473) -- 
(1.83269,1.81093) -- 
(1.83226,1.81715) -- 
(1.83169,1.82337) -- 
(1.831,1.82961) -- 
(1.83018,1.83585) -- 
(1.82926,1.84209) -- 
(1.82823,1.84833) -- 
(1.8271,1.85457) -- 
(1.82587,1.86081) -- 
(1.82455,1.86704) -- 
(1.82314,1.87327) -- 
(1.82165,1.87949) -- 
(1.82008,1.8857) -- 
(1.81843,1.89189) -- 
(1.81671,1.89808) -- 
(1.81492,1.90425) -- 
(1.81305,1.91041) -- 
(1.81111,1.91656) -- 
(1.80911,1.92269) -- 
(1.80705,1.9288) -- 
(1.80492,1.9349) -- 
(1.80272,1.94098) -- 
(1.80047,1.94704) -- 
(1.79815,1.95308) -- 
(1.79577,1.9591) -- 
(1.79333,1.9651) -- 
(1.79083,1.97108) -- 
(1.78827,1.97704) -- 
(1.78565,1.98297) -- 
(1.78297,1.98889) -- 
(1.78023,1.99478) -- 
(1.77743,2.00064) -- 
(1.77457,2.00649) -- 
(1.77164,2.01231) -- 
(1.76865,2.0181) -- 
(1.7656,2.02387) -- 
(1.76248,2.02961) -- 
(1.75929,2.03533) -- 
(1.75603,2.04101) -- 
(1.7527,2.04667) -- 
(1.7493,2.0523) -- 
(1.74582,2.0579) -- 
(1.74226,2.06348) -- 
(1.73863,2.06902) -- 
(1.7349,2.07452) -- 
(1.73109,2.08); 
\draw [fill=white, color = white] (0.2,1.7) circle (3.5pt);
\draw[color=black] (0.2,1.7) node {$p$};
\begin{scriptsize}
\draw [fill=black] (1.4,3.12) circle (2.5pt);
\draw [fill=black] (0.36,2.65) circle (2.5pt);
\draw [fill=black] (1.81,1.92) circle (2.5pt);
\draw [fill=black] (1.14,1.84) circle (2.5pt);
\draw [fill=white] (-2.3560293482952033,1.4225636599050508) circle (3.5pt);
\draw [fill=black] (-2.3560293482952033,1.4225636599050508) circle (1pt);
\draw [fill=white] (1.216901945449213,2.2185941929807407) circle (3.5pt);
\draw [fill=black] (1.216901945449213,2.2185941929807407) circle (1pt);
\draw [fill=white] (2.2208147044212625,0.7176154992548431) circle (3.5pt);
\draw [fill=black] (2.2208147044212625,0.7176154992548431) circle (1pt);
\draw [fill=white] (0.951720919100048,0.9130876017233132) circle (2.5pt);
\draw [fill=white] (5.069967759269214,0.2787748522299814) circle (2.5pt);
\draw [fill=black] (0.48,2.08) circle (3.5pt);
\end{scriptsize}
\end{tikzpicture}
\caption{The unexpected curve $C_4(Z)$ (in bold) of Example \ref{B3ex2}.}
\label{UnexpCurveDeg4} 
\end{figure}

\subsection{Curves and syzygies}
Theorem \ref{mainThm1} shows that the occurrence of unexpected curves through a certain fat point scheme $Z$
is related to the occurrence of syzygies of the Jacobian ideal of the form whose factors define the lines dual to the smooth points of $Z$. 
This result raises the question why there should be a connection between such curves and such syzygies.
Assume $\chr(\KK)=0$. Let $F$ be a squarefree product of linear homogeneous form in $\KK[\P^2]$.
Let $s=(s_0,s_1,s_2)$ be a minimal syzygy (i.e., of least degree possible) of $\nabla F=(F_x, F_y, F_z)$; i.e.,
$s\cdot \nabla F=0$, meaning
$$s_0F_x+s_1F_y+s_2F_z=0.$$
Since $s$ is minimal, the $s_i$ have no nonconstant common factor.
Thus $s$ defines a rational map $s:\P^2\dasharrow\P^2$ defined at all but a finite set of points.
Therefore, $s$ restricts as a morphism $s|_L:L\to \P^2$ to a general line $L$.

\begin{exercise}\label{s=Id}
If $s:\P^2\dasharrow\P^2$ is defined at $p\in \P^2$ and $s(p)=p$, then $F(p)=0$.
(Note: Look at $s\cdot \nabla F$ and use Euler's identity $xF_x+yF_y+zF_z=\deg(F)F$.)
In particular, if the locus of all points where $s(p)=p$ were to include a curve,
that curve consists of lines defined by one or more of the factors of $F$.
\end{exercise}

\SolnEater{\vskip\baselineskip 
\noindent{\it Details}: 
If $s(p)=p$, then $0=s(p)\cdot \nabla F(p) = p\cdot \nabla F(p)=\deg(F)F(p)$,
hence $F(p)=0$. 
\newline\qedsymbol\vskip\baselineskip}

\begin{exercise}\label{pxs}
Let $f=(f_0,f_1,f_2)=(x,y,z)\times s$, so $f_0=ys_2-zs_1$, $f_1=-(xs_2-zs_0)$ and $f_2=xs_1-ys_0$.
Let $\ell=Ax+By+Cz$ be a linear factor of $F$. 
For any point $p=(a,b,c)$ for which $s(p)\neq p$, one can show that $f(p)$ is the point dual to the line
through $p$ and $s(p)$. If in addition $\ell(p)=0$ but $p$ is not a singular point of $F=0$,
then $\ell(s(p))=0$ and one can conclude that $f(p)$ is the point dual to the line
defined by $\ell$. (Note: apply the product rule for $\nabla F$.)
\end{exercise}

\SolnEater{\vskip\baselineskip 
\noindent{\it Details}: 
That $f(p)$ is the point dual to the line
through $p$ and $s(p)$ is clear since the line through $(a,b,c)$ and $(a',b',c')$
is given by $(bc'-b'c)x-(ac'-a'c)y+(ab'-a'b)z=0$; i.e., the coefficient vector
is $(a,b,c)\times(a',b',c')$.

Now define $G$ by $F=\ell G$. We have $\nabla F=\nabla \ell G=G\nabla \ell+\ell\nabla G$.
Thus $0=s(p)\cdot (\nabla F)(p) = s(p)\cdot G(p)(\nabla \ell)(p) +\ell(p)(\nabla G)(p)
=s(p)\cdot G(p)(a,b,c)=G(p)(As_0(p)+Bs_1(p)+Cs_2(p))=G(p)\ell(s(p))$.
Since $p$ is not a crossing point of $F=0$, we see $G(p)\neq0$, so 
$\ell(s(p))=0$. Thus the line through $p$ and $s(p)$ is $\ell=0$, hence
$f(p)$ is the point dual to this line.
\newline\qedsymbol\vskip\baselineskip}

\begin{exercise}\label{pxs2}
If $s$ is not the identity on any line defined by $F=0$, then $f|_L:L\to \P^2$
defines a morphism whose image contains the points $Z$ dual to the lines defined by
the linear factors of $F$ and such that the points of $L\cap s(L)$ map to the point
dual to $L$. (Aside: In fact, $s(L)$ is a curve of degree $\deg(s_i)+1$
that contains $Z$ and has a point of multiplicity $\deg(s_i)$ at the point dual to $L$.)
\end{exercise}

\SolnEater{\vskip\baselineskip 
\noindent{\it Details}: If $s$ is not the identity on any line defined by $F=0$,
then it is the identity on only a finite set of points, but a general line $L$
will avoid those points. Thus $f$ is defined on $L$.
Since $L$ also avoids the singular points of $F=0$, 
we see by \Exercise\ \ref{pxs} that $f$ maps the points of $L$ where $F=0$
to the points dual to the lines defined by the linear factors of $F$,
and that if $q\in L\cap s(L)$, then there is a point $p\in L$ such that $s(p)=q$,
so $L$ is the line through $p$ and $s(p)$, hence $f(p)$ is the point 
dual to the line $L$. 
\newline\qedsymbol\vskip\baselineskip}

\begin{exercise}\label{pxs3}
If $s$ is not the identity on any line defined by $F=0$, then $(x,y,z)\times f = -\deg(F)Fs$,
hence we can recover $s$ from $f$.
\end{exercise}

\SolnEater{\vskip\baselineskip 
\noindent{\it Details}: This is just a calculation.
\newline\qedsymbol\vskip\baselineskip}

\end{document}